\documentclass[12pt, letterpaper]{amsart}

\usepackage{amsmath,amssymb,latexsym, color
}
\usepackage[dvips,centering,includehead,width=13.1cm,height=22.7cm,
paperheight=23.8cm,paperwidth=17cm
]{geometry}
\usepackage{enumitem}
\usepackage{hyperref}
\usepackage{pict2e}
\usepackage{epic,graphicx}
\usepackage{rotating}
\usepackage{mathtools}
\usepackage{tikz-cd}
\usepackage{ccaption}
\usepackage[commentmarkup=footnote]{changes}
\usepackage{subfiles}
\usepackage{nicefrac}

\changecaptionwidth
\captionwidth{5.6in}

\renewcommand{\ref}{\hyperref}

\def\K{{\mathbb K}}
\def\Z{{\mathbb Z}}
\def\R{{\mathbb R}}
\def\Q{{\mathbb Q}}
\def\NN{{\mathbb N}}
\def\C{{\mathbb C}}

\def\PP{{\mathbb P}}

\newcommand{\ba}{\boldsymbol{a}}

\newcommand{\bw}{\boldsymbol{w}}
\newcommand{\bp}{\boldsymbol{p}}
\newcommand{\bq}{\boldsymbol{q}}
\newcommand{\bu}{\boldsymbol{u}}

\newcommand{\bn}{\boldsymbol{n}}\newcommand{\bg}{\boldsymbol{g}}

\newcommand{\uz}{\underline{z}}
\newcommand{\ux}{\underline{x}}
\newcommand{\oa}{\overline{a}}
\newcommand{\ob}{\overline{b}}

\DeclareMathOperator{\Int}{Int}\DeclareMathOperator{\Aut}{Aut}
\DeclareMathOperator{\ini}{ini}
\DeclareMathOperator{\const}{const}\DeclareMathOperator{\Ini}{Ini}

\DeclareMathOperator{\val}{Val}
\DeclareMathOperator{\conv}{conv}
\DeclareMathOperator{\Tr}{Tr}
\DeclareMathOperator{\Tor}{Tor}\DeclareMathOperator{\gen}{ge}

\DeclareMathOperator{\Ima}{Image}
\DeclareMathOperator{\GS}{GS}

\DeclareMathOperator{\wid}{wd}

\DeclareMathOperator{\ord}{ord}
\DeclareMathOperator{\wt}{wt}
\DeclareMathOperator{\Con}{Con}\DeclareMathOperator{\OG}{OG}

\newtheorem{theorem}{Theorem}[section]
\newtheorem{proposition}[theorem]{Proposition}

\newtheorem{corollary}[theorem]{Corollary}
\newtheorem{lemma}[theorem]{Lemma}

\theoremstyle{definition}

\newtheorem{remark}[theorem]{Remark}

\newtheorem{example}[theorem]{Example}
\newtheorem{definition}[theorem]{Definition}

\newtheorem{Construction}{Construction}

\numberwithin{equation}{section}

\newcommand{\eps}{{\varepsilon}}

\definechangesauthor[name={Uriel}, color=purple]{uriel}

\usetikzlibrary{arrows.meta, calc, decorations.markings}

\usepackage{environ}
\makeatletter
\newsavebox{\measure@tikzpicture}
\NewEnviron{scaletikzpicturetowidth}[1]{%
	\def\tikz@width{#1}%
	\def\tikzscale{1}\begin{lrbox}{\measure@tikzpicture}%
		\BODY
	\end{lrbox}%
	\pgfmathparse{#1/\wd\measure@tikzpicture}%
	\edef\tikzscale{\pgfmathresult}%
	\BODY
}
\makeatother

\usepackage{todonotes}

\newcommand{\todos}{\makeatletter
	\providecommand\@dotsep{5}
	\makeatother
	\listoftodos\relax}

\newcommand{\bs}[1]{\boldsymbol{#1}}
\newcommand{\mbb}[1]{\mathbb{#1}}

\newcommand\moduli[1]{\overline{\mathcal{M}}^{#1}}
\newcommand\openmoduli[1]{\mathcal{M}^{#1}}
\newcommand\modulispace{\overline{\mathcal{M}}^{\text{trop}}_{g, (n_1, n_2)}(\Delta)}
\newcommand{\ogamma}{\overline{\Gamma}}

	\DeclareMathOperator{\mdim}{mdim}
	\DeclarePairedDelimiterX\setc[2]{\{}{\}}{\,#1 \;\delimsize\vert\; #2\,}

\begin{document}
\title{Refined tropical invariants and characteristic numbers}

\author{Eugenii Shustin}
\address{School of Mathematical Sciences, Tel Aviv
	University, Tel Aviv 69978,
	Israel}
\email{shustin@tauex.tau.ac.il}

\author{Uriel Sinichkin}
\address{School of Mathematical Sciences, Tel Aviv
	University, Tel Aviv 69978,
	Israel}
\email{sinichkin@mail.tau.ac.il}


\date{
	\today
}

\begin{abstract}
	We prove that the G\"ottsche-Schroeter and Schroeter-Shustin refined invariants specialize at $q=1$ to the enumeration of rational, resp. elliptic complex curves on arbitrary toric surfaces matching constraints that consist of points and of points with a contact element. Furthermore, we show that the refined invariant extends to the case of any genus $g\ge2$ and either one contact constraint or points in Mikhalkin position, and it again specializes to the corresponding characteristic number at $q=1$.
	In the appendix we show the limitations of extending this count to a more general setting.
\end{abstract}



\maketitle



\tableofcontents


\section{Introduction}

	The refined enumerative invariants were introduced by G\"ottsche and Shende \cite{GoSh} and then defined in the tropical geometry setting by Block and G\"ottsche \cite{BG}. Since then refined enumerative invariants, notably, refined tropical invariants became a subject of an intensive research that has brought a number of exciting properties and geometric phenomena allowing one to significantly advance in real and complex enumerative geometry, see, for instance, \cite{Bou,FiSt,GS,Mi1}.

To motivate the goals and results of our paper, we briefly recall some important developments in tropical enumerative geometry. Starting with the foundational Mikhalkin's work \cite{Mi}, the direct correspondence between complex and real algebraic curves on one side and tropical curves on the other side has become one of the corner stones of the tropical enumerative geometry. In this context, the tropical geometry has provided not only an efficient way to compute and study various complex and real enumerative invariants (see \cite{Mi} and some further examples in \cite{Bl24,Br15,FM,GS21,IKS04,Sh06}), but also led to new types of enumerative invariants - notably, refined tropical invariants. The first invariant of such a type was discovered by Block and G\"ottsche \cite{BG} as a one-parameter deformation of Mikhalkin's tropical formula for Gromov-Witten invariants of toric surfaces \cite{Mi}. It is well-defined for arbitrary tropical degrees and genera and admits an elementary proof of the invariance with respect to the choice of point constraints \cite{IM}.
Amazingly, this invariant interpolates between the count of complex curves and the count of real curves (with Welschinger signs). Furthermore, it naturally appears in other problems related to tropical vertex \cite{FiSt}, quantum index \cite{Mi1}, $\lambda$-classes on the moduli spaces of curves \cite{Bou}.

Later G\"ottsche and Schroeter suggested another refined invariant (further on briefly called {\it GS-invariant}) \cite{GS} that counts marked plane rational tropical curves of a given degree with marked points located partly on the edges of tropical curves (as all markings in the Block-G\"ottsche case) and partly at the vertices. It interpolates between the rational descendant Gromov-Witten invariants of type $\langle\tau(0)^a\tau(2)^b\rangle_0$ and the mixed Welschinger invariants counting real rational curves on toric surfaces passing through an appropriate number of real points and pairs of complex conjugate points. However, this property came not from direct correspondence statements but rather from the coincidence of algebraic-geometric and tropical enumerative numbers based on some external arguments \cite{MR} and \cite{GMS}. In \cite{SS}, it was shown that the GS-invariants can uniquely be extended to the case of genus one tropical curves. The weights of tropical curves in the resulting invariant appeared to be significantly more complicated than in the rational GS-invariant. Furthermore, the algebraic-geometric meaning and possible extensions to higher genera remained completely open questions.

In this paper, we provide answers to the two latter questions: namely, we show to what extent the GS-type refined tropical invariants exist for arbitrary genera, and we prove a correspondence statement relating tropical curves counted by the new refined invariants and some complex algebraic curves on toric surfaces (the correspondence between the spoken tropical curves and real algebraic curves is studied in our paper \cite{SS26}).
More precisely,
\begin{itemize}\item In addition to \cite{SS}, where the GS-invariants were defined for genus $1$, we prove that
the GS-type invariants can be defined for any positive genus, any toric surface, and the whole range of tropical constraints, provided that there is exactly one marked point at a vertex of any counted tropical curve (Theorem \ref{thm:refined_invariant_one_marked} in Section \ref{sec:refined_definition}). Also, we show (see Proposition \ref{prop:no_local_refined_high_genus} in Appendix \ref{sec:refined_non_local}) that beyond the two above cases, there do not exist GS-type invariants defined for the arbitrary generic choice of constraints and satisfying the local invariance condition.
\item For a wide range of toric surfaces, any genus, any distribution of marked points among edges and vertices of counted tropical curves, and for tropical constraints in a special position (Mikhalkin position, see assumptions (A1) and (A2) in Section \ref{scn4}), there exists a refined tropical invariant naturally extending the GS-invariants for genus $0$ and $1$ (Theorem \ref{tcn4} in Section \ref{sec:refined_definition}).
    \item According to \cite{GS}, the GS-invariants evaluated at the parameter equal to $1$ coincide with the rational descendant Gromov-Witten invariants $\langle\tau(0)^a\tau(2)^b\rangle_0$ for toric surfaces. It was proved in \cite{GKP} that the latter descendant invariants coincide with the rational {\it characteristic numbers}, i.e., the numbers of
        rational complex curves of a given degree and genus subject to the conditions: (i) to pass through a certain number of generic fixed points, (ii) to be tangent to given lines at some of the fixed points; moreover, for positive genera, similar descendant invariants and characteristic numbers satisfying conditions (i) and (ii) are tightly related. Here, we prove that
        the GS-type invariants from the preceding item, evaluated at the parameter equal to $1$ coincide with the (complex) characteristic numbers of the considered toric surfaces
        (Theorem \ref{tcn1} in Section \ref{t4.3} for genus zero, and Theorem \ref{tcn3} in Section \ref{scn4} for positive genera). We prove this relation between the tropical invariants and the characteristic numbers by establishing a direct correspondence between the counted algebraic curves and their tropicalization (tropical curves), see Lemmas \ref{lemma:number_of_modifications} and \ref{lemma:correspondence_final_step} in Section \ref{t4.3} for genus zero, and the proof of Theorem \ref{tcn3} in Section \ref{scn4} for positive genera.
\end{itemize}
Among additional results, we mention 
the splitting of the complicated and a bit mysterious refined weight associated with tropical curves in \cite[Formula (12)]{SS} into three summands, each of them having a clear geometric meaning (see Remark \ref{rem:summing_over_weights_get_Psi} in Section \ref{sec:refined_definition}) and precisely matching the aforementioned correspondence between tropical and algebraic curves.

The main body of the paper consists of Sections \ref{sec-pr}-\ref{rbi}: in Section \ref{sec-pr}, we present basic material on tropical curves and tropicalization of algebraic curves over the field of Puiseux series and introduce characteristic numbers, in Sections \ref{t4.3} and \ref{scn4}, we prove the correspondence statements and establish tropical formulas for characteristic numbers, Section \ref{rbi} and Appendix \ref{sec:refined_non_local} are devoted to the study of refined tropical invariants.

We include a list of frequently used notation throughout the paper in the end of Section \ref{sec-nota}.

\subsection*{Acknowledgements}
We are grateful to Ilya Tyomkin for valuable discussion, to Gurvan M\'evel for pointing out a mistake in the previous version of this paper, and to Erwan Brugall\'e  for important remarks and corrections.
We also thank the anonymous referee for careful reading and helpful suggestions.

\section{Preliminaries}\label{sec-pr}

    In this section we set the notations used in the paper, define the problem of computing characteristic numbers, recall the necessary results from tropical geometry, and provide related auxiliary statements.

\subsection{Notations}\label{sec-nota}
{\bf(1)} We use the complex field $\C$ and the field $\K$ of locally convergent complex Puiseux series.
For an element $a=\sum_{r\ge r_0}a_rt^r\in\K^*$, $a_{r_0}\in\C^*$, denote
\begin{equation*}\val(a)=-r_0,\quad \ini(a)=a_{r_0},\quad \Ini(a)=a_{r_0}t^{r_0}\ .\end{equation*}
Let $F=\sum_{\omega\in P\cap\Z^n}a_\omega\uz^\omega\in\K[\uz]$, $\uz=(z_1,...,z_n)$, have Newton polytope $P$. It
yields a {\it tropical polynomial}
$$N(\ux)=N_F(\ux)=\max_{\omega\in P\cap\Z^n}(\langle\omega,\ux\rangle+\val(a_\omega)),\quad N:\R^n\to\R\ ,$$
and its Legendre dual, {\it valuation function} $\nu=\nu_N:P\to\R$, whose graph defines a subdivision $\Sigma_\nu$
of $P$ into linearity domains which all are convex lattice polytopes. One can write
$$F(\uz)=\sum_{\omega\in \Delta\cap\Z^n}(a^0_\omega+O(t^{>0}))t^{\nu(\omega)}\uz^\omega\ ,$$ where
$a^0_\omega\in\C$, $a^0_\omega\ne0$ for all $\omega$ vertices of the subdivision $\Sigma_\nu$. Given a face $\delta$ of
the subdivision $\Sigma_\nu$, we write
$$F^\delta(\uz)=\sum_{\omega\in\delta\cap\Z^n}a_\omega\uz^\omega,\quad\ini(F^\delta)(\uz)=
\sum_{\omega\in\delta\cap\Z^n}a^0_\omega\uz^\omega\in\C[\uz]\ ,$$
$$\Ini(F^\delta)(\uz)=\sum_{\omega\in\delta\cap\Z^n}a^0_\omega t^{\nu(\omega)}\uz^\omega
\in\K[\uz]\ .$$

\smallskip

{\bf(2)} All lattice polyhedra we consider lie in Euclidean spaces $\R^2$ with fixed integral lattices
$\Z^2\subset\R^2$. We consider these spaces as $\Z^2\otimes\R$ and denote them by $\R^2_\Z$.

For a convex lattice polygon $P\subset\R^2_\Z$, we denote
by $\Tor(P)$ the complex toric surface associated with $P$
and by $\Tor^*(P)\subset\Tor(P)$ the big torus (the dense orbit of the torus action). Next, denote by ${\mathcal L}_P$ the tautological line bundle over $\Tor(P)$, by
$|{\mathcal L}_P|$ the linear system generated by the non-zero global sections (equivalently, by the monomials
$x^iy^j$, $(i,j)\in P\cap\Z^2$). We also use the notation ${\mathcal L}_P(-Z)$ for
${\mathcal L}_P\otimes{\mathcal J}_{Z/\Tor(P)}$, where ${\mathcal J}_{Z/\Tor(P)}$ is the ideal sheaf of a zero-dimensional subscheme $Z\subset\Tor(P)$.

Let $\bn:\widehat C\to X$ be a non-constant morphism of a complete smooth irreducible curve $\widehat C$ to a toric surface $X$. We call it {\it peripherally unibranch} if, for any
toric divisor $D\subset X$, the divisor $\bn^*(D)\subset\widehat C$ is concentrated at one point.
Respectively we call the curve $C=\bn_*\widehat C$ {\it peripherally smooth},
if it is smooth at its intersection points with each toric divisor.

\smallskip
{\bf(3)}  For a vector $v\in\Z^n$, resp. a lattice segment $\sigma\subset\R^n_\Z$, we denote
its lattice length by $\|v\|_\Z$,
resp. $\|\sigma\|_\Z$. More generally, for an $m$-dimensional lattice polytope
$P\subset\R^n_\Z$, $n\ge m$, we denote by $\|P\|_\Z$ its $m$-dimensional lattice volume (i.e., the ratio
of the Euclidean $m$-dimensional volume of $P$ and the minimal
Euclidean volume of a lattice simplex inside the affine $m$-dimensional subspace of $\R^n_\Z$ spanned by $P$).

\smallskip{\bf(4)} We always take the standard basis in $\R^2_\Z$ and identify $\Lambda^2(\R^2_\Z)\simeq\R_\Z$ by
letting $$\oa\wedge\ob=\det\left(\begin{matrix}a_1&a_2\\ b_1&b_2\end{matrix}\right),
\quad\oa=(a_1,a_2),\ \ob=(b_1,b_2)\in\R^2_\Z\ .$$

\smallskip
{\bf(5)} For the reader's convenience we collect here the main recurring notation used throughout the paper.
We use boldface letters ($\bw$, $\bu$, $\bp$, $\bq$, $\bg$, \ldots) for ordered collections or sequences.
\begin{center}
\begin{tabular}{ll}
$\K$ & field of locally convergent complex Puiseux series\\
$\val$, $\ini$ & valuation and initial coefficient of an element of~$\K$\\
$\Tor_{\K}(P)$ & toric surface associated with a convex lattice polygon~$P$\\
$\Delta$ & multiset of primitive outer normals to~$P$\\
$n_1,\,n_2$ & numbers of simple (mult.~$1$) and tangent (mult.~$2$) constraints\\
$w^{(1)}_i,\; w^{(2)}_j$ & point constraint and contact-element constraint over~$\K$\\
$u^{(1)}_i,\; u^{(2)}_j$ & their tropicalizations in~$\R^2$; $\bu=\bu_{(1)}\cup\bu_{(2)}$\\
$\Gamma^0,\;\Gamma^1,\;\Gamma^1_\infty$ & vertices, edges, and ends of a tropical curve~$\Gamma$\\
$q_i\in\bq$ & marked point on the tropical curve\\
$p_i\in\bp$ & marked point on the algebraic curve\\
$\oa_V(E)$ & directing vector of edge~$E$ at vertex~$V$\\
$\wt(E)$ & weight of an edge\\
$\mu(V)$ & Mikhalkin multiplicity of a vertex\\
$CN_g(P,(n_1,n_2))$ & characteristic number as defined in \eqref{eq:char_num_def}\\
$\GS_y$ & the refined (GS-type) invariant
\end{tabular}
\end{center}

    \subsection{Characteristic numbers}\label{seccn1}
	
	Let $P\subset\R^2_\Z$ be a nondegenerate convex lattice polygon, and let $\Delta$ be the multiset of primitive outer integral normals $\ba_\sigma$ to the sides $\sigma$ of $P$ so that each $\ba_\sigma$ is repeated $\|\sigma\|_\Z$ times; in particular, $|\Delta|=|\partial P\cap\Z^2|$. Fix a non negative integer ${0\le g \le |\Int(P)\cap \Z^2|}$ and a partition
	$$|\Delta|+g-1=n_1+2n_2\quad\text{with}\quad n_1,n_2\ge0,$$
	and introduce two sequences: a sequence of points
	$$w^{(1)}_i\in(\K^*)^2,\quad i=1,...,n_1,$$
	and a sequence of contact elements
	$$(w^{(2)}_j,\lambda_j):\quad w^{(2)}_j\in(\K^*)^2,\quad \lambda_j\in T^*_{w^{(2)}_j}(\K^*)^2\setminus\{0\},\quad j=1,...,n_2,$$
	such that the configuration $\bw=\{w^{(1)}_i\}_{i=1}^{n_1}\cup\{(w^{(2)}_j,\lambda_j)\}_{j=1}^{n_2}$ is in general position.
	Denote by ${\mathcal M}_{g,(n_1,n_2)}(P,\bw)$ the set of isomorphism classes of maps of marked curves $\bn:(\widehat C,\bp)\to\Tor_\K(P)$ such that
	\begin{itemize}\item $\widehat C$ is a smooth connected curve of genus $g$, and $\bn_*\widehat C\in|{\mathcal L}_P|$;
		\item $\bp=\bp_{(1)}\cup\bp_{(2)}\subset \widehat C$ and sequence of $n_1+n_2$ distinct points, ${\bp_{(1)}=\{p^{(1)}_i\}_{i=1}^{n_1}}$, $\bp_{(2)}=\{p^{(2)}_j\}_{j=1}^{n_2}$, such that $\bn(p^{(1)}_i)=w^{(1)}_i$, $i=1,...,n_1$, and $\bn(p^{(2)}_j)=w^{(2)}_j$, $\langle\lambda_j,D\bn(T_{p^{(2)}_j}\widehat C)\rangle=0$, $j=1,...,n_2$.
	\end{itemize}
	
	\begin{lemma}\label{lcn1}
		The set ${\mathcal M}_{g,(n_1,n_2)}(P,\bw)$ is finite. Furthermore, for each element $[\bn:(\widehat C,\bp)\to\Tor_\K(P)]\in{\mathcal M}_{g,(n_1,n_2)}(P,\bw)$, the map $\bn$ is birational onto its image, and the curve $C=\bn_*\widehat C$ is smooth at all points $w^{(1)}_i$, $i=1,...,n_1$, $w^{(2)}_j$, $j=1,...,n_2$.
	\end{lemma}
	
	\begin{proof} Note that $n_1+2n_2=|\Delta|+g-1=c_1({\mathcal L}_P)\cdot c_1(\Tor_\K(P))+g-1$, which is the upper bound for the dimension of components of the Severi variety ${\mathcal V}_{P,g}\subset|{\mathcal L}_P|$ parameterizing irreducible genus $g$ curves in the linear system $|{\mathcal L}_P|$, see, for instance, \cite[Theorem 3.1]{Va}. It yields the finiteness of ${\mathcal M}_{g,(n_1,n_2)}(P,\bw)$ and the non-vanishing of the differential $D\bn$ at $\bp$.
		Similarly, for the dimension reason, $\bn$ is birational onto its image for all elements of ${\mathcal M}_{g,(n_1,n_2)}(P,\bw)$.
	\end{proof}
	
	We state the problem to compute the {\it characteristic numbers}
	\begin{equation}\label{eq:char_num_def}
		CN_g(P,(n_1,n_2))\overset{\text{def}}{=}\#{\mathcal M}_{g,(n_1,n_2)}(P,\bw),
	\end{equation}
	and we intend to apply the tropical enumerative geometry for this purpose.
	
	As a preparation, we specify the choice of the configuration $\bw$ as follows. Let
	\begin{align*}w^{(1)}_i&=(t^{a_i}(\eta^{(1)}_i+O(t^M)),t^{b_i}(\zeta^{(1)}_i+O(t^M)))\in(\K^*)^2,\\ &\qquad a_i,b_i\in\Q,\ \eta^{(1)}_i,\zeta^{(1)}_i\in\C^*,\quad i=1,...,n_1,\end{align*}
	and
	\begin{align}
		w^{(2)}_j&=(t^{c_j}(\eta^{(2)}_j+O(t^M)),t^{d_j}(\zeta^{(2)}_j+O(t^M)))\in(\K^*)^2,\label{ecn15}\\
		&\qquad c_j,d_j\in\Q,\ \eta^{(2)}_j,\zeta^{(2)}_j\in\C^*,\quad j=1,...,n_2,\nonumber\\
		\lambda_j&=(\lambda_{j,x}+O(t^M))\frac{dx}{x}+(\lambda_{j,y}+O(t^M))\frac{dy}{y}\in T^*_{w^{(2)}_j}(\K^*)^2,\label{ecn5}\\
		&\qquad\lambda_{j,x}^{(0)},\lambda_{j,y}^{(0)}\in\C^*,\quad j=1,...,n_2,\nonumber
	\end{align}
	where $M\gg0$, all constants $\eta^{(1)}_i,\zeta^{(1)}_i,\eta^{(2)}_j,\zeta^{(2)}_j,\lambda_{j,x},\lambda_{j,y}$ are generically chosen in $\C^*$, and the sequence of points
	$$\bu=\bu_{(1)}\cup\bu_{(2)}\subset\R^2,\quad \bu_{(1)}=\{u^{(1)}_i\}_{i=1}^{n_1},\ \bu_{(2)}=\{u^{(2)}_j\}_{j=1}^{n_2},$$ $$u^{(1)}_i=(-a_i,-b_i),\ i=1,...,n_1,\quad u^{(2)}_j=(-c_j,-d_j),\ j=1,...,n_2,$$
	is in tropical general position (cf. \cite[Section 4.2]{Mi}).
	
	
	\begin{remark}\label{rcn2} In what follows, we shall use finitely many toric automorphisms of $\Tor_\K(P)$ as well as non-toric transformations related to modifications.
		Observe that a toric transformations preserves the shape (\ref{ecn5}) of the contact element and simply multiplies the vector $(\lambda_{j,x}+O(t^M),{\lambda_{j,y}+O(t^M)})$ by a unimodular integral $2\times2$ matrix. In case of a modification, we start with a point $w^{(2)}_j=(t^{c_j}(\eta^{(2)}_j+O(t^M)),\zeta^{(2)}_j+O(t^M))$ and a contact element (\ref{ecn5}) attached to it, assuming that the point $u^{(2)}_j$ lies on an edge of the embedded plane tropical curve $h_*\Gamma$ that corresponds to a segment $[(0,l),(0,0)]$ of the dual subdivision of $P$. We furthermore assume that the function $\nu:P\to\R$ vanishes along the above segment being positive outside it. In such a case, the polynomial $F(x,y)$ defining the curve $\bn_*\widehat C\subset\Tor_\K(P)$ takes the form
		$$F(x,y)=(y-\zeta^{(2)}_j)^l+O(t^{>0}).$$
		The modification amounts to two coordinate changes. The first change is $y=\zeta^{(2)}_j+y_1$, and it results in the polynomial
		$$F_1(x,y_1):=F(x,\zeta^{(2)}_j+y_1)=y_1^l+O(t^{>0}).$$
		
		The second change is $y_1=t^\alpha y_2$, $\alpha>0$, which results in a polynomial
		$$F_2(x,y_2)=t^{-l\alpha}F_1(x,t^\alpha y_2)$$
		having all coefficients nonpositively valued and at least two coefficients with valuation zero (i.e., $y^l$ and some more monomial $x^{k_1}y^{k_2}$, $0\le k_2<l$).
		After this transformation the contact element at $w^{(2)}_j$ turns to be
		$$(\lambda_{j,x}+O(t^M))\frac{dx}{x}+t^\alpha(\lambda_{j,y}+O(t^M))\frac{dy_2}{\zeta^{(2)}_j+t^\alpha y_2},$$
		and it specializes at $t=0$ to
		\begin{equation}\lambda_{j,x}\frac{dx}{x}.\label{ecn6}\end{equation}
	\end{remark}

\subsection{Plane tropical curves}\label{sec-tropcur}

A {\it plane tropical curve} is a tuple $(\Gamma,h, \bg)$, where
\begin{itemize}\item
	$\Gamma$ is a finite
	connected metric graph, whose set $\Gamma^0$ of vertices is nonempty and does not contain
	univalent vertices,
	the set of edges $\Gamma^1$ contains a
	subset $\Gamma^1_\infty\ne\emptyset$ consisting of edges isometric to $[0,\infty)$ (called ends), while
	$\Gamma^1\setminus\Gamma^1_\infty$ consists of edges isometric to compact segments in $\R$ (called finite edges);
	
	\item $h:\Gamma\to\R^2$ is a
	continuous map such that
	$h$ is
	affine-integral on each edge of $\Gamma$ in the length coordinate, and it is nonconstant on at least one edge of
	$\Gamma$;
	furthermore, at each vertex $V$ of $\Gamma$, the balancing condition holds
	$$\sum_{E\in\Gamma^1,\ V\in E}
	\oa_V(E)=0\ ,$$ where $\oa_V(E)$ is the image under the differential $D(h\big|_E)$
	of the unit tangent vector to $E$ emanating from
	its endpoint $V$. We call $\oa_V(E)$ the {\it directing vector} of $E$ (centered at $V$).

	\item $\bg:\Gamma^0\to\Z_{\ge 0}$ is a nonnegative function assigning to each vertex $V\in \Gamma^0$ its genus, such that the genus of each bivalent vertex is positive.
\end{itemize}

For a tropical curve $\Gamma$ we denote by $\overline{\Gamma}$ its compactification obtained by adding a vertex to each end of $\Gamma$.

Define the {\it genus} of $(\Gamma, h,\bg)$ as
\begin{equation*}g=\gen(\Gamma,h,\bg):=b_1(\Gamma)+\sum_{V\in\Gamma^0}\bg(V)\ .\end{equation*}

We will sometimes, when it will not lead to a confusion, denote the tuple $ (\Gamma, h, \bg) $ by $ \Gamma $.
A \emph{collinear cycle} in $ \Gamma $ is a sub-graph consisting of 2 vertices that are connected by 2 edges, see Figure \ref{fig:central_collinear_cycle}.
\begin{remark}\label{remark:figure_convention}
In Figure \ref{fig:central_collinear_cycle} and throughout this paper, in all figures a black circle ($\bullet$) denotes a marking that is restricted to lie on a vertex, while a black cross ($\times$) denotes a marking which is not restricted.
\end{remark}

The \emph{valency} of a vertex $ V\in \Gamma^0 $ is the number of edges incident to it.
A vertex $ V\in \Gamma $ is called {\it flat} or {\it collinear} if the linear span of $ {\setc{\oa_V(E)}{E\in \Gamma^1, V\in E}} $ is one dimensional.
Given a non collinear vertex $V\in\Gamma^0$, the directing vectors $\oa_V(E)$ over all edges $E$ incident to $V$, being positively rotated by $\frac{\pi}{2}$ form a nondegenerate convex lattice polygon
$P(V)$, which we call {\it dual} to $V$.
The multi-set $\Delta=\Delta(\Gamma,h)
\subset\R^2_\Z$ of vectors $\{\oa_V(E)\ne0\ : \ E\in\Gamma^1_\infty,
\ V\in\Gamma^0,\ V\in E\}$ is called the {\it degree} of $(\Gamma,h)$. It is easy to see that $\Delta$
is nonempty and {\it balanced} (i.e., the vectors of $\Delta$ sum up
to zero. The degree $\Delta$ is called {\it primitive} if it consists of primitive vectors (i.e., vectors of lattice length $1$).
Positively rotated by $\frac{\pi}{2}$, the vectors of $\Delta$
can be combined into a convex lattice polygon $P=P(\Gamma,h)$, called
the {\it Newton polygon} of $(\Gamma,h)$. In this case, we say that the degree $\Delta$ {\it induces} the polygon $P$. The degree $\Delta$ is called {\it nondegenerate} if $\dim P(\Gamma,h)=2$.

To each edge $E\in\Gamma^1$ we assign the {\it weight} $$\wt(E)=\big\|\oa_V(E)\big\|_\Z\ .$$

In the preceding notation, the image $h(\Gamma)\subset\R^2$ is a
closed rational finite one-dimensional polyhedral complex without univalent and bivalent vertices.
It can be converted into a plane tropical curve as follows. Each edge
$e\subset h(\Gamma)$ is assigned a weight $\wt(e)$ which is the sum of the weights of the (non-contracted) edges of
$\Gamma$ intersecting $h^{-1}(x)$, where $x\in e$ is a generic point. We define a parametrization $\widehat h:\Gamma_h\to h(\Gamma)$, which is a homeomorphism and is affine-integral on each edge $E$ of $\Gamma_h$ with $\oa_V(E)=\wt(e)\cdot\oa$, where $e=h(E)$, and $\oa$ is a primitive integral vector parallel to $e$. The balancing condition evidently holds, and we call the resulting tropical curve $h_*(\Gamma)=(\Gamma_h,\widehat h)$
an {\it embedded plane tropical curve}.

We recall that the graph $\Gamma_h$ is a corner locus of a tropical polynomial
$N:\R^2\to\R$ with the Newton polygon $P$. The Legendre dual function $\nu_N:P\to\R$
defines a subdivision $\Sigma$ of $P$ into convex lattice polygons, and this subdivision is completely
determined by $h_*(\Gamma)$. There is a duality reversing the incidence relation: The polygons of $\Sigma$ are in bijection with the vertices of $\Gamma_h$ so that the number of sides of a polygon in $\Sigma$ equals to the valency of the dual vertex of $\Gamma^0_h$, while the edges of $\Sigma$ are in bijection with the edges of $\Gamma_h$
so that the lattice length of an edge of $\Sigma$ equals the weight of the dual edge of $\Gamma_h$.

A {\it marked plane tropical curve} is a tuple $(\Gamma,h,\bg,\bq)$, where $ (\Gamma,h,\bg) $ is a plane tropical curve and $\bq$ is an ordered subset of
$n\ge1$ distinct points of $\Gamma$.

Two marked tropical curves $ (\Gamma, h, \bg, \bq) $ and  $ (\Gamma', h', \bg', \bq') $ are said to be \emph{isomorphic} if there exists an isomorphism of metric graphs $ \varphi:\Gamma\to\Gamma' $ with $ h'\circ\varphi=h $ and $ \varphi(\bq) = \bq' $.

\begin{definition}
	Let $(\Gamma,h)$ be a plane tropical curve, and let $V\in\Gamma^0$ be a trivalent vertex, which is not incident to
	contracted edges.
	Define the Mikhalkin multiplicity of $V$ by
	$$\mu(V)=\big|\oa_V(E_1)\wedge\oa_V(E_2)\big|\ ,$$
	where $E_1,E_2\in\Gamma^1$ are some two incident to $V$ edges.
\end{definition}

\begin{definition}
	Let $ (\Gamma, h, \bg, \bq) $ be a marked plane tropical curve.
	For every marked point $ q_i\in \bq $ denote by $ c_i \in \Gamma^1\cup \Gamma^0 $ the unique cell of $ \Gamma $ containing $ q_i $ in its relative interior, and set $\bs{c}=(c_1,\ldots,c_n)$.
	Additionally denote by $ \underline{\Gamma} $ the abstract graph obtained from $ \Gamma $ by forgetting the metric structure.
	We define the \emph{combinatorial type} (or simply \emph{type}) of $ (\Gamma, h, \bg, \bq) $ to be the tuple $ (\underline{\Gamma}, (\wt(E))_{E\in \Gamma^1}, \bs{c}, (\oa_V(E))_{V\in E\in \Gamma^1}) $.
	We will denote the combinatorial type of $ \Gamma $ by $ [\Gamma] $.
\end{definition}

\begin{definition}\label{def:center_collinear_cycle}
	A \emph{centrally embedded collinear cycle} in $ \Gamma $ is a collinear cycle whose endpoints are trivalent and the $2$ adjacent edges to it are of the same length, see Figure \ref{fig:central_collinear_cycle}.
\end{definition}

\begin{remark}
The condition of the lengths in Definition \ref{def:center_collinear_cycle} to be equal can be seen as an example of Speyer's well-spacedness condition which is known in certain situations to be a sufficient condition for a tropical curve to be the tropicalization of an algebraic curve, see \cite{R, Sp}.
\end{remark}

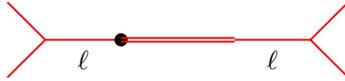
\begin{figure}
	\begin{center}
	\begin{tikzpicture}[scale=0.5]
		\tikzset{
			every path/.style = {line width=0.7pt, red},
			every node/.style = {black}
		}
		\draw (-1, -1) -- (0, 0) -- (-1, 1);
		\draw (0, 0) -- node [below] {\footnotesize $\ell$} (2, 0);
		\draw [black, fill=black] (2, 0) circle (0.15);
		\draw (2, -0.05) -- (5, -0.05);
		\draw (2, 0.05) -- (5, 0.05);
		\draw (5, 0) -- node [below] {\footnotesize $\ell$} (7,0);
		\draw (8, -1) -- (7,0) -- (8, 1);
	\end{tikzpicture}
\end{center}
\caption{An example of a centrally embedded collinear cycle, see Definition \ref{def:center_collinear_cycle}. The labels $\ell$ on the edges denote their lengths, to indicate those lengths are equal.}\label{fig:central_collinear_cycle}
\end{figure}

\begin{definition}\label{def:realizable_trop_curve}
	A tropical curve is called \emph{realizable} if the following 2 conditions hold:
	\begin{enumerate}
		\item Every collinear vertex is trivalent and it is either marked or is an endpoint of a centrally embedded collinear cycle whose second endpoint is marked.
		
		\item No 2 collinear marked vertices are joined by an edge.
	\end{enumerate}
\end{definition}

\begin{remark}
	From this point onward, we restrict our attention to realizable tropical curves.
\end{remark}

\subsection{Moduli space of tropical curves}\label{sec:moduli_space}
We fix a nondegenerate balanced multiset $ \Delta \subseteq \mbb{Z}^2\setminus \{0\} $ that induces the polygon $P$ and three non negative integers $ g, n_2, n_1 $ with $ g\le \left|\Int(P)\cap \Z^2 \right| $ and set $ n:=n_2+n_1 $.

Given a combinatorial type $ \alpha = (\Gamma, \bs{w}, \bs{c}, (\oa_V(E))) $ of degree $ \Delta $, genus $ g $ and $ n $ marked points such that the first $ n_2 $ of them lie on vertices, we can parameterize the realizable tropical curves of type $ \alpha $ by a polyhedron $ \openmoduli{\alpha} \subseteq \mbb{R}^{|\Gamma^1|-|\Delta|+|\bs{c}\cap \Gamma^1|+2} $.
Indeed, every such tropical curve is defined by the length of its bounded edges, the distance between vertices and neighboring marked points and a position of an arbitrarily chosen vertex (and those lengths and distances should satisfy linear equations and inequalities in order to represent a tropical curve).

The points in the closure $ {\moduli{\alpha}} \subseteq \mbb{R}^{|\Gamma^1|-|\Delta|+|\bs{c}\cap \Gamma^1|+2} $ correspond to tropical curve of combinatorial type that is attained from $ \alpha $ by possibly contracting some of the edges and moving marked points from an edge to one of its boundary vertices.
Thus, for such combinatorial types $ \beta $ we get that $ {\moduli{\beta}} $ is naturally a face of $ {\moduli{\alpha}} $.
We will say that $\beta$ is a \emph{degeneration} of $\alpha$, and $\alpha$ is a \emph{regeneration} of $\beta$.
Gluing $ {\moduli{\alpha}} $ using this identification we get the \emph{moduli space of genus $ g $ realizable tropical curves with $ n_2 $ marked vertices and $ n_1 $ marked edges}:
\[ \modulispace := \lim_{\stackrel{\longrightarrow}{\alpha}} \moduli{\alpha}. \]

It is sometimes more convenient to work with curves without collinear vertices.
This is the purpose of the following definition and construction.

\begin{definition}
	A tropical curve is called \emph{cycle reduced} if it has no collinear cycles with at least one collinear endpoint.
	It is called \emph{vertex reduced} if it has no collinear vertices.
\end{definition}

\begin{Construction}
	Let $ (\Gamma, h, \bg, \bq) $ be a realizable marked plane tropical curve.
	We construct a cycle reduced tropical curve $ (\Gamma', h', \bg', \bq') $ called the \emph{cycle reduction} of $ \Gamma $, by replacing every collinear cycle $ \{E_1, E_2\} $ that is attached to a collinear vertex, by a single edge $ E $ with weight $ w'(E)=w(E_1)+w(E_2) $ and which is mapped to the same image under $ h' $ as $ E_1 $ and $ E_2 $ are mapped under $ h $.
	Note that this might introduce rational bivalent vertices, so we need to \emph{stabilize}, i.e. replace pairs of adjacent edges to a rational bivalent vertex (those edges has to have the same weight by the balancing condition) by a unique edge of the same weight.
	The combinatorial type of $ \Gamma' $ is uniquely determined by the type of $ \Gamma $, we will call it the \emph{cycle reduction of $ [\Gamma] $}.
	
	This construction might produce a curve with collinear vertices, so we present a stronger construction.
	Suppose that $V$ is a collinear vertex of $\Gamma'$, and $E_1, E_2$ are two edges adjacent to it with $h(E_1)\subseteq h(E_2)$ (see Figure \ref{fcn4}).
	Denote the other endpoint of $E_1$ by $V_1$ and the other endpoint of $E_2$ by $V_2$.
	We remove $E_1$ and $E_2$ from $\Gamma'$ and add another two edges:
	\begin{itemize}
		\item An edge $E'$ of weight $\wt(E_1)+\wt(E_2)$ connecting $V$ with $V_1$ and with $h(E')=h(E_1)$,
		
		\item and an edge $E''$ of weight $\wt(E_2)$ connecting $V_1$ with $V_2$ and having $h(E'')=(h(E_2)\setminus h(E_1))\cup\{V_1\}$.
	\end{itemize}
	As before, if $V$ is rational and becomes bivalent, we stabilize it.
	Applying this sequentially for all collinear vertices will produce a tropical curve $(\Gamma_{red}, h_{red}, \bg_{red}, \bq_{red})$ with no collinear vertices, called the \emph{vertex reduction of $\Gamma$}.
\end{Construction}

\begin{figure}
	\setlength{\unitlength}{1mm}
	\begin{picture}(95,15)(0,0)
		\thinlines
		
		\thicklines
		{\color{blue}
		}
		{\color{red}
			\put(0,5){\line(1,1){5}}\put(0,15){\line(1,-1){5}}\put(5,10){\line(1,0){5}}
			\put(10,10.5){\line(1,0){10}}\put(10,9.5){\line(1,0){9}}\put(21,9.5){\line(1,0){9}}
			\put(20,10.5){\line(0,-1){10.5}}\put(20,10.5){\line(1,1){4.5}}\put(30,9.5){\line(0,-1){9.5}}\put(30,9.5){\line(1,1){5.5}}
			\put(60,5){\line(1,1){5}}\put(60,15){\line(1,-1){5}}\put(65,10){\line(1,0){25}}
			\put(80,10){\line(1,1){5}}
			\put(80,10){\line(0,-1){10}}\put(90,10){\line(0,-1){10}}\put(90,10){\line(1,1){5}}
		}
		\put(10,10){\circle*{1}}\put(70,10){\circle*{1}}
		\put(14,11.5){\scriptsize $E_1$}\put(23,6){\scriptsize $E_2$}\put(9,12){\scriptsize $V$}
		\put(21, 10){\scriptsize $V_1$}\put(31, 8){\scriptsize $V_2$}

		\put(74,11.5){\scriptsize $E'$}\put(83,6){\scriptsize $E''$}\put(69,12){\scriptsize $V$}
		\put(81, 10){\scriptsize $V_1$}\put(91, 8){\scriptsize $V_2$}
		\put(44,8){$\Longrightarrow$}
		
	\end{picture}
	\caption{Vertex reduction procedure}\label{fcn4}
\end{figure}
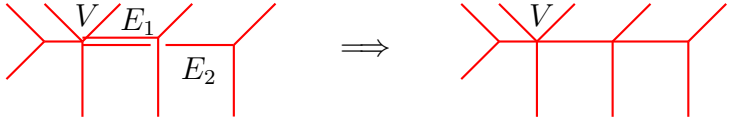

\begin{remark}\label{rem:vertex_reduction_iso}
	It is immediate that both cycle reduction and vertex reduction induce a linear bijection of polyhedra $ \openmoduli{[\Gamma]} \to \openmoduli{[\Gamma']} $ and $\openmoduli{[\Gamma]} \to \openmoduli{[\Gamma_{red}]}$.
\end{remark}

\begin{definition}\label{def:regular_trop_curve}
	A tropical curve $ (\Gamma, h, \bg, \bq) $ is called \emph{regular} if every connected component of $ \Gamma\setminus \bq $ is simply connected and contains exactly one end (cf.\ \cite[Section~2.6]{Mi}).
	
	If $(\Gamma, h, \bg, \bq)$ is regular and $ K $ is a connected component of $ \Gamma\setminus\bq $, the edges of $ K $ can be oriented in a unique way for which the unbounded edge oriented towards infinity and every unmarked vertex has exactly one outgoing edge.
	We call such orientation \emph{the regular orientation on $K$}.

\end{definition}

\begin{remark}
	It is evident that any collinear cycle in a regular marked tropical curve joins a marked vertex with an unmarked one.
\end{remark}

We can now compute the dimension of $ \modulispace $ and characterize its top dimensional and codimension 1 cells.

\begin{definition}
	The \emph{maximal dimension} of a combinatorial type $ \alpha $ is
	\[ \mdim \alpha := |\Delta|+b_1(\Gamma)-1+|\bs{c}\setminus \Gamma^0|, \]
	where $|\bs{c}\setminus \Gamma^0|$ counts the marked points that lie in the interior of edges (rather than at vertices); each such marking contributes one extra degree of freedom corresponding to its position along the edge.
\end{definition}

\begin{lemma}\label{lemma:dim_moduli_genus1}
	Let $ (\Gamma,  h, \bg, \bq)$ be an elliptic plane tropical curve (i.e. $\gen(\Gamma) = 1$), and let $\alpha:=[\Gamma]$ be its combinatorial type.
	Let $V_1,\dots,V_n$ be the vertices of the unique cycle of $\Gamma$ and let $E_1,\dots,E_n$ be its edges (so that $E_i\cap E_{i+1} = \{V_i\}$).
	Then
	\[ \dim \openmoduli{\alpha} = \mdim \alpha - \text{ov}(\Gamma) + \bigl(2 - \dim\,\mathrm{span}(\{\oa_{V_1}(E_1),\dots, \oa_{V_n}(E_n)\})\bigr), \]
	where
	$$ \text{ov}(\Gamma) := \sum_{V\in \Gamma^{0}} (|\setc{E\in \Gamma^1}{V \in E}| - 3). $$
\end{lemma}

\begin{proof}
	It was mentioned in the beginning of the current subsection that $\openmoduli{\alpha}$ is a polyhedron in $\R^{|\Gamma^1|-|\Delta|+2+|\bs{c}\cap \Gamma^1|}$ (here we use the fact that $|\bq\setminus\Gamma^0| = |\bs{c}\cap \Gamma^1|$), where the coordinates of this Euclidean space are the lengths of the bounded edges $\setc{\ell_{E}}{E\in \Gamma^1\setminus \Gamma^1_\infty} $, the positions of the markings, and the image of a distinguished vertex $h(V_1)$.
	The affine span of this polyhedron is the space of solutions of the system
	\[ \sum_{i=1}^n \oa_{V_i}(E_i) \ell_{E_i} = 0.  \]
	Let $r$ be the rank of the matrix whose columns are $\{ \oa_{V_1}(E_1), \dots, \oa_{V_n}(E_n) \}$.
	We get that
	\[ \dim \openmoduli{\alpha} = |\Gamma^1| - |\Delta| + 2 + |\bq\setminus \Gamma^0| - r, \]
	and so to finish the proof we need to show that $|\Gamma^1| = 2|\Delta| + \text{ov}(\Gamma)$.
	By the definition of $\text{ov}$ we get that $2|\Gamma^1| = |\Delta|+3|\Gamma^0|+\text{ov}(\Gamma)$ and by the Euler characteristic computation we get that $|\Delta|+|\Gamma^0|-|\Gamma^1| = \chi(\Gamma) = 0$.
	The result follows by expressing $|\Gamma^0|$ and $|\Gamma^1|$ in terms of $|\Delta|$ and $\text{ov}(\Gamma)$ from those linear equations.
\end{proof}

\begin{lemma}\label{lemma:dim_moduli_cells}
	Let $ (\Gamma, h, \bg, \bq)$ be a cycle reduced plane tropical curve, and suppose that either $\gen(\Gamma) \le 1$, or $|\bq\cap \Gamma^{0}|\le 1$.
	Then for the combinatorial type $ \alpha := \left[\Gamma\right] $, we have
	\begin{equation}\label{eq:dim_moduli_polyhedron_ineq}
		\dim \openmoduli{\alpha} \le \mdim \alpha
	\end{equation}
	and additionally
	\begin{itemize}
		\item $ \dim \openmoduli{\alpha} =  \mdim \alpha $ if, and only if, $ \Gamma $ is trivalent;
		
		\item If $ \dim \openmoduli{\alpha} =  \mdim \alpha - 1 $ then one of the following holds:
		\begin{enumerate}
			\item Either all the vertices of $\Gamma$ are trivalent, except two $4$-valent vertices that are joined by a collinear cycle;

			\item or $\Gamma$ has one $4$-valent marked vertex and all the other vertices of $\Gamma$ are trivalent;

			\item or all the marked vertices of $\Gamma$ are trivalent and all the unmarked vertices of $\Gamma$ are of valency at most $4$;

			\item or $\Gamma$ has $3$ pairwise adjacent vertices $V_1, V_2, V_3$, with $V_1$ and $V_2$ $4$-valent, $V_3$ trivalent, flat, and marked, and all other vertices of $\Gamma$ trivalent, see Figure \ref{fig:flat_triangle};

			\item or $\Gamma$ has $4$ vertices $V_1, V_2, V_3, V_4$ forming a $4$-cycle (with edges $V_1 V_2$, $V_2 V_3$, $V_3 V_4$, $V_4 V_1$), with $V_1$ and $V_3$ $4$-valent and $V_2$, $V_4$ marked and collinear, see Figure \ref{fig:flat_quadrangle}.
		\end{enumerate}
	\end{itemize}
\end{lemma}

\begin{remark}
	Being cycle reduced does not mean that $\Gamma$ has no collinear cycles at all: it only means that $\Gamma$ has no collinear cycle with at least one \emph{collinear} endpoint.
	In particular, $\Gamma$ may still contain a collinear cycle joining two \emph{non}-collinear vertices, and replacing such a cycle by a single edge of the appropriate weight does not require a subsequent vertex reduction.
\end{remark}

\begin{proof}
	If $\Gamma$ contains a collinear cycle we can replace it by a single edge.
	This operation reduces $\mdim \alpha$ by one while leaving $\dim \openmoduli{\alpha}$ unchanged, so it is sufficient to prove the lemma for $\Gamma$ without a collinear cycle.

	If $\gen(\Gamma) = 0$ the statement were proven in \cite[Proposition 3.9]{Gathmann2007} and if $\gen(\Gamma) = 1$ the statement follows from Lemma \ref{lemma:dim_moduli_genus1}; let us spell out how the two cases occurring in the formula of Lemma \ref{lemma:dim_moduli_genus1} match the case list above, in the sub-case where the unique cycle $E_1,\ldots,E_n$ of $\Gamma$ (with vertices $V_1,\ldots,V_n$) is contained in a line, i.e.\ $\dim\,\mathrm{span}(\{\oa_{V_1}(E_1),\dots,\oa_{V_n}(E_n)\}) = 1$.
	By the balancing condition every trivalent vertex among $\{V_1,\dots,V_n\}$ is flat.
	Since $\Gamma$ has no collinear cycles, by Definition \ref{def:realizable_trop_curve}, every flat vertex $V_i$ is marked trivalent and both $V_{i-1}$ and $V_{i+1}$ are not flat (and thus not trivalent).
	By Lemma \ref{lemma:dim_moduli_genus1}, if $\dim \openmoduli{\alpha}\le \mdim \alpha - 1$, then $\text{ov}(\Gamma)\le 2$ and so at most $2$ vertices of $V_1,\dots,V_n$ have valency larger than $3$.
	Since we are assuming $n\ge 3$, there exists a collinear vertex $V_i$ of the cycle, and $V_{i-1}\ne V_{i+1}$ so they are the only non trivalent vertices of $\Gamma$.
	If $n=3$ and $V_{i-1}$ is connected to $V_{i+1}$ by an edge which is part of the cycle, we get option $4$ in the case list above, otherwise $V_{i+2} = V_{i-2}$ is their common flat neighbor, $n=4$, and we get option $5$ in the case list above.

	Thus it remains to prove the assertion for $\Gamma$ with a unique marked vertex $V$, which will be the assumption for the rest of the proof.
	
	First suppose that $V$ is not collinear.
	By Definition \ref{def:realizable_trop_curve} this means that $\Gamma$ has no flat vertices.
	Thus the conditions of \cite[Proposition 3.9]{Gathmann2007} hold, and the assertion follows from it.
	
	Finally, suppose that $V$ is flat, and thus, by Definition \ref{def:realizable_trop_curve}, trivalent.
	Let $E_1$ and $E_2$ be the adjacent edges to $V$ with $h(E_1)\subsetneq h(E_2)$ and let $V_1$ be the endpoint of $E_1$ different from $V.$
	Consider the vertex reduction $\Gamma_{red}$ of $\Gamma$, and denote by $\alpha_{red}$ its combinatorial type.
	Then ${\mdim \alpha_{red} = \mdim \alpha + 1}$ (since it has one more marking in the interior of an edge) and $\dim \openmoduli{\alpha_{red}} = \dim \openmoduli{\alpha}$ by Remark \ref{rem:vertex_reduction_iso}.
	Moreover, the valency of the image of $V_1$ in $\Gamma_{red}$ is bigger by one than the valency of $V_1$.
	Thus, by applying \cite[Proposition 3.9]{Gathmann2007} to the vertex reduced curve $\Gamma_{red}$, we readily see that $\dim \openmoduli{\alpha} \le \mdim \alpha$ with equality exactly for trivalent curves.
	
	Now suppose that $\Gamma$ has a vertex $V_2$ of valency higher than $4$.
	If $ V_2 = V_1 $ then its image in $\Gamma_{red}$ is of valency higher than $5$ which implies, by the proof of \cite[Proposition 3.9]{Gathmann2007}, that $\dim \openmoduli{\alpha_{red}} < \mdim \alpha_{red} - 2$ and thus $\dim \openmoduli{\alpha} < \mdim \alpha - 1$.
	If $V_2 \ne V_1$, then in $\Gamma_{red}$ we have two vertices of combined valency at least $9$ that are not joined by a collinear cycle.
	Picking the line $L$ in the proof of \cite[Proposition 3.9]{Gathmann2007} to be the line thorough those vertices gives us again $\dim \openmoduli{\alpha_{red}} < \mdim \alpha_{red}-2$ and thus finishes the proof.
\end{proof}

\begin{figure}
	\begin{center}
	\begin{tikzpicture}[scale=0.9]
		\tikzset{
			every path/.style = {line width=0.7pt, red},
			every node/.style = {black}
		}
		\coordinate (V1) at (-1.5, 0);
		\coordinate (V2) at (1.5, 0);
		\coordinate (V3) at (0.5, -0.1);

		\draw (V1) -- (V2);
		\draw (-1.5, -0.1) -- (0.5, -0.1);
		\draw (1.5, -0.05) -- (0.5, -0.05);

		\draw (V1) -- (-3, 1.2);
		\draw (V1) -- (-3, -1.2);
		\draw (V2) -- (3, 1.2);
		\draw (V2) -- (3, -1.2);

		\draw (V3) -- (3, -0.1);
		\draw [black, fill=black] (0.5, -0.1) circle (0.06);

		\node [above] at (V1) {\footnotesize $V_1$};
		\node [above] at (V2) {\footnotesize $V_2$};
		\node [below right] at (V3) {\footnotesize $V_3$};
	\end{tikzpicture}
	\end{center}
	\caption{A codimension $1$ configuration with $3$ pairwise adjacent vertices $V_1, V_2, V_3$: $V_1$ and $V_2$ are $4$-valent, and $V_3$ is trivalent, flat and marked.}\label{fig:flat_triangle}
\end{figure}

\begin{figure}
	\begin{center}
	\begin{tikzpicture}[scale=0.9]
		\tikzset{
			every path/.style = {line width=0.7pt, red},
			every node/.style = {black}
		}
		\coordinate (V1) at (-1.5, 0);
		\coordinate (V2) at (1.5, 0);
		\coordinate (V3) at (0.5, -0.1);
		\coordinate (V4) at (-0.7, 0.1);

		\draw (-1.5, -0.1) -- (0.5, -0.1);
		\draw (1.5, -0.0) -- (0.5, -0.0);
		\draw (-1.5, 0.0) -- (-0.7, 0.0);
		\draw (1.5, 0.1) -- (-0.7, 0.1);

		\draw (V1) -- (-3, 1.2);
		\draw (V1) -- (-3, -1.2);
		\draw (V2) -- (3, 1.2);
		\draw (V2) -- (3, -1.2);

		\draw (V3) -- (3, -0.1);
		\draw (V4) -- (-3, 0.1);
		\draw [black, fill=black] (0.5, -0.1) circle (0.06);
		\draw [black, fill=black] (-0.7, 0.1) circle (0.06);

		\node [above] at (V1) {\footnotesize $V_1$};
		\node [above] at (V2) {\footnotesize $V_3$};
		\node [below right] at (V3) {\footnotesize $V_2$};
		\node [above right] at (V4) {\footnotesize $V_4$};
	\end{tikzpicture}
	\end{center}
	\caption{A codimension $1$ configuration with $4$ vertices $V_1, V_2, V_3, V_4$ forming a $4$-cycle: $V_1$ and $V_3$ are $4$-valent, and $V_2$, $V_4$ are marked and collinear.}\label{fig:flat_quadrangle}
\end{figure}

\begin{corollary}\label{cor:dim_moduli_space}
	If either $g\le 1$ or $n_2\le 1$ then the dimension of $ \modulispace $ is $ |\Delta|+g-1+n_1 $.
\end{corollary}

\begin{example}
	Consider the curve in Figure \ref{fig:trop_curve_3_4valent_1_codim}, let $\alpha$ denote its combinatorial type.
	It is of genus $3$, has $3$ unbounded ends, and has $3$ marking in the interior of edges.
	Thus $\mdim \alpha = |\Delta|+g-1+n_1 = 3+3-1+3=8$.
	Since $\alpha$ has two $4$-valent vertices, we expect $\dim \openmoduli{\alpha}$ to be $\mdim \alpha-2=6$.
	But $\dim \openmoduli{\alpha}=7$ since we can move every marking in a one dimensional family, translate the whole curve in a two dimensional family, and scale it in a one dimensional family.
\end{example}

\begin{figure}
	\begin{center}
	\begin{tikzpicture}
		\tikzset{
			every path/.style = {line width=0.7pt, red},
			every node/.style = {black}
		}
		\begin{scope}[scale=3, decoration={markings,mark=at position 0.5 with {\arrow{>}}}]
			\draw [] (0.545455, 0.536455) -- (0.272727, 0.536455);
			\draw [] (0.545455, 0.536455) -- (0.727273, 0.727273);
			\draw [] (0.545455, 0.536455) -- (0.545455, 0.454545);
			\draw [] (0.545455, 0.363636) -- (0.545455, 0.272727);
			\draw [] (0.545455, 0.363636) -- (0.427919, 0.481172);
			\draw [] (0.545455, 0.363636) -- (0.727273, 0.727273);
			\draw [] (0.545455, 0.363636) -- (0.545455, 0.454545);
			\draw [] (0.727273, 0.727273) -- (0.363636, 0.554455);
			\draw [postaction={decorate}] (0.727273, 0.727273) -- node [above, very near end] {\footnotesize $v_3$} (1.000000, 1.000000);
			\draw [] (0.363636, 0.554455) -- (0.272727, 0.554455);
			\draw [] (0.363636, 0.554455) -- (0.427919, 0.481172);
			\draw [black, fill=black] (0.272727, 0.545455) circle (0.015000);
			\draw [postaction={decorate}] (0.272727, 0.545455) -- node [above, very near end] {\footnotesize $v_1$} (0.000000, 0.545455);
			\draw [postaction={decorate}] (0.545455, 0.272727) -- node [right, very near end] {\footnotesize $v_2$} (0.545455, 0.000000);
			\draw [black] (0.530455, 0.257727) -- (0.560455, 0.287727);
			\draw [black] (0.530455, 0.287727) -- (0.560455, 0.257727);
			\draw [black] (0.427919, 0.452172) -- (0.427919, 0.510172);
			\draw [black] (0.398919, 0.481172) -- (0.450919, 0.481172);
			\draw [black] (0.530455, 0.439545) -- (0.560455, 0.469545);
			\draw [black] (0.530455, 0.469545) -- (0.560455, 0.439545);
			
		\end{scope}
	\end{tikzpicture}
	\end{center}

	\caption{A tropical curve with two $4$-valent vertices and $\dim \openmoduli{\alpha}=\mdim \alpha-1$.}\label{fig:trop_curve_3_4valent_1_codim}
\end{figure}

In case $\Gamma$ is either rational or does not have marked vertices, a stronger version of Lemma \ref{lemma:dim_moduli_cells} was proven in \cite[Proposition 3.9]{Gathmann2007}.
\begin{lemma}\label{lemma:dim_moduli_genus_0}
	Let $(\Gamma, h, \bg, \bq)$ be a marked plane tropical curve with either $\gen(\Gamma)=0$ or $\bq\cap \Gamma^{0}=\emptyset$.
	Then $\dim \openmoduli{[\Gamma]} = \mdim [\Gamma] - 1$ if, and only if, $\Gamma$ is trivalent except either a unique $4$-valent vertex or two $4$-valent vertices joined by a collinear cycle.
\end{lemma}

\subsection{The evaluation map in codimension zero and one}\label{sec:eval_map}
We want to understand the positions of the marked points in the generic case and in degenerations of codimension 1.
For that, we introduce the evaluation map
\begin{align*}
	\text{Ev} : \modulispace & \to \mbb{R}^{2n} \\
	(\Gamma, h, \bg, \bq) & \mapsto (h(q_1),\dots, h(q_n)).
\end{align*}
From now on, we set $ n_1+2n_2 = |\Delta|+g-1 $.
By Corollary \ref{cor:dim_moduli_space}, this is precisely the condition for which $ \text{Ev} $ is a map of polyhedral complexes of the same dimension.

We say that a combinatorial type $\alpha$ is \emph{enumeratively essential} if exist a combinatorial type $\beta$ with $\openmoduli{\alpha}\subseteq \moduli{\beta}$ and the image of $\openmoduli{\beta}$ under $\text{Ev}$ is of dimension $2n$.

The proof of the following lemma is immediate, and omitted for conciseness.
\begin{lemma}\label{lemma:Ev_injective}
	For a tropical curve $ (\Gamma,h, \bg, \bq) $ of combinatorial type $ \alpha $, the restriction $ \text{Ev}|_{\openmoduli{\alpha}} $ is injective if, and only if, $ \ogamma\setminus\bq $ does not contain neither components of positive genus nor components with more than one unbounded edge.
\end{lemma}

Since we set $ n_1+2n_2 = |\Delta|+g-1 $, a calculation of Euler characteristic gives us that if $\Gamma\setminus\bq$ contains a bounded component it has to contain either a component with more than one unbounded end or a component of positive genus.
Thus Lemma \ref{lemma:Ev_injective} implies the following result.
\begin{lemma}\label{lemma:generic_trop_curve_is_regular}
	The image of the non-regular curves under the map
	$$ {\text{Ev}:\modulispace\to\mbb{R}^{2n}} $$
	is of positive codimension in $ \mbb{R}^{2n} $.
\end{lemma}

\begin{definition}\label{def:simple_vertex}
	A vertex $ V\in \Gamma^0 $ is called \emph{simple} if its image in the cycle reduction of $ \Gamma $ is either a trivalent vertex or in the interior of an edge.
\end{definition}

\begin{lemma}\label{lemma:codim1_degenerations_list}
	Let $ n_1,n_2 $ be non negative integers, $ n:=n_1+n_2 $, suppose that $ n_1+2n_2 = |\Delta|+g-1 $ and that either $g\le 1$ or $n_2\le 1$.
	Let $ (\Gamma, h, \bg, \bq) $ be a tropical curve whose combinatorial type $ \alpha $ is enumeratively essential, and suppose $ \dim \moduli{\alpha} = |\Delta|+g-2+n_1 $ (i.e. of codimension 1 in $ \modulispace $).
	Then one of the following holds:
	\begin{enumerate}
		\item Either $ |\bq\cap\Gamma^0|=n_2 $, $ (\Gamma, h, \bg, \bq) $ is regular, and all its vertices are rational and at most $4$-valent;
		
		\item or $ |\bq\cap\Gamma^0|=n_2 $, $ (\Gamma, h, \bg, \bq) $ is regular, all of its vertices are rational and $3$-valent except a unique vertex which is rational, $5$-valent, and contained in a collinear cycle;
		
		\item or $ |\bq\cap\Gamma^0|=n_2 $, all the vertices of $ \Gamma $ are rational and simple except a unique rational marked vertex $ V $, and
		\begin{enumerate}
			\item either $ V $ is four valent, is not attached to collinear cycles and $ r $ of the edges adjacent to $ V $ are contained in the unique bounded component of $ \Gamma\setminus\bq $ for some $ 1\le r\le 3 $,
			
			\item or $ V $ is five valent, attached to three edges in different directions and to a bounded collinear cycle in $ \Gamma\setminus\bq $ while all the other components of $ \Gamma\setminus\bq $ are unbounded;
			
		\end{enumerate}
		
		\item or $ |\bq\cap\Gamma^0|=n_2 $, all the vertices of $ \Gamma $ are rational and trivalent except two marked rational four valent vertices that are joined by a collinear cycle, and all the components of $ \Gamma\setminus\bq $ except this collinear cycle are unbounded;
		
		\item or $ |\bq\cap\Gamma^0|=n_2 $, all the vertices of $ \Gamma $ are rational and trivalent except a marked vertex $V_1$ and an unmarked vertex $V_2$ which are both rational, $4$-valent, and joined by a collinear cycle, and $ \Gamma\setminus\bq $ have a unique bounded component which contain either $1$ or $2$ edges incident to $V_1$;
		
		\item or $ |\bq\cap\Gamma^0|=n_2 $, $ (\Gamma, h, \bg, \bq) $ is regular, and all its vertices are rational and trivalent except $2 $ rational unmarked vertices $V_1$ and $V_2$ which are joined by $ 2 $ marked edges, $V_1$ adjacent to an additional edge and additional collinear cycle, and $V_2$ is adjacent to $2$ additional edges;
		
		\item or $ |\bq\cap\Gamma^0|=n_2+1 $, all the vertices of $ \Gamma $ are rational and simple, and $ \Gamma\setminus\bq $ has either a unique bounded component or a bounded collinear cycle attached to the unique marked vertex other than $ q_1,\dots,q_{n_2} $;
		
		\item or $ |\bq\cap\Gamma^0|=n_2 $, all the vertices of $ \Gamma $ are rational and simple except one elliptic marked simple vertex $V$, and $ \Gamma\setminus\bq $ has a unique bounded component which is attached to $V$.
	\end{enumerate}
\end{lemma}

\begin{proof}
	First note that since $ \alpha $ is enumeratively essential, the map $ \text{Ev}|_{\openmoduli{\alpha}} $ is injective, and thus, by Lemma \ref{lemma:Ev_injective}, all the components of $ \Gamma\setminus\bq $ are of genus $ 0 $ and contain at most one unbounded edge.
	Moreover, since $ \alpha $ is a degeneration of a regular type, any bounded component of $ \Gamma\setminus\bq $ has to be attached to a marked vertex.
	
	Denote by $ k $ the number of collinear cycles in $ \Gamma $ with at least one flat endpoint, by $ (\Gamma', h', \bg', \bq') $ the cycle reduction of $ (\Gamma, h, \bg, \bq) $ and by $ \alpha' $ its type.
	Note that the conditions on $g$ and $n_2$ imply that $k\le 1$.
	Next, denote by $ b $ the number of bounded components of $ \Gamma'\setminus\bq' $.
	Denote by $ \ogamma'_1,\dots,\ogamma'_{|\Delta|+b} $ compactifications of the components of $ \Gamma'\setminus\bq' $ obtained by adding a vertex at each non closed edge.
	For every marked vertex $ V\in \bq'\cap\Gamma'^0 $ denote by $ r_V $ its valency.
	Since every such connected component is rational we can compute the euler characteristic of $ \ogamma' $ in two ways:
	\begin{align*}
		1-b_1(\ogamma)+k = 1- b_1(\ogamma') = \chi(\ogamma') & \\
 =  \sum_{i=1}^{|\Delta|+b}\chi(\ogamma'_i) - |\bq'\setminus\Gamma'^0|
 - & \sum_{V\in \bq\cap \Gamma'^0}(r_V-1) +2k \\
		= |\Delta|+b - (|\bq\setminus\Gamma^0| + k) & - \sum_{V\in \bq\cap \Gamma^0}(r_V-1) + 2k\\
		= n_1+2n_2-g+1+b+k & -|\bq\setminus\Gamma^0| - \sum_{V\in \bq\cap \Gamma^0}(r_V-1) \\
		= (n_1-|\bq\setminus\Gamma^0|) + 2(n_2-|\bq\cap\Gamma^0|) & - g + 1 + b + k - \sum_{V\in \bq\cap \Gamma^0}(r_V-3).
	\end{align*}
	After rearranging and using the fact that $ |\bq\cap\Gamma^0|+|\bq\setminus\Gamma^0| = n = n_2+n_1 $, we get
	\begin{equation}\label{eq:bounded_components_count}
		b = (g-b_1(\ogamma))+(|\bq\cap\Gamma^0|-n_2) + \sum_{V\in \bq'\cap\Gamma'^0}(r_V-3).
	\end{equation}
	
	By Lemma \ref{lemma:dim_moduli_cells}, $ \mdim \alpha $ is either $ |\Delta|+g-1+n_1 $ or $ |\Delta|+g-2+n_1 $.
	We start by dealing with the former case, i.e. when the first betti number of $ \Gamma $ is $ g $ and it has exactly $ n_2 $ marked vertices.
	In that case $ {\dim \openmoduli{\alpha'} = \dim \openmoduli{\alpha} = \mdim \alpha - 1 = \mdim \alpha' - 1} $, we will consider separately the $3$ cases in the second possibility of Lemma \ref{lemma:dim_moduli_cells}.
	
	First suppose that $\Gamma'$ has no collinear cycle, all its marked vertices are trivalent, and all the vertices of $\Gamma'$ have valency at most $4$.
	Then by \eqref{eq:bounded_components_count}, $\Gamma'$ is regular, meaning that $\Gamma$ is regular as well.
	Recall that $\Gamma$ has at most one collinear cycle.
	If it has no collinear cycles we get the first option in the statement of the Lemma.
	If it has one collinear cycle then $\Gamma'$ is either rational or has no marked vertices, so by Lemma \ref{lemma:dim_moduli_genus_0} it has a unique $4$-valent vertex.
	If the unique collinear cycle of $\Gamma$ is attached to the $4$-valent vertex of $\Gamma'$ we get option $(2)$ of the statement, otherwise we get option $(1)$ again.
	
	Next suppose that $\Gamma'$ has a unique marked $4$-valent vertex $V$ and all its other vertices are trivalent.
	By \eqref{eq:bounded_components_count} we see that $ \Gamma' \setminus \bq' $ has a unique bounded component and since $ \alpha $ is a degeneration of a regular type, this bounded component has to be attached to $ V $.
	Depending on whether the preimage in $\Gamma$  of the unique bounded component of $\Gamma'\setminus \bq'$ is a collinear cycle attached to $V$ or a single bounded component containing $1$, $2$, or $3$ edges incident to $V$, we get either option $(3a)$ or $(3b)$.
	
	The last option is when $ \Gamma' $ has two four valent vertices $ V_1 $ and $ V_2 $ joined by a collinear cycle.
	By \eqref{eq:bounded_components_count} the number of bounded components of $ \Gamma'\setminus\bq' $ is equal to the number of markings on $ \{V_1,V_2\} $.
	If they are both marked then the collinear cycle between them contains at least two bounded components and thus, by \eqref{eq:bounded_components_count} it contain exactly two bounded components (i.e. does not contain any markings) and those are the only bounded components of $ \Gamma\setminus\bq $.
	If only one of $ \{V_1, V_2\} $ are marked we clearly get option $ (5) $.
	If both $ V_1 $ and $ V_2 $ are unmarked we get either option $(1)$ if neither $V_1$ nor $V_2$ are adjacent to a collinear cycle, and option $ (6) $ otherwise.
	
	We now deal with the case $ \mdim \alpha = |\Delta|+g-2+n_1 $. This can happen if either $ |\bq'\cap\Gamma'^0|=n_2-k+1 $ or if $ \gen(\Gamma') = g-k-1 $.
	In both those cases \eqref{eq:bounded_components_count} implies that $ \Gamma'\setminus\bq' $ has exactly one bounded component, so we get options $ (7) $ and $ (8) $ in the classification respectively.
\end{proof}

\subsection{Tropicalization of algebraic curves over a non-Archimedean field}\label{sec-trop}
In this section we will describe \emph{tropicalization}, i.e. the process of associating a tropical curve to an algebraic curve over $\K$.
In fact, we will see that the tropicalization of an algebraic curve produces a richer structure than just tropical curve, which we will call \emph{enhanced tropical curve}.

\subsubsection{Enhanced tropical curves}\label{sec:enhanced_trop_curves}
Let $(\Gamma, h, \bg)$ be a plane tropical curve, and let $V\in \Gamma^0$ be a vertex.
A \emph{limit curve} associated to $V$ is a triple $(C_V, (p_{V,E})_{V\in E}, \varphi_V)$, where $C_V$ is a smooth projective curve over $\C$ of genus $\bg(V)$, $(p_{V,E})_E\subset C_V$ is a collection of distinct points on $C_V$ indexed by the adjacent edges to $V$, and $\varphi_V$ is a morphism $C_V\setminus \{p_{V,E}\} \to (\C^*)^2 $ with $ \ord_{p_{V,E}}(\varphi_{V, i}) = (\oa_{V,E})_i$.
For a non-flat vertex $V$, a limit curve associated to a vertex $V$ can be extended to a morphism $\varphi:C_V\to \Tor_{\C}(P_V)$, where $P_V$ is the polygon dual to $V$.
In particular, if $E\in \Gamma^1$ is an edge with endpoints $V_1$ and $V_2$, then both $\varphi_{V_1}(p_{V_1,E})$ and  $\varphi_{V_2}(p_{V_2,E})$ are contained in the mutual toric boundary $\Tor_{\C}(P_{V_1})\cap \Tor_{\C}(P_{V_2})$.
We say that the limit curves are \emph{compatible along $E$} if $\varphi_{V_1}(p_{V_1,E}) = \varphi_{V_2}(p_{V_2,E})$.
An \emph{enhanced plane tropical curve} is a plane tropical curve with compatible (along all edges) limit curves associated to its vertices.

\subsubsection{Embedded tropical limit}\label{sec-tl}

Let $\Delta\subset\Z^2\setminus\{0\}$ be a nondegenerate, primitive, balanced multiset which induces a convex lattice polygon $P$, and let $C\in|{\mathcal L}_P|_\K$ be a reduced,
irreducible curve of genus $\gen(C)=g$, which passes through a sequence $\bu$ of $n$ distinct points in $(\K^*)^2\subset\Tor_\K(P)$ and does not hit intersection points of toric divisors. In particular, it can be given by an
equation
$$F(x,y)=\sum_{(i,j)\in P\cap\Z^2}t^{\nu'(i,j)}(a^0_{ij}+O(t^{>0}))x^iy^j\ ,$$
where $a^0_{ij}\in\C$, $(i,j)\in P\cap\Z^2$, and $a^0_{ij}\ne0$ if the coefficient
of $x^iy^j$ in $F$ does not vanish (for instance, when $(i,j)$ is a vertex of $P$).
We then define a convex, piecewise-linear function $\nu:P\to\R$, whose graph is the lower part of
the $\conv\{(i,j,\nu'(i,j)),\ (i,j)\in P\cap \Z^2\}\subset\R^3_\Z$.
Via a parameter change $t\mapsto t^r$, we can make
$\nu(P\cap\Z^2)\subset\Z$. Denote by $\Sigma_\nu$ the subdivision
of $P$ into linearity domains of $\nu$, which are convex lattice polygons $P_1,...,P_m$. We then have
\begin{equation}F(x,y)=\sum_{(i,j)\in P\cap\Z^2}t^{\nu(i,j)}(c^0_{ij}+O(t^{>0}))x^iy^j\ ,\label{e-new60}\end{equation}
where $c^0_{ij}\ne0$ for $(i,j)$ a vertex of some of the $P_1,...,P_m$. This data defines a flat family
$\Phi:{\mathfrak X}\to\C$, where ${\mathfrak X}=\Tor(\OG(\nu))$ and
$$\OG(\nu)=\{(i,j,c)\in\R^3_\Z\ :\ (i,j)\in P,\ c\ge\nu(i,j)\}$$ is the overgraph of $\nu$. The central fiber
${\mathfrak X}_0=\Phi^{-1}(0)$ splits into the union of toric surfaces $\Tor(P_k)$, $1\le k\le m$, while the other fibers are isomorphic to
$\Tor(P)$. The evaluation of the parameter $t$ turns the given curve $C$ into an inscribed family of curves
\begin{equation*}C^{(t)}\subset{\mathfrak X}_t,\quad C^{(t)}\in|{\mathcal L}_{P}|,\quad t\in(\C,0)\setminus\{0\}\ ,\end{equation*} (where $(\C,0)$ always means a sufficiently small disc in $\C$ centered at zero)
which closes up to a flat family over $(\C,0)$ with the central element
$$C^{(0)}=\bigcup_{k=1}^mC^{(0)}_k\ ,$$ where
$$C^{(0)}_k=\left\{F^{(0)}_k(x,y):=\sum_{(i,j)\in P_k\cap\Z^2}c^0_{ij}x^iy^j=0\right\}\in|{\mathcal L}_{P_k}|,
\ 1\le k\le m\ .$$
The function $\nu:P\to\R$ defines an embedded plane tropical curve $\Tr(C)$ in the sense of Section \ref{sec-tropcur}.
Its support is the closure of the valuation image of $C$.

We define the {\it embedded tropical limit} of $C$ to be the pair $\qquad\qquad$
$(\Tr(C),\{(P_k,C^{(0)}_k)\}_{k=1}^m)$, where the pairs $(P_1,C^{(0)}_1)$, ..., $(P_m,C^{(0)}_m)$ are called
{\it limit curves}, cf. \cite[Section 2]{Sh05}.

Note that whenever $P_i$ and $P_j$ share an edge, the corresponding limit curves $C^{(0)}_i$ and $C^{(0)}_j$ intersect at a point of $\Tor(P_i)\cap \Tor(P_j)$.
We call a plane tropical curve $\Gamma$ together with a collection of limit curves $\{(P_k,C^{(0)}_k)\}_{k=1}^m$ satisfying this compatibility condition an {\it enhanced plane tropical curve}.

\subsubsection{Parameterized tropical limit}\label{sec-ptl}
In the notation of the preceding section, let $\bn:(\widehat C,\bp)\to(C,\bu)$ be the normalization (where $\bp$ is a sequence of $n$ distinct points on $\widehat C$), or, equivalently, the family
\begin{equation}\bn_t:(\widehat C^{(t)},\bp(t))\to (C^{(t)},\bu(t))\hookrightarrow{\mathfrak X}_t,\quad t\in(\C,0)\setminus\{0\}\ ,
	\label{euc3}\end{equation} where
each $\widehat C^{(t)}$ is a smooth complex curve of genus $\gen(C)=g$ (cf. \cite[Theorem 1, page 73]{Tei} or \cite[Proposition 3.3]{ChL}).

After a suitable untwist $t\mapsto t^r$, the family
(\ref{euc3}) admits a flat extension to $0\in(\C,0)$
with the central element $\bn_0:(\widehat C^{(0)},\bp(0))\to{\mathfrak X}_0$, where $\widehat C^{(0)}$ is a connected nodal curve of arithmetic genus $g$ (see, for instance \cite[Theorem 1.4.1]{AV}), none of whose components is entirely mapped to a toric divisor, $\bp(0)$ is a sequence of $n$ distinct points of $\widehat C^{(0)}$ disjoint from the intersection points of the components of $\widehat C^{(0)}$,
and such that $(\bn_0)_*(\widehat C^{(0)},\bp(0))=(C^{(0)},\bu(0))$.

With the central fiber one can associate an enhanced plane marked
tropical curve $(\widehat\Gamma,\widehat h,\widehat\bg,\widehat\bq)$ as defined in \cite[Section 2.2.1]{Ty}. In particular (all other details can be found in \cite[Section 2]{Ty}),
\begin{itemize}\item The vertices of $\widehat\Gamma$ bijectively correspond to components of $\widehat C^{(0)}$, the finite edges of $\widehat\Gamma$ bijectively correspond to the intersection points of distinct components of $\widehat C^{(0)}$, and the infinite edges of $\widehat\Gamma$ correspond to the points of $\widehat C^{(0)}$ mapped to toric divisors $\Tor(e)$ such that $e\subset\partial P$;
	\item the genus function $\widehat\bg$ assigns zero to all edges of $\widehat\Gamma$, while for a vertex $V\in\widehat\Gamma^0$, we set $\widehat\bg(V)=\gen(\widehat C_V^{(0)})$, where $\widehat C_V^{(0)}$ is the corresponding component of $\widehat C^{(0)}$;
	\item the marked points $\widehat\bq$ are identified with the vertices of $\widehat\Gamma$ that correspond to the components of $\widehat C$ containing the points of the sequence $\bp$.
\end{itemize}
Thus, we obtain the
{\it parameterized tropical limit} (briefly, {\it PTL}) of $\quad$ $(\bn:\widehat C\to C,\bw)$ to be the pair consisting of the above enhanced plane marked tropical curve $(\widehat\Gamma,\widehat h,\widehat\bg,\widehat\bq)$ and the marked parameterized
complex curve $(\bn_0:\widehat C^{(0)}\to C^{(0)},\bw(0))$ in ${\mathfrak X}_0$. The parameterized tropical limit determines the parameterized embedded tropical limit, which consists of the stable maps $\bn:\widehat C^{(0)}_i\to \Tor(P_k)$, $k=1,...,m$, where $\widehat C^{(0)}_i$ ranges over the non-contracted components of $\widehat C^{(0)}$.

\section{Tropical formula for characteristic numbers: genus zero}\label{t4.3}

	In this section, we consider the case of $g=0$ and prove a correspondence between the tropical count and characteristic numbers in this case.

	Suppose that $g=0$ and accordingly
	\begin{equation}n=n_1+2n_2=|\Delta|-1.\label{ecn3}\end{equation}
	We omit the totally vanishing function $\bg$ from the notation. Denote by ${\mathcal M}^{trop}_{0,(n_1,n_2)}(P,\bu)$ the preimage of $\bu$ under $\text{Ev}:\modulispace \to \R^{2n}$.
	
	Using the genericity assumptions on $\bu$ we get, from Lemma \ref{lemma:dim_moduli_cells}, the following statement.
	\begin{lemma}The set ${\mathcal M}^{trop}_{0,(n_1,n_2)}(P,\bu)$ is finite.
		Each curve
		$${(\Gamma,h,\bq)\in{\mathcal M}^{trop}_{0,(n_1,n_2)}(P,\bu)}$$
		is trivalent and regular. Furthermore,
		\begin{itemize}\item
			$\bq_{(1)}\overset{\text{def}}{=}\{q^{(1)}_i\}_{i=1}^{n_1}\subset\Gamma\setminus\Gamma^0$, \item each vertex $V\in\Gamma^0\setminus\bq_{(2)}$ is not flat,
			\item each flat vertex $V$ either is incident to two cooriented ends of $\Gamma$, or the $\widehat h$-images of the edges incident to $V$ differ from each other.
			
		\end{itemize}
	\end{lemma}
	
	Now we define the multiplicity of a curve $(\Gamma,h,\bq)\in{\mathcal M}^{trop}_{0,(n_1,n_2)}(P,\bu)$ to be
	\begin{equation}\label{eq:trop_cn_def_genus0}
	CN^{trop}_0(\Gamma,h,\bq)=\frac{1}{|\Aut(\Gamma,h,\bq)|}\prod_{V\in\Gamma^0\setminus\bq}\mu(V).
	\end{equation}
	
	\begin{remark}
		In the considered case, $\Aut(\Gamma,h,\bq)\simeq(\Z/2)^k$, where $k$ is the number of flat vertices $q^{(2)}_i\in\Gamma^0$ incident to two ends of $\Gamma$.
	\end{remark}
	
	\begin{theorem}\label{tcn1}
		\begin{equation}CN_0(P,(n_1,n_2))=\sum_{(\Gamma,h,\bq)\in{\mathcal M}^{trop}_{0,(n_1,n_2)}(P,\bu)}CN^{trop}_0(\Gamma,h,\bq).
			\label{ecn4}\end{equation}
	\end{theorem}	
	
	\begin{remark}
		In the case of $P=T_d=\conv\{(0,0),(d,0),(0,d)\}$ (i.e., degree $d$ curves in the plane $\PP^2$), relation (\ref{ecn4}) follows from \cite[Theorem 8.4]{MR}, which yields that the right-hand side of (\ref{ecn4}) is the descendant Gromov-Witten invariant $\langle\tau_0(0)^{n_1}\tau_0(2)^{n_2}\rangle$, and from \cite[Proposition 3.4.1 and Theorem 4.1.4(0)]{GKP}, where, in the rational case, the above descendant invariants are reduced to the characteristic numbers.
	\end{remark}

	We prove Theorem \ref{tcn1} in few steps.
	We begin by showing that the tropical part of a parameterized tropical limit of a curve in ${\mathcal M}_{0,(n_1,n_2)}(P,\bw)$ is a tropical curve in ${\mathcal M}^{trop}_{0,(n_1,n_2)}(P,\bu)$, thus getting a well defined map $ {\mathcal M}_{0,(n_1,n_2)}(P,\bw)\to{\mathcal M}^{trop}_{0,(n_1,n_2)}(P,\bu)$ (Lemma \ref{lemma:genus_0_tropicalization_defined}).
	We next count the number of plane enhancements of a trivalent plane tropical curve satisfying the point constraints (Lemma \ref{lemma:number_of_plane_enhancements}).
	Enhanced tropical curves do not correspond bijectively to algebraic curves in $\mathcal{M}_{0,(n_1,n_2)}(P,\bw)$, to get a bijection we need to introduce the additional data of \emph{modifications} (Construction \ref{con:modified_tropical_limit} and Definition \ref{def:modification_data}).
	Finally we count the number of modifications in Lemma \ref{lemma:number_of_modifications} and prove they correspond bijectively to algebraic curves in $\mathcal{M}_{0,(n_1,n_2)}(P,\bw)$ in Lemma \ref{lemma:correspondence_final_step}, which concludes the proof.

	\begin{lemma}\label{lemma:genus_0_tropicalization_defined}
		Let $(\Gamma, h, \bq)$ be the tropical part of a parameterized tropical limit of a curve $[\bn:(\widehat C,\bp)\to\Tor_\K(P)]\in{\mathcal M}_{0,(n_1,n_2)}(P,\bw)$ (see Section \ref{sec-pr} or \cite[Section 1.3.2]{GS21}).
		Then $\bq^{(2)}\subset {\Gamma}^0$.
	\end{lemma}

	\begin{proof}
		Assume, towards a contradiction, that $q^{(2)}_j$ lies in the interior of some edge of $\Gamma$.
		We consider the case of a finite edge, whereas the case of an end can easily be settled in the same way.
		The modification along the edge containing $q^{(2)}_j$ admits the following description (cf. \cite[Section 2.2]{GS21} and Remark \ref{rcn2}).
		In the notation of Remark \ref{rcn2}, the subdivision of the Newton polygon of the polynomial $F_2(x,y_2)$ has a fragment whose pieces tile the triangle $T=\conv\{(-k_1,0),(k_2,0),(0,l)\}$ (see Figure \ref{fcn1}(a)) so that the segments $[(-k_1,0),(0,l)]$ and $[(k_2,0),(0,l)]$ are not subdivided.
		Furthermore, the limit curves corresponding to the pieces of the latter subdivision must be rational, and their union must be a curve of arithmetic genus zero being unibranch along the toric divisors $\Tor_\C([(-k_1,0),(0,l)])$ and $\Tor_\C([(k_2,0),(0,l)])$.
		It follows that the subdivision of $T$ has no vertices in the interior of $T$, thus, must be like shown in Figure \ref{fcn1}(b), and the corresponding limit curves must be unibranch along the toric divisors induced by the incline segments of the subdivision.
		The dual tropical modification is shown in Figure \ref{fcn1}(c). By construction, the coordinate $y_2$ of the point $w^{(2)}_j$ is $O(t^M)$ (cf. formula (\ref{ecn15})); hence the tropicalization $u^{(2)}_j$ of $w^{(2)}_j$ lies on one of the vertical ends (see Figure \ref{fcn1}(c)).
		The limit curve associated with the trivalent vertex incident to the end containing $u^{(2)}_j$ should have a parametrization of the form
		$$x=a\tau^l,\quad y_2=b\tau^r(\tau-1)^s,\quad s>0,\quad a,b\in\C,$$
		which however yields $\frac{dx}{d\tau}\big|_{\tau=1}=al\ne0$ contradicting the condition that the contact element (\ref{ecn6}) vanishes along the tangent to the curve.
		
		\begin{figure}
			\setlength{\unitlength}{1mm}
			\begin{picture}(140,40)(5,0)
				\thinlines
				\put(0,10){\vector(1,0){40}}\put(20,10){\vector(0,1){25}}
				\put(50,10){\vector(1,0){30}}\put(65,10){\vector(0,1){25}}
				\dashline{2}(90,35)(140,35)
				%
				
				
				\thicklines
				{\color{blue}
					\put(10,10){\line(1,0){20}}\put(10,10){\line(1,2){10}}\put(30,10){\line(-1,2){10}}
					\put(55,10){\line(1,0){20}}\put(55,10){\line(1,2){10}}
					\put(60,10){\line(1,4){5}}\put(65,10){\line(0,1){20}}\put(70,10){\line(-1,4){5}}\put(75,10){\line(-1,2){10}}
				}
				{\color{red}
					\put(92,35){\line(1,-1){10}}\put(102,25){\line(2,-1){10}}\put(112,20){\line(1,0){5}}\put(117,20){\line(2,1){10}}
					\put(127,25){\line(1,1){10}}\put(102,25){\line(0,-1){15}}\put(112,20){\line(0,-1){10}}\put(117,20){\line(0,-1){10}}
					\put(127,25){\line(0,-1){15}}
				}
				%
				%
				\put(17,0){(a)}\put(62,0){(b)}\put(112,0){(c)}
				\put(21,30){$\ell$}\put(66,30){$\ell$}\put(6,6){$-k_1$}\put(28,6){$k_2$}
				\put(101,15){$\bullet$}\put(95,15){$u^{(2)}_j$}
				%
				
				
			\end{picture}
			\caption{Proof of Lemma \ref{lemma:genus_0_tropicalization_defined}.}\label{fcn1}
		\end{figure}
		
		Whenever the condition $\bq^{(2)}\subset\Gamma^0$ is fulfilled, it follows from \cite[Lemma 2.21]{GMS} that each curve $C\in{\mathcal M}_{0,(n_1,n_2)}(P,\bw)$ tropicalizes to a tropical curve from ${\mathcal M}^{trop}_{0,(n_1,n_2)}(P,\bu)$.
	\end{proof}

	\begin{definition}
		We say that an enhanced plane tropical curve \\ $(\Gamma, h, \bq; \{\varphi_V\}_{V\in \Gamma^0})$ with $\bq\cap\Gamma^0 = \bq^{(2)}$ {\it satisfies the point constraints} if:
		\begin{enumerate}
			\item For every marked edge $E$, dual to $\sigma$, with endpoints $V_1$ and $V_2$ and marking $q_j^{(1)}$, the mutual point of intersection with the toric boundary $\varphi_{V_1}(p_{V_1,E}) = \varphi_{V_2}(p_{V_2,E})\in \Tor_\C(\sigma)$ is the one defined by $\ini(w_j)$.
			
			\item For every marked vertex $V=q\in \bq$ corresponding to the marking $w\in (\K^*)^2$, the curve ${\varphi_V}_*(\PP^1)$ pass through the point $\ini(w)$.
			
			\item For all $q_j^{(2)}$ ($j=1,\dots,n_2$), connected to the vertex $V\in \Gamma^0$, the curve $({\varphi_{V}})_*(\PP^1)$ is matching the contact element $\lambda_j^{(0)}:=\lambda^{(0)}_{j,x}\frac{dx}{x} + \lambda^{(0)}_{j,y}\frac{dy}{y}$ at the point $\ini(w_j^{(2)})$.
		\end{enumerate}
	\end{definition}

	\begin{lemma}\label{lemma:number_of_plane_enhancements}
		Let $(\Gamma, h, \bq)$ be a trivalent plane tropical curve.
		Then there exist
		\begin{equation*}
			\prod_{V\in \Gamma^0}\mu(V)\cdot \prod_{E\in \Gamma^1}\wt(E)^{-1} \cdot \prod_{E\in \Gamma^1, E\cap\bq\ne\emptyset}\wt(E)^{-1}
		\end{equation*}
	isomorphism classes of enhancements of $(\Gamma, h, \bq)$ that satisfy the point constraints.
	\end{lemma}

	\begin{proof}		
		If $q^{(2)}_i\in\Gamma^0\cap\bq$ is a non-flat vertex, the corresponding admissible limit curve $\varphi_V$ can be exended to a rational, peripherally unibranch curve $\PP^1\to C'\in|{\mathcal L}_{P'}|$ matching the constraint $((w^{(2)}_j)^{(0)},(\lambda_j)^{(0)})$. Assuming without loss of generality that (see Figure \ref{fig:plane_enhancements})
		$$P'=\conv\{(0,p),(r,0),(s,0)\},\quad p>0,\ r\ge0,\ s>r,$$ $$(w^{(2)}_j)^{(0)}=(\eta_j,\zeta_j),\quad(\lambda_j)^{(0)}=\lambda_{j,x}\frac{dx}{x}+\lambda_{j,y}\frac{dy}{y},$$
		we get that $\PP^1\to C'$ is given by
		$$x=a\tau^p,\quad y=b\tau^r(\tau-1)^{s-r},\quad a,b=\const\in\C^*,\quad\tau\in\C\subset\PP^1.$$ The conditions imposed on $C'$ amount to
		$$a\tau_0^m=\eta_j,\quad b\tau_0^r(\tau_0-1)^s=\zeta_j,\quad
		\frac{s\tau_0-r}{p(\tau_0-1)}=-\frac{\lambda_{j,x}\eta_j}{\lambda_{j,y}\zeta_j},$$ 
		which yield a unique solution $(a,b,\tau_0)$ defining the unique curve $\PP^1\to C'$.

		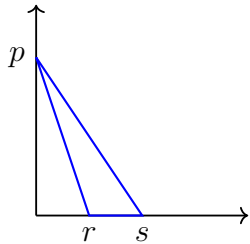
\begin{figure}
			\begin{tikzpicture}[scale = 0.7]
				\tikzset{
					every path/.style = {line width=0.7pt, blue},
					every node/.style = {black}
				}
				\draw[->, black] (0,0) -- (4,0);
				\draw[->, black] (0,0) -- (0,4);

				\draw [thick] (0,3) -- (1,0) -- (2,0) -- (0,3);
				\draw (0,3) node [left] {$p$};
				\draw (1,0) node [below] {$r$};
				\draw (2,0) node [below] {$s$};
			\end{tikzpicture}
			\caption{The triangle $P'$ in the proof of Lemma \ref{lemma:number_of_plane_enhancements}.}\label{fig:plane_enhancements}
		\end{figure}
		Similarly, if $V\in \Gamma^0\cap \bq $ is a flat trivalent vertex, then we can assume that image under $h$ of the edges adjacent to $V$ are all horizontal, and so $\varphi_V: \PP^1 \dashrightarrow (\C^*)^2 $ is given by
		$$ x = a\tau^r(\tau-1)^s,\quad y = b\quad a,b=\const\in\C^*,\quad r,s\in \NN \quad\tau\in\C\setminus\{0,1\}.$$
		The conditions imposed on $\varphi_V$ are
		$$a\tau_0^r(\tau_0-1)^s=\eta_j,\quad b=\zeta_j,\quad(r+s)\tau_0-r=0,$$
		which yield a unique solution $(a,b,\tau_0)$ defining the unique branched covering $\PP^1\to \PP^1\subset \PP^1\times \PP^1$.
		
		Finally, if $V\in\Gamma^0\setminus\bq$ is a trivalent vertex dual to a triangle $P'\subset P$, then $\PP^1\to C'\in|{\mathcal L}_{P'}|$ is a peripherally unibranch curve which matches points on the toric divisors dual to the incoming (with respect to the regular orientation on the connected component of $\Gamma\setminus\bq$ containing $V$) incident to $V$ edges $E_1,E_2$; these points are determined by the points of $\bw^{(0)}$ and by previously chosen admissible limit curves. By \cite[Lemma 3.5]{Sh05} there are $\mu(V)(\wt(E_1)\wt(E_2))^{-1}$ choices for such a curve.
	\end{proof}

	To get a unique lift to an algebraic curve we need to modify the tropical limit as follows.

	\begin{Construction}\label{con:modified_tropical_limit}
		Let $[\bn:(\widehat C,\bp)\to\Tor_\K(P)]\in{\mathcal M}_{0,(n_1,n_2)}(P,\bw)$ be a rational curve.
		We define the \emph{modified tropical limit of $C$} to be the parameterized tropical limit of $\widehat C$, together with \emph{modification data} associated to some of its fragments as follows:
		\begin{enumerate}
			\item The modification of an unmarked edge $E$ of weight $\ell>1$ is described in \cite[Sections 3.5 and 3.6]{Sh05} or \cite[Section 2.2(2)]{GS21}, where it is shown that if we perform a toric change of coordinates that makes $E$ horizontal, there exist a unique $\zeta \in \K$ such that the coefficient of $y^{\ell-1}$ in the polynomial $F'(x,y):=F(x,y+\zeta)$ vanishes, and that this results in adding a strata isomorphic to $\Tor_\C(P_{mod})$ to the degeneration of surfaces, where $$P_{mod}:=\conv\{(-1,0),(1,0),(0,\ell)\}.$$
			We record this toric change of coordinates, the shift $\zeta$, and the limit curve $\PP^1 \to C_{mod}\subset\Tor_\C(P_{mod})$ as the modification data associated to $E$.

			\item The modification along an edge $E$ of $\Gamma$ having weight $\ell>1$ and containing a point $q^{(1)}_i$ is, in fact, described in \cite[Section 5.4]{Sh05} or \cite[Section 2.5.9]{IMS}. We, however, provide here another presentation, which yields a better geometric view and keeps the style accepted in the present construction.
			Namely, assuming that the $h$-image of the considered edge is horizontal (see Figure \ref{fcn2}(a)), and $w^{(1)}_i=(\xi^{(1)}_i+O(t^M)),\eta^{(1)}_i+O(t^M)))$, we perform the coordinate change $y=y_1+\eta^{(1)}_i$. Since the limit curves associated with the trivalent vertices $V_1,V_2$ are smooth at their common point (see \cite[Lemma 3.5]{Sh05}), the Newton polygons of their defining polynomials after the above coordinate change leave empty the triangle $P_{mod}=\conv\{(-1,0),(1,0),(0,\ell)\}$, Figure \ref{fcn2}(b) (cf. \cite[Figure 2.18]{IMS}). Furthermore, the new coordinates of the point $w^{(1)}_i$ are  $(\xi^{(1)}_i+O(t^M)),O(t^M)))$; hence, in view of $M\gg0$, its tropicalization $u^{(1)}_i$ lies on the vertical end oriented to $-\infty$. Taking into account that the union of the limit curve in the modification must be of arithmetic genus zero, we end up with the tropical modification consisting of $2$ trivalent vertices $V_1'$ and $V_2'$ joined by an horizontal edge $E'$ as shown in Figure \ref{fcn2}(c).
			Its dual subdivision is the two triangles
			$$P_{1,mod}=\conv\{(-1,0),(0,0),(0,\ell)\},\quad P_{2,mod}=\conv\{(0,0),(1,0),(0,\ell)\},$$ (Figure \ref{fcn2}(d)), while the limit curves $\PP^1\to C_{1,mod}$, $\PP^1\to C_{2,mod}$ are peripherally unibranch.
			
			The modification data associated to $E$ is the toric change of coordinates that makes $E$ horizontal, the curves $C_{1, mod}, C_{2, mod}$, and the modification data associated to the edge $E'$ in the modified curve.

			\item The last modification we consider is along the edges incident to a marked flat vertex $q^{(2)}_j=V\in\Gamma^0$.
			It can be incident to either two ends of $\Gamma$, or to one end and two finite edges of $\Gamma$, or to only finite edges of $\Gamma$ (see Figures \ref{fcn2}(e,f,g)). The modification results in (cf. Remark \ref{rcn2}, formula (\ref{ecn15}))
			\begin{equation*}w^{(2)}_j=(\xi^{(2)}_j+O(t^M),O(t^M)),\quad\lambda_j=\lambda_{j,x}
				\frac{dx}{x},\quad\xi^{(2)}_j,\lambda_{j,x}\ne0.\end{equation*}
			It follows that the modified fragments of the tropical curve look as shown in Figures \ref{fcn2}(h,i,j), respectively.

			\begin{enumerate}
				\item [(a)] In the case of two unbounded ends adjacent to $q^{(2)}_j$ (see Figure \ref{fcn2}(e) for the original fragment and Figure \ref{fcn2}(h) for the modification), both the ends have weight $1$; hence, the dual triangle $P_{mod}$ of the modification is equal to $\conv\{(-1,0), (0,2), (0,0)\} $ (see Figure \ref{fcn2}(k)). The corresponding limit curve $\PP^1\to C_{mod}\in|{\mathcal L}_{P_{mod}}|$ crosses the toric divisors transversally.

				The modification data associated with the marked flat vertex $q^{(2)}_j$ is the toric change of coordinates that makes the ends horizontal together with the curve $C_{mod}$.
			
				\item[(b)] In the case of one unbounded end adjacent to $q^{(2)}_j $ (see Figure \ref{fcn1}(f) for the original fragment, and Figure \ref{fcn1}(i) for the modification), the end has weight $1$, while the finite edge incident to $V$ has weight $\ell-1$ ($\ell\ge2$); hence the dual subdivision is the union of the triangle $P_{1,mod}:=\conv\{ (-1,0), (0, \ell), (0,0)\}$ and the trapeze $P_{2,mod}:=\conv\{(0,0), (0,\ell), (1,1), (1,0)\}$ (see Figure \ref{fcn2}(l)). The limit curve $\PP^1\to C_{1,mod}\in|{\mathcal L}_{P_{1,mod}}|$ crosses the toric divisor $\Tor([(0,0),(0,\ell)])$ at a point $z_1$ transversally and at a point $z_2$ with multiplicity $\ell-1$.
				The limit curve associated with $P_{2,mod}$ is the union of the horizontal line $C'_{2,mod}$ through the point $z_1\in \Tor([(0,0),(0,\ell)])$ and the peripherally unibranch component $\PP^1\to C''_{2,mod}$ crossing the toric divisor $\Tor([(0,0),(0,\ell)])$ at $z_2$.

				The modification data associated with the marked flat vertex $q^{(2)}_j$ is the toric change of coordinates that makes the end horizontal together with the curves $C_{1,mod}$, $C'_{2, mod} $,  and $C''_{2,mod}$, as well as the modification data associated to the bounded horizontal edge of the modified fragment.
			
				\item[(c)] We now deal with the case when all edges adjacent to $q^{(2)}_j$ are bounded (see Figure \ref{fcn2}(g))
				The modification is depicted in Figure \ref{fcn2}(j), its dual subdivision consists of the tringle $P_{1,mod}:=\conv\{(-1,0), (0,0), (0,\ell)\} $, the trapeze $P_{2,mod}:=\conv \{ (0,0), (0,\ell), (1,r), (1,0)\} $ for some $0< r < \ell$, and the tringle $P_{3, mod}:=\conv\{ (0,r), (1,r), (2,0)\} $, see Figure \ref{fcn2}(m).
				The limit curve $\PP^1\to C_{1,mod}\in|{\mathcal L}_{P_{1,mod}}|$ crosses the toric divisor $\Tor([(0,0,),(0,\ell)])$ in two points $z_1,z_2$ with multiplicities $r$, $\ell-r$, respectively.
				The limit curve associated with the trapeze $P_{2,mod}$ consists of the $r$-multiple covering of the horizontal line $C'_{2,mod}$ through $z_1$ and of a peripherally unibranch component $\PP^1\to C''_{2,mod}$ through the point $z_2$.
				The limit curve $\PP^1\to C_{3,mod}\in|{\mathcal L}_{P_{3,mod}}|$ for $P_{3,mod}$ is peripherally unibranch.

				The modification data associated with the marked flat vertex $q^{(2)}_j$ is the toric change of coordinates that makes the edges horizontal together with the curves $C_{1,mod}$, $C'_{2,mod}$, $C''_{2,mod}$, and $C_{3,mod}$, as well as the modification data associated to the horizontal edges of the modified fragment.
			\end{enumerate}
		\end{enumerate}
	\end{Construction}

	\begin{figure}
		\setlength{\unitlength}{1mm}
		\begin{picture}(145,175)(0,0)
			\thinlines
			\put(0,50){\vector(1,0){25}}\put(15,50){\vector(0,1){20}}
			\dashline{2}(0,110)(25,110)\dashline{2}(35,110)(75,110)\dashline{2}(85,110)(130,110)
			\dashline{2}(70,170)(100,170)
			\put(35,50){\vector(1,0){40}}\put(50,50){\vector(0,1){20}}
			\put(85,50){\vector(1,0){40}}\put(100,50){\vector(0,1){20}}
			\put(35,150){\vector(1,0){30}}\put(50,150){\vector(0,1){20}}
			\put(105,150){\vector(1,0){30}}\put(120,150){\vector(0,1){20}}
			
			\dottedline{1}(50,55)(60,55)\dottedline{1}(100,55)(110,55)
			
			\thicklines
			{\color{blue}
				\put(5,50){\line(2,3){10}}\put(5,50){\line(1,0){10}}\put(15,50){\line(0,1){15}}
				\put(40,50){\line(2,3){10}}\put(40,50){\line(1,0){20}}\put(50,50){\line(0,1){15}}
				\put(60,50){\line(0,1){5}}\put(50,65){\line(1,-1){10}}
				\put(90,50){\line(2,3){10}}\put(90,50){\line(1,0){30}}\put(100,50){\line(0,1){15}}
				\put(110,50){\line(0,1){5}}\put(100,65){\line(1,-1){10}}
				\put(110,55){\line(2,-1){10}}
				\put(40,150){\line(2,3){10}}\put(40,150){\line(1,0){20}}\put(60,150){\line(-2,3){10}}
				\put(110,150){\line(2,3){10}}\put(110,150){\line(1,0){20}}\put(130,150){\line(-2,3){10}}\put(120,150){\line(0,1){15}}
			}
			{\color{red}
				\put(5,110){\line(1,-1){10}}\put(15,100){\line(0,-1){10}}\put(15,100){\line(1,0){10}}\put(15,99){\line(1,0){10}}
				\put(40,110){\line(1,-1){10}}\put(50,100){\line(0,-1){10}}\put(50,100){\line(1,0){10}}\put(50,99){\line(1,0){9}}
				\put(61,99){\line(1,0){9}}\put(60,100){\line(0,-1){10}}\put(60,100){\line(1,1){10}}
				\put(90,110){\line(1,-1){10}}\put(100,100){\line(0,-1){10}}\put(100,100){\line(1,0){10}}\put(100,99){\line(1,0){9}}
				\put(111,99){\line(1,0){4}}\put(110,100){\line(0,-1){10}}\put(110,100){\line(1,1){10}}
				\put(115,99){\line(0,-1){9}}\put(115,99){\line(1,1){11}}
				\put(5,130){\line(1,0){10}}\put(15,130.5){\line(1,0){10}}\put(15,129.5){\line(1,0){10}}
				\put(40,130){\line(1,0){10}}\put(50,130.5){\line(1,0){10}}\put(50,129.5){\line(1,0){9}}
				\put(61,129.5){\line(1,0){9}}\put(60,130.5){\line(0,-1){6}}\put(60,130.5){\line(1,1){5}}
				\put(90,130){\line(1,0){10}}\put(100,130.5){\line(1,0){10}}\put(100,129.5){\line(1,0){9}}
				\put(111,129.5){\line(1,0){4}}\put(110,130.5){\line(0,-1){6}}\put(110,130.5){\line(1,1){5}}
				\put(115,129.5){\line(0,-1){5}}\put(115,129.5){\line(1,1){6}}
				\put(0,155){\line(1,1){5}}\put(0,165){\line(1,-1){5}}\put(5,160){\line(1,0){15}}
				\put(20,160){\line(1,1){5}}\put(20,160){\line(1,-1){5}}
				\put(75,170){\line(1,-1){5}}\put(80,165){\line(1,0){10}}\put(80,165){\line(0,-1){15}}
				\put(90,165){\line(0,-1){15}}\put(90,165){\line(1,1){5}}
			}
			\put(12,140){(a)}\put(46,140){(b)}\put(83,140){(c)}\put(116,140){(d)}
			\put(12,120){(e)}\put(51,120){(f)}\put(98,120){(g)}
			\put(12,80){(h)}\put(53,80){(i)}\put(103,80){(j)}
			\put(11,40){(k)}\put(50,40){(l)}\put(98,40){(m)}
			\put(50,165){$\ell$}\put(120,165){$\ell$}\put(50,65){$\ell$}\put(100,65){$\ell$}\put(15,64){$2$}
			\put(36,146){$-1$}\put(57,146){$1$}\put(106,146){$-1$}\put(127,146){$1$}
			\put(2,46){$-1$}\put(37,46){$-1$}\put(87,46){$-1$}\put(57,46){$1$}\put(107,46){$1$}\put(117,46){$2$}
			\put(8,155){$u^{(1)}_i$}\put(73,154){$u^{(1)}_i$}\put(8,94){$u^{(2)}_j$}\put(43,94){$u^{(2)}_j$}\put(93,94){$u^{(2)}_j$}
			\put(4,161.5){$V_1$}\put(16,161.5){$V_2$}\put(72,166){$V_1$}\put(95,166){$V_2$}
			\put(68,112){$V_1$}\put(117,112){$V_1$}\put(124,112){$V_2$}
			\put(56.5,132){$V_1$}\put(106.5,132){$V_1$}\put(116,126){$V_2$}
			\put(9,159){$\bullet$}\put(79,154){$\bullet$}
			\put(14,129){$\bullet$}\put(49,129){$\bullet$}
			\put(99,129){$\bullet$}\put(79,154){$\bullet$}
			\put(14,94){$\bullet$}\put(49,94){$\bullet$}\put(99,94){$\bullet$}
			\put(99.5,101.5){$V_3$}\put(106,101.5){$V_4$}
			\put(116,96){$V_5$}
			
		\end{picture}
		\caption{The modifications in Construction \ref{con:modified_tropical_limit} and Definition \ref{def:modification_data}. }\label{fcn2}
	\end{figure}
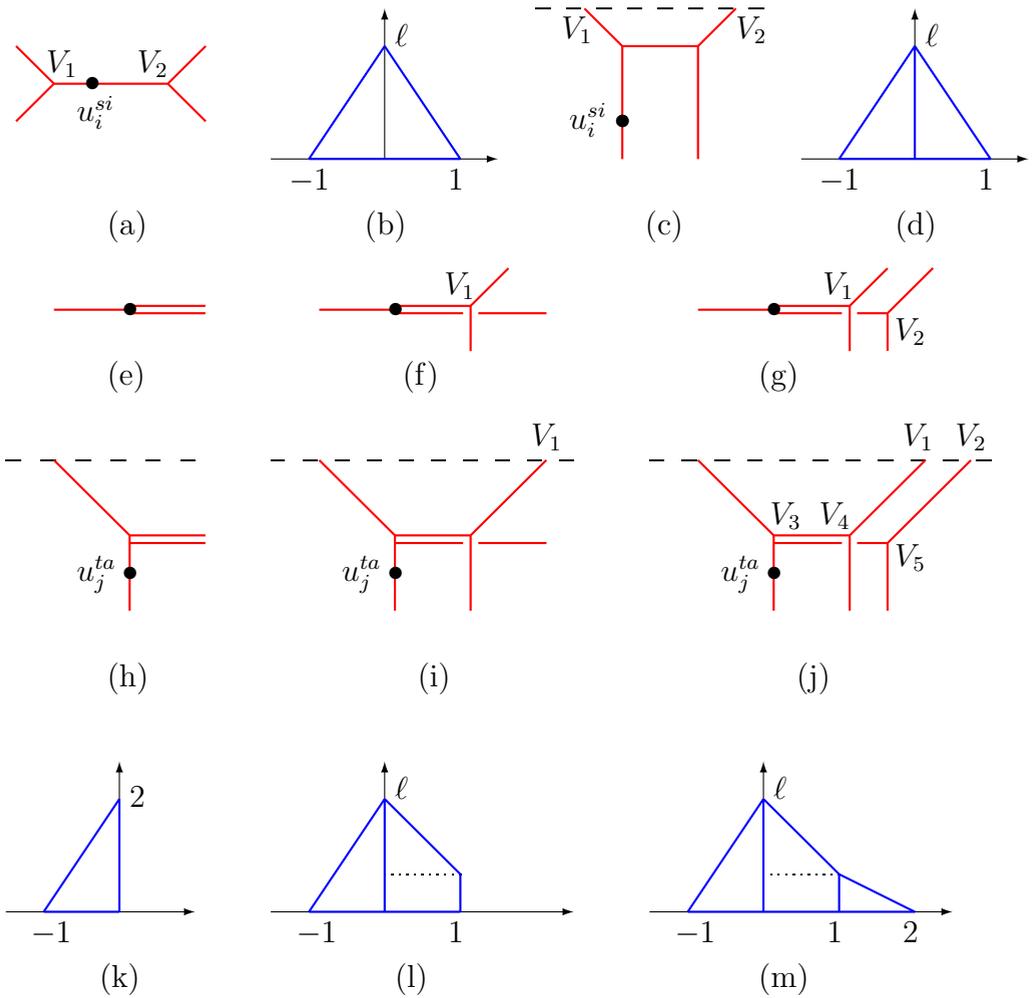
	
	\begin{remark}
		The modification data associated to a marked flat vertex together adjacent to one (respectively two, respectively three) bounded edges is equivalent to a toric change of coordinates that makes the edges horizontal together with an enhanced tropical curve whose tropical part is depicted in Figure \ref{fcn2}(h) (respectively Figure \ref{fcn2}(i), respectively Figure \ref{fcn2}(j)) and possibly with additional modification data corresponding to the horizontal edges of the modified fragment.
	\end{remark}

	The following definition is the precise description of what limit curves can occur as a modification data of the modified tropical limit.

	\begin{definition}\label{def:modification_data}
		Let $(\Gamma, h, \bq; \{\varphi_V:\PP^1\to C_{V}\subset \Tor_\C(P_V)\}_{V\in \Gamma^0})$ be an enhanced plane tropical curve.
		Let $X\subset \Gamma$ be a fragment which is either an edge adjacent to two non flat vertices, or a flat marked vertex together with the edges incident to it.
		We define \emph{modification data associated to $X$} separately for each case as follows:
		\begin{enumerate}
			\item A modification data associated to an unmarked edge $E$ of weight $\ell$ adjacent to two non flat vertices $V_1$ and $V_2$ is a triple $(\Phi, \zeta, \varphi_{mod})$, where $\Phi$ is a toric change of coordinates that make $E$ horizontal, $\zeta\in \K$, and $\varphi_{mod}:\PP^1\to C_{mod}\subset \Tor_\C(P_{mod})$ is a rational curve, where $P_{mod}=\conv\{(-1,0),(1,0),(0,\ell)\}$, and additionally the coefficient of $y^{\ell-1}$ in the defining equation of $C_{mod}$ vanishes.
			
			Denote by $\sigma = P_{V_1}\cap P_{V_2}$ the edge of the dual subdivision corresponding to $E$.
			Both $C_{V_1} $ and $C_{V_2}$ intersect $\Tor_\C(\sigma)$ at unique point $z$.
			We say the modification data is \emph{coherent} if the second coordinate of $z$ (after applying $\Phi$) is $\zeta(0)$ and the defining equation of $C_{mod}$ restricted to $\conv\{(-1,0), (0,\ell)\}$ (respectively, $\conv\{(1,0), (0, \ell)\}$) is the local equation of the defining equation of $C_{V_1}$ (respectively, $C_{V_2}$) at the point of intersection with the toric divisor $\Tor_\C(\sigma)$ where $\sigma = P_{V_1}\cap P_{V_2}$ is the edge of the dual subdivision corresponding to $E$.
			Explicitly, in the coordinates on $\Tor_\C(P_{V_i})$ where $\Tor_\C(\sigma)$ is defined by $x=0$, the limit curve $C_{V_i}$ ($i=1,2$) intersect $\Tor_\C(\sigma)$ in a single point $z=(0,\zeta(0))$ where its defining equation is of the form $ (y-\zeta(0))^\ell + \theta_i(y) x + O(x^2)$ for some $\theta_i\in \C[y]$, so coherence means that the restriction of the defining equation of $C_{mod}$ to $\conv\{(-1, 0) (0, \ell)\}$ (respectively $\conv\{(1,0), (0, \ell)\}$) is $y^\ell + \theta_i(\zeta(0)) x^{-1}$ (respectively $y^\ell + \theta_i(\zeta(0))x$).

			\item
			A modification data associated to an edge $E$ of $\Gamma$ having weight $\ell>1$ and containing a point $q^{(1)}_i$ is a tuple $(\Phi, \varphi_{1,mod}, \varphi_{2,mod}, \Omega)$, where $\Phi$ is a toric change of coordinates that makes $E$ horizontal, $\varphi_{1,mod}:\PP^1\to C_{1,mod}\subset \Tor_\C(P_{1,mod})$ and $\varphi_{2,mod}:\PP^1\to C_{2,mod}\subset \Tor_\C(P_{2,mod})$ are rational peripherally unibranch curves, where $P_{1,mod}=\conv\{(-1,0),(0,\ell),(0,0)\}$ and $P_{2,mod}=\conv\{(0,0),(1,0),(0,\ell)\}$, and $\Omega$ is a modification data associated to the unmarked edge dual to $\conv\{(0,0), (0,\ell)\}$.

			This data is called \emph{coherent} if $\Omega$ is coherent, the intersection of $C_{1,mod}$ with the toric divisor $\Tor_\C(\conv\{(-1,0), (0,0)\})$ is $\xi^{(1)}_i$, the initial term of the first coordinate of the marked point $w^{(1)}_i$ after the toric coordinate change, and additionally the defining equation of $C_{1,mod}$ (respectively $C_{2,mod}$) restricted to $\conv\{(-1,0), (0,\ell)\}$ (respectively $\conv\{(0,\ell), (1, 0)\}$) is the local equation of the defining equation of $C_{V_1}$ (respectively $C_{V_2}$) at the point of intersection with the toric divisor $\Tor_\C(\sigma)$ where $\sigma = P_{V_1}\cap P_{V_2}$ is the edge of the dual subdivision corresponding to $E$.

			\item
			A modification data associated to a flat vertex $V$ marked by $q^{(2)}_j$ adjacent to one bounded edge of weight $2$ whose other end is $V'$ and $2$ unbounded ends of weight $1$ is the pair $(\Phi, \varphi_{mod})$, where $\Phi$ is a toric change of coordinates that makes the edges horizontal, and $\varphi_{mod}:\PP^1\to C_{mod}\subset \Tor_\C(P_{mod})$ is a rational curve, where $P_{mod}=\conv\{(-1,0),(0,0), (0,2)\}$.

			It is coherent if the intersection of $C_{mod}$ with the toric divisor $\Tor_\C(\conv\{(-1,0), (0,0)\}$ is $\xi^{(2)}_j$, the initial term of the first coordinate of the marked point $w^{(2)}_j$ after the toric coordinate change, and additionally the defining equation of $C_{mod}$ restricted to $\conv\{(-1,0), (0,2)\}$ is the local equation of the defining equation of $C_{V'}$ at the point of intersection with the toric divisor $\Tor_\C(\sigma)$ where $\sigma$ is the edge of the dual subdivision corresponding to $E$.

			\item
			A modification data associated to a flat vertex $V$ marked by $q^{(2)}_j$ adjacent to two bounded edges $E_1$ of weight $\ell$ and $E_2$ of weight $\ell-1$ whose other endpoints are $V_1$ and $V_2$ respectively and to one unbounded end is the tuple $(\Phi, \Gamma_{mod}, \Omega)$, where $\Phi$ is a toric change of coordinates that makes the edges horizontal, $\Gamma_{mod}$ is an enhanced tropical curve whose tropical part is depicted in Figure \ref{fcn2}(i), and $\Omega$ is a modification data associated to the horizontal edge of weight $\ell-1$ in $\Gamma_{mod}$.

			This data is coherent if $\Omega$ is coherent, the intersection of $C_{1, mod}$ with the toric divisor $\Tor_\C(\conv\{(-1,0), (0,0)\})$ is $\xi^{(2)}_j$, the initial term of the first coordinate of the marked point $w^{(2)}_j$ after the toric coordinate change, and additionally the defining equation of $C_{1, mod}$ (respectively $C'_{2,mod}$) restricted to $\conv\{(-1,0), (0,\ell)\}$ (respectively $\conv\{(0, \ell), (1, 1)\}$) is the local equation of the defining equation of $C_{V_1}$ (respectively $C_{V_2}$) at the point of intersection with the toric divisor $\Tor_\C(\sigma_1)$ (respectively $\Tor_\C(\sigma_2)$) where $\sigma_i$ is the edge of the dual subdivision corresponding to $E_i$.

			\item
			A modification data associated to a flat vertex $V$ marked by $q^{(2)}_j$ adjacent to three bounded edges $E_1$ of weight $\ell$, $E_2$ of weight $r$ and $E_3$ of weight $\ell-r$ whose other endpoints are $V_1$, $V_2$ and $V_3$ respectively is the tuple $(\Phi, \Gamma_{mod}, \Omega_1, \Omega_2)$, where $\Phi$ is a toric change of coordinates that makes the edges horizontal, $\Gamma_{mod}$ is an enhanced tropical curve whose tropical part is depicted in Figure \ref{fcn2}(j), and $\Omega_1$ and $\Omega_2$ are modification data associated to the horizontal edges of weight $r$ and $\ell-r$ respectively.
			
			This data is coherent if $\Omega_1$ and $\Omega_2$ are coherent, the intersection of $C_{1, mod}$ with the toric divisor $\Tor_\C(\conv\{(-1,0), (0,0)\})$ is $\xi^{(2)}_j$, the initial term of the first coordinate of the marked point $w^{(2)}_j$ after the toric coordinate change, and additionally the defining equation of $C_{1, mod}$ (respectively $C'_{2,mod}$) restricted to $\conv\{(-1,0), (0,\ell)\}$ (respectively $\conv\{(0, \ell), (1, r)\}$) is the local equation of the defining equation of $C_{V_1}$ (respectively $C_{V_2}$) at the point of intersection with the toric divisor $\Tor_\C(\sigma_1)$ (respectively $\Tor_\C(\sigma_2)$) where $\sigma_i$ is the edge of the dual subdivision corresponding to $E_i$.

		\end{enumerate}
	\end{definition}

	\begin{remark}
		It is obvious that the modification data constructed in Construction \ref{con:modified_tropical_limit} are coherent.
	\end{remark}

	\begin{definition}\label{def:equivalence_modification}
		The toric changes of coordinates in $GL_2(\Z)$ that make a given edge horizontal form a coset of the stabilizer
		$$\operatorname{Stab}_{GL_2(\Z)}((1,0)) = \left\{ \begin{pmatrix} 1 & a \\ 0 & \epsilon \end{pmatrix} : a \in \Z,\, \epsilon = \pm 1 \right\}.$$
		Two modification data associated to a fragment $X$ are \emph{equivalent} if they arise from applying Construction~\ref{con:modified_tropical_limit} to the same algebraic curve $[\bn:(\widehat C,\bp)\to\Tor_\K(P)]\in{\mathcal M}_{0,(n_1,n_2)}(P,\bw)$ with toric changes of coordinates $\Phi_1$ and $\Phi_2 = T\Phi_1$, where $T\in\operatorname{Stab}_{GL_2(\Z)}((1,0))$.
	\end{definition}

	\begin{example}\label{ex:equivalent_modification}
		Consider the tropical curve dual to the subdivision $\{\conv\{(0,1),\allowbreak(1,0),\allowbreak(1,2)\},\allowbreak \conv\{(1,0),\allowbreak(1,2),\allowbreak(2,1)\}\}$ of $P=\conv\{(0,1),(1,0),(1,2),(2,1)\}$, having two trivalent vertices $V_1, V_2$ joined by a single unmarked bounded edge $E$ of weight $\ell=2$, dual to the shared edge $\sigma = \conv\{(1,0),(1,2)\}$ (see Figure \ref{fig:equivalent_modification}).
		Since $\sigma$ has primitive direction $(0,1)$, the edge $E$ has direction $(1,0)$ and is already horizontal.
		The modification polygon is $P_{mod} = \conv\{(-1,0),(1,0),(0,2)\}$, and the defining equation of $C_{mod}$ takes the form
		$$y^2 + \alpha x^{-1} + \beta + \gamma x = 0,$$
		where $\alpha,\gamma$ are fixed by the coherence conditions. The rationality requirement (the unique interior lattice point $(0,1)$ of $P_{mod}$ corresponds to the vanishing monomial $y$) forces $C_{mod}$ to be nodal, i.e., $\beta^2 = 4\alpha\gamma$, yielding exactly $\ell=2$ modification curves $\varphi_{mod}$ and $\widetilde\varphi_{mod}$, corresponding to $\beta = \pm 2\sqrt{\alpha\gamma}$.

		Now take $\Phi_2 = \begin{psmallmatrix} 1 & 1 \\ 0 & 1 \end{psmallmatrix} = T\Phi_1$ with $T = \begin{psmallmatrix} 1 & 1 \\ 0 & 1 \end{psmallmatrix} \in \operatorname{Stab}_{GL_2(\Z)}((1,0))$.
		Since $T$ acts on the torus by $(x,y)\mapsto(xy,y)$, the $y$-coordinate is unchanged and hence $\zeta'=\zeta$.
		Applying Construction~\ref{con:modified_tropical_limit}(1) to the same algebraic curve in $\Phi_2$-coordinates yields a modification curve $\varphi'_{mod}$ with defining equation
		$$y^2 + \alpha' x^{-1} + \beta' + \gamma' x = 0,$$
		where $\alpha',\gamma'$ are determined by the coherence conditions applied to the limit curves expressed in $\Phi_2$-coordinates and ${\beta'}^2=4\alpha'\gamma'$.
		The data $(\Phi_1,\zeta,\varphi_{mod})$ and $(\Phi_2,\zeta,\varphi'_{mod})$ are then equivalent.
		On the other hand, $(\Phi_1, \zeta, \varphi_{mod})$ and $(\Phi_1, \zeta, \widetilde\varphi_{mod})$ are not equivalent, as the two modification curves (with $\beta = \pm 2\sqrt{\alpha\gamma}$) correspond to distinct algebraic curves.
	\end{example}

	\begin{figure}[h]
	\centering
	\begin{tikzpicture}[scale=1.5, font=\small, thick]
		\begin{scope}[yshift=-3cm]


	\draw[draw=blue] (0,1) -- (1,0) -- (2,1) -- (1,2) -- cycle;

	\draw[blue,very thick] (1,0) -- node [right] {\footnotesize $\sigma$} (1,2);

	\draw[blue] (0,1) -- (1,2);
	\draw[blue] (0,1) -- (1,0);

	\draw[blue] (1,2) -- (2,1);
	\draw[blue] (2,1) -- (1,0);

	\foreach \pt in {(0,1), (1,0), (1,1), (1,2), (2,1)} {
	  \fill[blue] \pt circle (1.0pt);
	}
		\end{scope}


	\begin{scope}[xshift=0.0cm, scale=0.7]
		\draw[red] (-1, -1) -- (0,0) -- (-1, 1);
		\draw[red] (0,0) -- node [above] {\footnotesize $E$} (2, 0);
		\draw[red] (2,0) -- (3, 1);
		\draw[red] (2,0) -- (3, -1);
		\node[red, right] at (2,0) {\footnotesize $V_2$};
		\node[red, left] at (0,0) {\footnotesize $V_1$};
	\end{scope}

	\begin{scope}[yshift=-3cm, xshift=4cm]

	
		\draw[draw=blue] (0,2) -- (1,0) -- (2,0) -- (1,2) -- cycle;
	
		\draw[blue,very thick] (1,0) -- (1,2);
	
		\draw[blue] (0,2) -- (1,2);
		\draw[blue] (0,2) -- (1,0);
	
		\draw[blue] (1,2) -- (2,0);
		\draw[blue] (2,0) -- (1,0);
	
		\foreach \pt in {(0,2), (1,0), (1,1), (1,2), (2,0)} {
		  \fill[blue] \pt circle (1.0pt);
		}
			\end{scope}
	
	
		\begin{scope}[xshift=4cm, scale=0.7]
			\draw[red] (-2, -1) -- (0,0) -- (0, 1);
			\draw[red] (0,0) -- (2, 0);
			\draw[red] (2,0) -- (4, 1);
			\draw[red] (2,0) -- (2, -1);
		\end{scope}

	\end{tikzpicture}
	\caption{Tropical curve and dual subdivision as in Example~\ref{ex:equivalent_modification}.}
	\label{fig:equivalent_modification}
	\end{figure}

	\begin{lemma}\label{lemma:number_of_modifications}
		Let $(\Gamma, h, \bq; \{C_V\}_{V\in \Gamma^0})$ be an enhanced plane tropical curve satisfying the point constraints.
		Then there exist
		$$\prod_{E\in \Gamma^1}\wt(E) \cdot \prod_{\substack{i \\ q^{(1)}_i \in E}}\wt(E)$$
		equivalence classes of coherent modification data associated to $\Gamma$.
	\end{lemma}

	\begin{proof}
		We enumerate the possible coherent modification data associated to each fragment as in Definition \ref{def:modification_data}.

		\begin{enumerate}

			\item For an unmarked edge $E$ of $\Gamma$ of weight $\ell$, the admissible modified limit curve is $\PP^1\to C'\subset\Tor(P')$, where $$P'=\conv\{(-1,0),(1,0),(0,\ell)\}$$ (see Figure \ref{fcn2}(b)). The constraints are the fixed points on the toric divisors $\Tor([(-1,0),(0,\ell)])$ and $\Tor([(1,0),(0,\ell)])$, determined by the chosen admissible limit curves of the embedded part, and the condition of the vanishing coefficient of $y^{\ell-1}$ (cf. \cite[Lemma 3.9]{Sh05} and \cite[Lemma 2.49]{IMS}). These restrictions are satisfied by exactly $\ell$ curves \cite[Lemma 3.9]{Sh05}.
		
			\item For an edge of $\Gamma$ of weight $\ell>1$ containing a point $q^{(1)}_i$, we should construct two admissible limit curves $C_{1,mod},C_{2,mod}$, as well as modification data for an unmarked edge of weight $\ell$, as described in item 2 of Definition \ref{def:modification_data}.
			A curve $C_{1,mod}\in|{\mathcal L}_{P_{1,mod}}|$ must match a point on the divisor $\Tor([(-1,0),(0,\ell)])$, prescribed by previously chosen admissible embedded limit curves, and the point $\xi^{(1)}_i\in\Tor([(-1,0),(0,0)])$ , and should cross the divisor $\Tor([(0,0),(0,\ell)])$ at one point. Its equation turns to be $(y-\tau)^\ell+ax^{-1}=0$ with a fixed $a$ and a fixed $\tau^\ell$, which yields $\ell$ options for the $C_{1,mod}$. In turn, the curve $C_{2,mod}\in|{\mathcal L}_{P_{2,mod}}|$ is uniquely determined by the prescribed points on the divisors $\Tor([(0,0),(0,\ell)])$ and $Tor([(1,0),(0,\ell)])$.
			At last, the modification data of the unmarked edge have $\ell$ options as shown above.
		
			\item Let $q^{(2)}_j=V\in\Gamma^0$ be a marked flat vertex of $(\Gamma,h,\bq)$. Among the three cases dealt with in Definition \ref{def:modification_data}(3)-(5), we consider the one with all edges incident to $q^{(2)}_j$ finite; the other two admit the same treatment, we only provide final answers. The tropical modification in the considered case and the dual subdivision are shown in Figure \ref{fcn2}(j,m). The admissible modified limit curve $$C_{1,mod}\in|{\mathcal L}_{P_{1,mod}}|,\quad P_{1,mod}=\conv\{(-1,0),(0,0),(0,\ell)\},$$
			matches the fixed point on the divisor $\Tor([(-1,0),(0,\ell)])$, passes through the point $\eta_j\in\Tor([(-1,0),(0,0)])$ having there a vertical tangent (see Remark \ref{rcn2} and Step 1 of the proof), and crosses the divisor $\Tor([(0,0),(0,\ell)])$ at two points $z_1,z_2$ with multiplicities $r,\ell-r$, respectively. Hence, it is given by an equation $ax^{-1}+(y-z_1)^r(y-z_2)^{\ell-r}=0$, where $a$ is fixed, and the other restrictions reduce to
			$$(-1)^\ell z_1^rz_2^{\ell-r}=\const\in\C^*,\quad rz_2+(\ell-r)z_1=0,$$
			which yields $\ell$ solutions $(z_1,z_2)$.
			Note that in case $\ell$ even and ${r=\ell/2}$, we obtain $\ell/2$ curves; however, since they glue up to different fragments, we count here curves with ordered points $z_1,z_2$.
			The other limit curves in the modified fragment are then determined uniquely.
			The two additional modification data of the unmarked edges in the modified fragment can be chosen in $r(\ell-r)$ ways.
			So, in total, we obtain $\ell r(\ell-r)$ admissible combinations of admissible modified limit curves in this situation.
			Similarly, in the case of one end incident to $q^{(2)}_j$ (Figure \ref{fcn2}(f,i)) we get $\ell(\ell-1)$ combinations, and in the case of two ends incident to $q^{(2)}_j$ (Figure \ref{fcn2}(e,h)) we get a unique combination.
			Note that the latter answer is twice less than what is predicted by the preceding answers, and this is a consequence of the $\Z_2$-symmetry of the fragments shown in Figure \ref{fcn2}(e,h).
		\end{enumerate}
	\end{proof}

	\begin{lemma}\label{lemma:correspondence_final_step}
		Let $(\Gamma, h, \bq; \{ \varphi_V \}_{V\in \Gamma^0})$ be an enhanced plane tropical curve satisfying the point constraints.
		For any coherent modification data associated to $\Gamma$ exist a unique curve  $[\bn:(\widehat C,\bp)\to\Tor_\K(P)]\in{\mathcal M}_{0,(n_1,n_2)}(P,\bw)$ whose modified tropical limit (in the sense of Construction \ref{con:modified_tropical_limit}) has this modification data.
	\end{lemma}

	\begin{proof}
		We will use \cite[Theorems 2.15 and 3.1]{ST1} combined with \cite[Theorem 2.4]{Sh06a}\footnote{In principle, one can derive the required deformation statement from \cite[Theorem 6.2]{Ty} or \cite[Theorem 78]{Ni}.}.

Here we work with parameterized curves, while \cite[Theorem 2.4]{Sh06a} and \cite[Theorem 2.15]{ST1} are stated for embedded curves. However, both settings can naturally be translated into each other. Indeed, the modifications performed in the construction of the modified tropical limit can be interpreted as the family of parameterized curves
$$\begin{tikzcd}[column sep=normal]
(\widehat C,\bw) \ar[r, hook] \ar[d] & (\widetilde{\mathfrak X},\bw) \ar[d, "\rho"]\\ 
(\C,0) \ar[equal]{r} & (\C,0) 
\end{tikzcd}$$
such that the three-fold $\widetilde{\mathfrak X}$ is birational to ${\mathfrak X}$, the fibers $\widetilde{\mathfrak X}_t\simeq{\mathfrak X}_t$ are isomorphic for $t\ne0$, the central surface $\widetilde{\mathfrak X}_0$ contains extra components coming from toric surfaces associated with the polygons that appear in modifications (cf. Figure \ref{fcn2}(b,d,k,l,m)). In turn, the central embedded curve $C_0\subset\widetilde{\mathfrak X}_0$ is the union of limit curves (including the modified ones) which are immersed rational curves, and hence the parametrization ${\widehat C_0\to C_0}$ is easily recovered.

Consider the line bundle ${\mathcal L}$ on $\widetilde{\mathfrak X}$, or, equivalently, a family of line bundles ${\mathcal L}_t\to\widetilde{\mathfrak X}_t$, $t\in(\C,0)$, such that for $t\ne0$, ${\mathcal L}_t\to\widetilde{\mathfrak X}_t\simeq\Tor(\Delta)$ is the tautological line bundle associated with $\Delta$. It can happen that $h^0(\widetilde{\mathfrak X}_0,{\mathcal L}_0)>h^0(\widetilde{\mathfrak X}_t,{\mathcal L}_t)$, $t\ne0$, and the difference is the number of modifications performed along edges of weight $>1$ that do not contain marked points. Following \cite[Section 2.5]{ST1}, we introduce the line bundle $\rho_*{\mathcal L}\to(\C,0)$ and notice that the natural map $\rho_*{\mathcal L}\otimes k(t)\to H^0(\widetilde{\mathfrak X}_t,{\mathcal L}_t)$ is an isomorphism for $t\ne0$ and is injective for $t=0$. The difference between $H^0(\widetilde{\mathfrak X}_0,{\mathcal L}_0)$ and $\Ima(\rho_*{\mathcal L}\otimes k(0))$ comes from the components $\Tor(\conv\{(-1,0),(1,0),(0,\ell)\})\simeq\widetilde{\mathfrak X}_{0,i}\subset\widetilde{\mathfrak X}_0$ coming in modifications along unmarked edges of weight $\ell>1$. More precisely, ${\Ima(\rho_*{\mathcal L}\otimes k(0))\big|_{\widetilde{\mathfrak X}_{0,i}}}$ has codimension $1$ in $H^0(\widetilde{\mathfrak X}_{0,i},{\mathcal L}_0|_{\widetilde{\mathfrak X}_{0,i}})$ and is determined by a linear relation on the coefficients of the monomials $1,y,...,y^{\ell-1}$. We claim that the relation is just the vanishing of the coefficient of $y^{\ell-1}$. Indeed, we choose an admissible modified limit curve in any of the $\ell$ orbits of the action by shifts $y\mapsto y+a$ (see \cite[Lemma 3.9]{Sh06}). However, we know that there is only one choice in each orbit what we know from the tropical count of nodal curves in toric surfaces subject to usual point constraints \cite{Mi0,Sh06} and from \cite[Theorem 2.18]{ST1} which would yield a separate curve in the above count for any choice of the modified limit curve in  $\Ima(\rho_*{\mathcal L}\otimes k(0))\big|_{\widetilde{\mathfrak X}_{0,i}}$. That is, the linear equation of $\Ima(\rho_*{\mathcal L}\otimes k(0))\big|_{\widetilde{\mathfrak X}_{0,i}}$ is as claimed. In particular, by the construction of modified limit curves we get condition T1 required in \cite[Theorem 2.15]{ST1}.

The remaining conditions T2 and T3 of \cite[Theorem 2.15]{ST1} mean the vanishing of obstructions to the existence of the deformation of $C_0$ into a family of rational curves $C_t\subset\widetilde{\mathfrak X}_t$, $t\in(\C,0)$, matching the constraint $\bw$. Let $(\widehat\Gamma,\widehat\bq)$ be the marked tropical curve obtained from $(\Gamma,\bq)$ after performing all modifications. There is a natural projection $(\widehat\Gamma,\widehat\bq)\to(\Gamma,\bq)$. The components of the complement to the set of vertices and the set of marked points $\bq$ in $\Gamma$ can be oriented so that each marked point is a source, for each unmarked trivalent vertex, two incident edges are incoming and the third incident edge is outgoing. Then we perform modifications and each time define an orientation of the modified fragment assuming that the marked points (if any) are sources, the edges incident to the previously existed vertices are oriented towards the modified fragment, and for the four-valent vertex like in Figure \ref{fcn2}(i,j), the two parallel incident edges are outgoing (while the two other incident edges are incoming). This orientation has no oriented cycles, and hence induces a linear order on the components of $\widehat C_0$. Then conditions T2 and T3 reduce to the sequence of the following transversality statements:
\begin{itemize}\item
If $\lambda:\PP^1\to\widetilde{\mathfrak X}_{0,i}$ is a component of an admissible collection of limit curves (with $\widetilde{\mathfrak X}_{0,i}$ being a component of $\widetilde{\mathfrak X}_0$), then
$$H^1(\PP^1,\lambda^*{\mathcal L}_0\big|_{\widetilde{\mathfrak X}_{0,i}}\otimes {\mathcal J})=0,$$ 
where ${\mathcal J}$ is the ideal sheaf of the points belonging to the preceding components of $\widehat C_0$ and of the constraints $\bw^{(0)}$ that may appear on the considered component of $C_0$.
\item The transversal intersection of $H^0(\PP^1,\lambda^*{\mathcal L}_0\big|_{\widetilde{\mathfrak X}_{0,i}})$ and \linebreak ${\Ima(\rho_*{\mathcal L}\otimes k(0)\big|_{\widetilde{\mathfrak X}_{0,i}})}$ in $H^0(\widetilde{\mathfrak X}_{0,i},{\mathcal L}_0\big|_{\widetilde{\mathfrak X}_{0,i}})$.
\end{itemize} Both the transversality statements follow from the construction of limit curves which we subsequently obtain as simple solutions of certain algebraic equations.

\end{proof}

\section{Tropical formula for characteristic numbers: arbitrary genus}\label{scn4}

In this section we consider the case of arbitrary genus $g$.
We will prove a correspondence theorem between the tropical count and characteristic numbers under the assumptions of h-transversal polygon and points in Mikhalkin position (see assumptions (A1) and (A2) below).

Suppose that $g$ is arbitrary nonnegative number and accordingly
\begin{equation*}n=n_1+2n_2=|\Delta|+g-1.\end{equation*}

\begin{lemma}\label{lemma:pos_genus_bivalent_width_larger_1}
	Let $(\Gamma,h,\bg, \bq)$ be the tropical part of the parameterized tropical limit of a curve $[\bn:\widehat C\to\Tor_\K(P),\bp]\in{\mathcal M}_{g,(n_1,n_2)}(P,\bw)$, and let $V\in \Gamma^{(0)}$ be a marked bivalent vertex of positive genus $\bg(V)>0$.
	Let $V'$ be the closest neighboring vertex to $V$, and let $E\in \Gamma^{(1)}$ be the edge connecting $V$ and $V'$.
	Then $\wt(E)>1$ and the Mikhalkin multiplicity $\mu(V')$ (i.e. the lattice area of the corresponding cell in the dual subdivision) satisfy $\mu(V')>\wt(E)$.
\end{lemma}

\begin{proof}
	Assume without loss of generality that $E$ is horizontal.
	If ${\wt(E)=1}$, then the $x$ coordinate of the limit curve $\varphi_V: C_V \dashrightarrow (\C^*)^2$ induces an isomorphism between $C_V$ and $\PP^1$, in contradiction with $\bg(V)>0$.

	Now, assume towards a contradiction that $\mu(V') = \wt(E)$, and suppose that $V=q^{(2)}_j$.
	Performing a modification $y' = y - \zeta_j^{(2)}$ we get necessarily the fragment depicted in Figure \ref{fig:mod_bivalent_vertex}.
	From our assumptions it follows that the marked vertical end is of weight $1$, meaning that the marked bivalent vertex in the modified fragment has genus $0$.
	Furthermore, the trivalent vertex $\tilde{V}$ corresponds to the triangle $\conv\{(-1,0), (0,0), (0,\wt(E))\}$ in the dual subdivision.
	Since this triangle has no interior points, $\bg({\tilde{V}}) = 0$, which is a contradiction.
\end{proof}

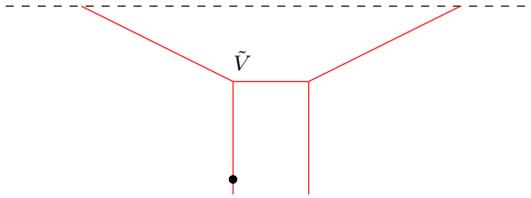
\begin{figure}
	\begin{tikzpicture}
		\draw [dashed] (-1,0) -- (6,0);
		\draw [red] (0,0) -- (2, -1);
		\draw [red] (2, -1) -- (2, -2.5);
		\draw [red] (2, -1) -- node [black, above, very near start] {\tiny $\tilde{V}$} (3, -1);
		\draw [red] (3, -1) -- (5, 0);
		\draw [red] (3, -1) -- (3, -2.5);
		\draw [black, fill=black] (2, -2.3) circle (0.05);
	\end{tikzpicture}
	\caption{A modification of a marked bivalent vertex as appearing in the proof of Lemma \ref{lemma:pos_genus_bivalent_width_larger_1}.}\label{fig:mod_bivalent_vertex}
\end{figure}

\begin{lemma}\label{lemma:flat_collinear_cycle_in_tropicalization_is_well_spaced}
	Let $(\Gamma,h,\bg, \bq)$ be the tropical part of the parameterized tropical limit of a curve $[\bn:\widehat C\to\Tor_\K(P),\bp]\in{\mathcal M}_{g,(n_1,n_2)}(P,\bw)$, and let $V_1, V_2$ be $2$ collinear vertices joined by a collinear cycle.
	Then this collinear cycle is centrally embedded (see Definition \ref{def:center_collinear_cycle}) and exactly one of its endpoints is marked.
\end{lemma}

\begin{proof}
	The fact that exactly one of the endpoints is marked follows from the fact that $\Gamma$ is regular.
	Now, let $V_1, V_2$ be collinear vertices joined by a collinear cycle, assume that the edges of the cycle are horizontal, and suppose that $V_1=q_j^{(2)}$ is marked.
	Performing the modification $y' = y - \zeta_j^{(2)}$ we get the fragment depicted in Figure \ref{fig:high_genus_corr_theorem}(c).
	It is clear that the edges adjacent to the collinear cycle have to be of the same length.
\end{proof}

\begin{lemma}\label{lemma:coll_cylce_in_tropicalization_width_larger_than_1}
	Let $(\Gamma,h,\bg, \bq)$ be the tropical part of the parameterized tropical limit of a curve $[\bn:\widehat C\to\Tor_\K(P),\bp]\in{\mathcal M}_{g,(n_1,n_2)}(P,\bw)$, and let $V_1\in \Gamma^{(0)}$ be a marked collinear vertex joined to $4$-valent vertex $V_2$ by a collinear cycle with edges $E_1, E_2$.
	Then $\mu(V_2)> \wt(E_1)+\wt(E_2)$, where we denote by $\mu(V_2)$ the lattice area of the triangle in the dual subdivision corresponding to $V_2$.
\end{lemma}

\begin{proof}
	Denote by $q_j^{(2)}$ the marking on $V_1$.
	We assume without loss of generality that $E_1, E_2$ are horizontal.
	Denote by $\nu:= \frac{\mu(V_2)}{\wt(E_1)+\wt(E_2)}$.
	We perform the modification $y' = y - \zeta_j^{(2)}$.
	Then the modified fragment is as depicted in Figure \ref{fig:mod_collinear_cycle_noncentral}.
	The two vertices $\tilde{V}$ and $\tilde{V}'$ are mapped to the same point, and the triangle in the dual subdivision that corresponds to both of them is $P_{mod}:=\conv\{(0,0), (\nu,0), (0, \wt(E_1)+\wt(E_2))\}$.
	Since both $\overline{\varphi_{\tilde{V}}(C_{\tilde{V}})}$ and $\overline{\varphi_{\tilde{V}'}(C_{\tilde{V}'})}$ intersect $\Tor_\C(\conv\{(0,0), (\nu, 0)\})\subset \Tor_\C(P_{mod})$, we get that $\nu>1$.
\end{proof}

\begin{figure}
	\begin{tikzpicture}
		\def\linesoffset{0.03}
		\draw [dashed] (-1,0) -- (5,0);
		\draw [red] (0,0) -- (2, -1);
		\draw [red] (2, -1) -- (2, -2.5);
		\draw [red] (2, -1 - \linesoffset) -- (3 + \linesoffset,-1 - \linesoffset);
		\draw [red] (2, -1 + \linesoffset) -- (3 - \linesoffset, -1 + \linesoffset);
		\draw [red] (3  -\linesoffset, -1 + \linesoffset) -- (3 - \linesoffset, -2.5);
		\draw [red] (3 + \linesoffset, -1 - \linesoffset) -- (3 + \linesoffset, -2.5);
		\draw [red] (3  - \linesoffset, -1 + \linesoffset) -- node [ above, very near start, black] {\tiny $\tilde{V}$} (4 - 1.41*\linesoffset, 0);
		\draw [red] (3 + \linesoffset, -1 - \linesoffset) -- node [right, very near start, black] {\tiny $\tilde{V}'$} (4 + 1.41*\linesoffset, 0);
		
		\draw [black, fill=black] (2, -2.3) circle (0.05);		
	\end{tikzpicture}
	\caption{A modification of a collinear cycle as appearing in the proof of Lemma \ref{lemma:coll_cylce_in_tropicalization_width_larger_than_1}.}\label{fig:mod_collinear_cycle_noncentral}
\end{figure}
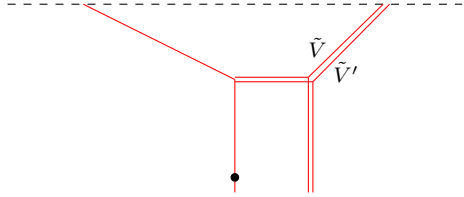

We proceed further under the two additional assumptions
\begin{enumerate}\item[(A1)] the polygon $P$ is $h$-transversal, which in terms of $\Delta$ means that each vector $\ba\in\Delta$ satisfies
	$$\text{either}\ \ba=(\pm1,0),\quad\text{or}\ \ba=(k,\pm1),\ k\in\Z.$$
	The corresponding class of toric surfaces includes the most interesting cases: the plane, the plane blown-up at $1$, $2$, or $3$ points, the quadric, all Hirzebruch surfaces etc. For such a polygon $P$, we introduce its width $\wid(P)$, the length of its horizontal projection.
	\item[(A2)] The sequence of points $\bu=\{u_1,...,u_n\}$ (where $n=n_1+n_2$) is as follows:
	$$u_i=(M_i,\eps M_i),\quad\eps=\const\ne0,\ |\eps|\ll1,\quad 0<M_i\ll M_{i+1}\ \forall\ i=1,...,n-1.$$
	Such a configuration sometimes is called a {\it configuration in Mikhalkin position} as it was introduced in \cite[Section 8.5]{Mi}. Note that whenever $|\eps|>0$ is small enough the linear functional $\varphi(x,y)=x+\eps y$ is injective on the set $P\cap\Z^2$.
\end{enumerate}

\begin{definition}\label{def:floors}
	Let $(\Gamma, h, \bg, \bq)$ be a plane tropical curve of degree $\Delta$ and genus $g$ passing through points $\bu$ in Mikhalkin position (i.e. assumptions (A1) and (A2) hold).
	An edge $E\in \Gamma^{(1)}$ is called horizontal if so is $h(E)$.
	A connected component of the union of the non horizontal edges is called a \emph{floor}.
	For a detailed treatment of floor decomposition of planar tropical curves passing through points in Mikhalkin position see \cite{BrM}.
\end{definition}

\begin{lemma}\label{lemma:trop_curve_through_mikh_pos}
	Let $(\Gamma, h, \bq)$ be a plane tropical curve of degree $\Delta$ and genus at most $g$ passing through points $\bu$ in Mikhalkin position (i.e. assumptions (A1) and (A2) hold), and having $\bq_{(2)}\subseteq \Gamma^{(0)}$.
	Then the following hold:
	\begin{enumerate}
		\item
		The genus of $(\Gamma, h, \bq)$ is exactly $g$.
		
		\item
		All the vectors $D\left(h\big|_E\right)=(a_x,a_y)$, where $E\in\Gamma^1$, satisfy \linebreak ${a_y\in\{0,\pm1\}}$.
		
		\item
		Each floor is isometric to $\R$, contains exactly one marked point, and is combined of edges of weight $1$.
		
		\item
		Every horizontal edge contains exactly one marked point, either in the interior, or at an end.
		
		\item
		$h^{-1}(\{y=\eps x\})=\bq$.
	\end{enumerate}
\end{lemma}

\begin{proof}
	By Lemma \ref{lemma:generic_trop_curve_is_regular} the set of points $\tilde\bu\in \R^{2n}$ that have either a non regular tropical curve passing through them or a curve of genus smaller than $g$ passing through them (of degree $\Delta$) is of positive codimension and thus contained in a codimension $1$ finite polyhedral complex $\Sigma\subset\R^{2n}$.
	Since $\Sigma$ is finite, taking $|\eps|$ small enough ensures that the line in $\R^{2n}$ resulting from varying $M_n$ is not contained in $\Sigma$.
	Thus we can take $M_n$ big enough such that all the curves of degree $\Delta$ and genus at most $g$ passing through $\bu$ are regular and of genus exactly $g$.

	From the condition $|\varepsilon|\ll 1$ we can deduce that there is no vertex $V\in \Gamma^{(0)}$ with two incoming (relative to the regular orientation on $\Gamma\setminus\bq$) edges $E_1,E_2$ such that the $y$ coordinate of $ D(h|_{E_1}) $ is positive, while the $y$ coordinate of $ D(h|_{E_2})$ is negative.
	Suppose, towards a contradiction, that there is an edge $E\in \Gamma^{(1)}$ with $D(h|_{E})=(a_x, a_y)$ where $|a_y|>1$.
	Let $E_1=E, E_2, \ldots, E_k$ be the path following the regular orientation from $E$ towards an end, i.e. $E_{i+1}$ is the unique edge pointing away from the endpoint of $E_i$ (for every $1\le i < k$), and $E_k$ is an end.
	Then by induction on $i$ we get that the $y$ coordinate of $ D(h|_{E_i})$ is larger, in absolute value, than $1$, which contradicts our assumption on $\Delta$.
	This proves $(2)$.
	
	An existence of a vertex with $2$ incoming non horizontal edges would imply, by the balancing condition, an existence of an edge $E$ with $y$ coordinate of $ D(h|_E)$ bigger, in absolute value, than $1$, which readily implies $(3)$.
	
	Finally, assume, towards a contradiction, that $E\in \Gamma^{(1)}$ is a horizontal edge without a marking.
	Then the regular orientation on $E$ is pointing away from one of its ends, but it would necessarily mean that the floor containing this end has to contain at least $2$ markings.
	Thus every horizontal edge has at least one marking, and regularity implies it has exactly one marking.
	
	Item $(5)$ follows immediately from the previous items and the condition $|\varepsilon|\ll 1$.
\end{proof}

\begin{lemma}\label{lemma:pos_genus_tropicalization_mikh_poss}
	Under assumptions (A1) and (A2), let $(\Gamma,h,\bg, \bq)$ be the tropical part of the parameterized tropical limit of a curve \linebreak ${[\bn:\widehat C\to\Tor_\K(P),\bp]\in{\mathcal M}_{g,(n_1,n_2)}(P,\bw)}$.
	Then $\bg(V)=0$ for all vertices $V\in\Gamma^0$, and the curve $(\Gamma,h,\bg,\bq)$ is realizable (see Definition \ref{def:realizable_trop_curve}) trivalent, regular plane tropical curve of degree $\Delta$ and genus $g$ such that $h(\bq)=\bu$.
	In particular, all its collinear cycles are centrally embedded (see Definition \ref{def:center_collinear_cycle}).
\end{lemma}

\begin{proof}
	As shown in Lemma \ref{lemma:genus_0_tropicalization_defined}, a point $q\in\bq^{(2)}$ cannot be neither outside $\Gamma^0$, nor a rational bivalent vertex.
	Suppose $\Gamma$ has $k$ bivalent vertices $V_1, \dots, V_k$ with corresponding markings $q(V_1), \dots, q(V_k)\in \bq^{(2)}$.
	Construct a plane tropical curve $(\Gamma', h', \bg', \bq')$ by, for every $1\le i \le k$, removing $V_i$, replacing the $2$ edges adjacent to $V_i$ by a unique edge with the same image under $h'$, and setting $q(V_i)\in \bq'^{(1)}$.
	Then $(\Gamma', h', \bq')$ is a plane tropical curve of degree $\Delta$ and genus at most $g-k$ passing through the same points in Mikhalkin position with $|\bq'^{(1)}|+2|\bq'^{(2)}|=|\bq^{(1)}|+2|\bq^{(2)}|-k$.
	Thus $(\Gamma', h', \bq')$ satisfies the assumptions of Lemma \ref{lemma:trop_curve_through_mikh_pos}, so $\gen(\Gamma')=g-k$, meaning that $\bg'(V)=0$ for all $V\in \Gamma^{(0)}$.
	In addition, all the non horizontal edges of $\Gamma'$ are of weight $1$ and so $V_1,\dots,V_k$ lie on horizontal edges by Lemma \ref{lemma:pos_genus_bivalent_width_larger_1}.
	Moreover, since the $y$ coordinate of $D(h|_E)$ is in $\{0, \pm 1\}$ for all $E\in \Gamma^{(1)}$, if $E_i\in \Gamma'^{(1)}$ is the horizontal edge containing $V_i$ and $V'$ is the closest endpoint of $E_i$ to $V_i$, $\mu(V')=\wt(E_i)$ in contradiction to the conclusion of Lemma \ref{lemma:pos_genus_bivalent_width_larger_1}.
	Thus $k=0$, meaning there is no bivalent vertices in $\Gamma$, and $\bg(V)=0$ for all $V\in \Gamma^{(0)}$.
	The remaining of the statement of the lemma follows from Lemmas \ref{lemma:coll_cylce_in_tropicalization_width_larger_than_1} and \ref{lemma:dim_moduli_cells}.
\end{proof}

It follows from Lemma \ref{lemma:pos_genus_tropicalization_mikh_poss} that the set ${\mathcal M}^{trop}_{g,(n_1,n_2)}(P,\bu)$ of plane tropical curves of genus $g$, degree $\Delta$, and passing through $\bu$ with $\bq_{(2)}\subseteq \Gamma^{(0)}$ is finite. For each curve $(\Gamma,h,\bq)\in{\mathcal M}^{trop}_{g,(n_1,n_2)}(P,\bu)$, we define

\begin{equation}\label{eq:def_trop_cn}
	CN^{trop}_g(\Gamma,h,\bg,\bq)=\frac{1}{|\Aut(\Gamma,h,\bg,\bq)|}
\prod_{\renewcommand{\arraystretch}{0.6}
	\begin{array}{c}
		\scriptstyle{V\in\Gamma^0\setminus\bq}\\
		\scriptstyle{\mu(V)\ne 0}\end{array}}\mu(V)
\cdot\prod_{\{E_1, E_2\}\in \mathcal{C}}\frac{\wt(E_1)\wt(E_2)}{\wt(E_1)+\wt(E_2)}
\end{equation}
where $\mathcal{C}$ denotes the set of pairs of edges that form collinear cycle (which by Lemma \ref{lemma:pos_genus_tropicalization_mikh_poss} has to be centrally embedded).
Note that the vertices we exclude in the first product of \eqref{eq:def_trop_cn} (i.e. those with Mikhalkin multiplicity $0$) are the unmarked endpoints of centrally embedded collinear cycles.

\begin{remark}
	It is clear that if $g=0$ then the second product in \eqref{eq:def_trop_cn} is empty, and the formula \eqref{eq:def_trop_cn} coincides with the definition of the tropical count $CN^{trop}_0(\Gamma,h,\bq)$ as given in \eqref{eq:trop_cn_def_genus0}.
\end{remark}

\begin{theorem}\label{tcn3}
	Under the assumptions (A1) and (A2), we have
	\begin{equation}CN_g(P,(n_1,n_2))=\sum_{(\Gamma,h,\bg,\bq)\in{\mathcal M}^{trop}_{g,(n_1,n_2)}(P,\bu)}CN^{trop}_g(\Gamma,h,\bg,\bq).
		\label{ecn23}\end{equation}
\end{theorem}

\begin{proof}
	As in the proof of Theorem \ref{tcn1} we apply the techniques of tropical and algebraic modifications to recover the tropical limits $(\widehat\Gamma,\widehat h,\widehat\bg,\widehat\bq)$ and the corresponding collections of limit curves induced by a curve \linebreak ${[\bn:\widehat C\to\Tor_\K(P),\bp]\in{\mathcal M}_{g,(n_1,n_2)}(P,\bw)}$. Comparing with the rational case, we have to analyze just one type of fragments: a horizontal centrally embedded collinear cycle with edges $E_1$, $E_2$ of weights $r, \ell-r$ (see Figure \ref{fig:high_genus_corr_theorem}(a)).
	Without loss of generality we assume that $h(E)$ does not intersect with other edges of $h_*\Gamma$. The corresponding fragment of the dual subdivision $\Sigma$ of $P$ looks as shown in Figure \ref{fig:high_genus_corr_theorem}(b) (cf. Lemma \ref{lemma:trop_curve_through_mikh_pos}(2)). We can suppose that $d_j=0$ in formula (\ref{ecn15}), that is, ${w^{(2)}_j=(t^{c_j}(\eta_j+O(t^{>0})),\zeta_j+O(t^{>0}))}$, $\eta_j,\zeta_j\in\C^*$. Then the truncation of the polynomial (\ref{e-new60}) defining the image of the given curve in $\Tor_\K(P)$ is $(a(y-\zeta_j)^\ell+O(t^{>0})$, $a\in\C^*$.
	The tropical modification of that fragment along the edge $E$ can be constructed similarly to item $(3)$ of Construction \ref{con:modified_tropical_limit}: the result is presented in Figure \ref{fig:high_genus_corr_theorem}(c) (the dual subdivision is shown in Figure \ref{fig:high_genus_corr_theorem}(d)). Note that the length of the cycle is uniquely determined by the fact that the distances from the vertices of the cycle to the neighboring non-flat trivalent vertices are equal to each other. For the given weights $r\le \ell-r$ of the horizontal edges in the modification, the modified limit curve associated with the triangle $\conv\{(-1,0),(0,0),(0,\ell)\}$ is as described in item (3) of Construction \ref{con:modified_tropical_limit}, and item (3) in Definition \ref{def:modification_data}.
	In particular, the choice of such a curve matching the imposed conditions can be done in $\ell$ ways if $r<\ell-r$, or in $\ell/2$ ways if $r=\ell/2$. The modified limit curve associated with the triangle $\conv\{(0,0),(1,0),(0,\ell)\}$ is then defined uniquely by the previous choice.
	The further modification along the horizontal edges of the fragment shown in Figure \ref{fig:high_genus_corr_theorem}(c) yields two more modified limit curves constructed in item (1) of Construction \ref{con:modified_tropical_limit} (see the details in \cite[Sections 3.5 and 3.6]{Sh05} or \cite[Section 2.2(2)]{GS21}). This pair of curves can be chosen in $r(\ell-r)$ ways.
	In total we get $r\cdot \ell \cdot (\ell-r)\cdot \frac{1}{|Aut(E_1, E_2)}$ options, which is precisely the product of weights of the collinear cycle and the neighboring  trivalent vertices.
	It follows that the number of admissible collections of limit curves which can be associated with a tropical curve
	$(\Gamma,h,\bg,\bq)\in{\mathcal M}^{trop}_{g,(n_1,n_2)}(P,\bu)$ equals $CN^{trop}_g(\Gamma,h,\bg,\bq)$.

To complete the proof of Theorem \ref{tcn3}, we need to show that each admissible collection of limit curves admits a unique deformation in a family of curves of genus $g$ in $\Tor(\Delta)$. To this end we apply Lemma \ref{lemma:correspondence_final_step}, which, in fact, does not depend on the genus of the considered curves. The only new element is the flat cycle and its (partial) modification as shown in Figure \ref{fig:high_genus_corr_theorem}, and it can be treated in the same way as the modifications considered in the proof of Theorem \ref{tcn1}.
\end{proof}

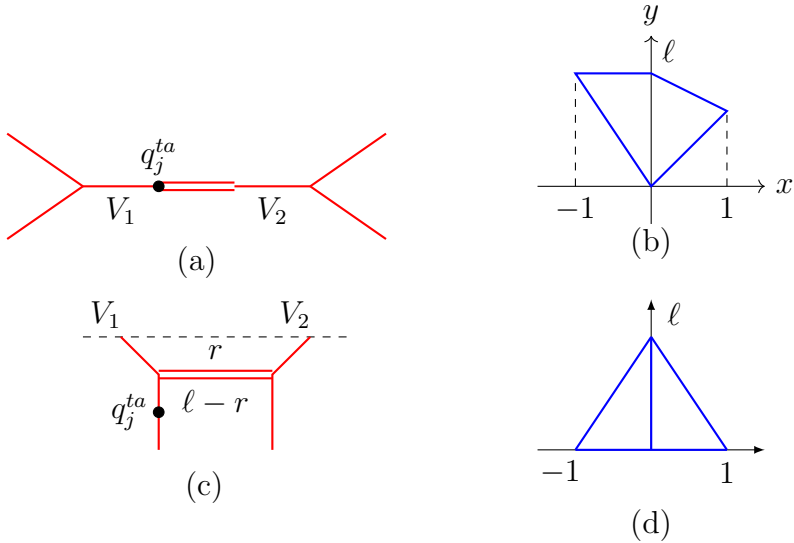
\begin{figure}
	\begin{tikzpicture}
		\tikzset{
			vector/.style={->, >=latex},
			dashed line/.style={dashed, line width=0.4pt},
			dotted line/.style={dotted, line width=0.4pt},
			thick line/.style={line width=1pt},
		}

		\begin{scope}[xshift=1cm,yshift=2cm,scale=1]
			\draw[red, thick] (-2,0)--(-1,0);
			\draw[red, thick] (0,0)--(1,0);
			\draw[red, thick] (-1,0.05)--(0,0.05);
			\draw[red, thick] (-1,-0.05)--(0,-0.05);
			\draw[red, thick] (-2,0)--(-3,0.7);
			\draw[red, thick] (-2,0)--(-3,-0.7);
			\draw[red, thick] (1,0)--(2,0.7);
			\draw[red, thick] (1,0)--(2,-0.7);
			\filldraw[black] (-1,0) circle (2pt);
			\node at (-1.5,-0) [below] {$V_1$};
			\node at (0.5,-0) [below] {$V_2$};
			\node at (-1,0) [above]{$q_j^{(2)}$};
			\node at (-0.5,-1) {(a)};
		\end{scope}

		\begin{scope}[xshift=6.5cm,yshift=2cm,scale=1]
			\draw[->] (-1.5,0) -- (1.5,0) node[right] {$x$};
			\draw[->] (0,-0.5) -- (0,2) node[above] {$y$};
			
			\draw (-1,0) node[below] {$-1$};
			\draw (1,0) node[below] {$1$};
			\draw [dashed] (-1,0) -- (-1, 1.5);
			\draw [dashed] (1,0) -- (1, 1);
			
			\draw[blue, thick] (-1,1.5) -- (0,1.5) -- (1,1) -- (0,0) -- cycle;
			
			\node[above right] at (0,1.5) {$\ell$};
			
			\node[below] at (0,-0.4) {(b)};
		\end{scope}

		\begin{scope}[xshift=-1cm,yshift=0cm]
			\draw[dashed line] (0,0) -- (3.5,0);
			
			\draw [thick, red] (1,-1.5) -- (1,-0.5) -- (0.5,0);
			\draw [thick, red] (2.5,-1.5) -- (2.5,-0.5) -- (3,0);
			\draw [thick, red] (1,-0.45) -- node [above, black] {$r$} (2.5,-0.45);
			\draw [thick, red] (1,-0.55) -- node [below, black] {$\ell-r$} (2.5,-0.55);

			\node at (0.3,0.3) {$V_1$};
			\node at (2.8,0.3) {$V_2$};

			\filldraw[black] (1,-1) circle (2pt);
			\node at (1,-1) [left] {$q_j^{(2)}$};

			\node at (1.6,-2) {(c)};
		\end{scope}

		\begin{scope}[xshift=6.5cm,yshift=-1.5cm]
			\draw[vector] (-1.5,0) -- (1.5,0);
			\draw[vector] (0,0) -- (0,2);
			\node at (-1.2,-0.3) {$-1$};
			\node at (1,-0.3) {$1$};

			\draw [thick, blue] (-1,0) -- (1,0);
			\draw [thick, blue] (-1,0) -- (0,1.5);
			\draw [thick, blue] (1,0) -- (0,1.5);
			\draw [thick, blue] (0,0) -- (0,1.5);
			
			\node at (0.3,1.8) {$\ell$};

			\node at (0,-1) {(d)};
		\end{scope}
		
		\end{tikzpicture}
	\caption{Proof of Theorem \ref{tcn3}}\label{fig:high_genus_corr_theorem}
\end{figure}

\section{Refined tropical invariants}\label{rbi}

In this section we define the refined tropical invariants, and prove their invariance under the change of the configuration of points in the following situations: either $g\in\{0,1\}$, or $n_2\in\{0,1\}$, or the points are in Mikhalkin position (see assumptions (A1) and (A2) in Section \ref{scn4}).
We also show that the refined tropical invariants we define specialize to the characteristic numbers in these cases when the parameter tends to $1$ (see Corollary \ref{ccn1}).

\subsection{Definition of the refined count}\label{sec:refined_definition}

For an integer $ \mu \in \mbb{Z} $ set
\begin{equation*}
	[\mu]^{-}_y = \frac{y^{\mu/2} - y^{-\mu/2}}{y^{1/2}-y^{-1/2}} \; , \; [\mu]^{+}_y = \frac{y^{\mu/2} + y^{-\mu/2}}{y^{1/2}+y^{-1/2}}
\end{equation*}
where $ y $ is a formal parameter.

\begin{definition}\label{def:separated_curve}
	A realizable marked plane tropical curve $ (\Gamma, h, \bg, \bq) $ is called \emph{simple} if all of its vertices are rational and simple (see Definition \ref{def:simple_vertex}) and of valency at most $4$ (so no two collinear cycles meet).
\end{definition}

\begin{definition}\label{def:refined_broccoli_weight}
	Let $ (\Gamma, h, \bg, \bq) $ be a simple tropical curve of genus $ g $ with $ n_2 $ marked vertices and $ n_1 $ markings that lie on the relative interiors of edges.
	Denote by $ \mathcal{C}_{\text{central}}(\Gamma) $ the set of pairs of edges  of $ \Gamma $ that form a centrally embedded collinear cycle (see Definition \ref{def:center_collinear_cycle}) and by $ \mathcal{C}_{\text{non-central}}(\Gamma) $ the set of triplets  $(E_1, E_2, V)$ where $E_1, E_2$ form a non centrally embedded collinear cycle, and $V$ is the unique $4$-valent vertex adjacent to this collinear cycle.
	Define the \emph{refined 
weight of $ \Gamma $} as
	\begin{align}\label{eq:refined_weight_def}
		\text{GS}_y(\Gamma) = & \frac{1}{\left| \Aut(\Gamma)\right|}\cdot \prod_{V \in \Gamma^0\cap \bq}[\mu(V)]_{y}^{+}\cdot \prod_{\substack{V\in \Gamma^0\setminus\bq \\ \val(V)=3 \\ \mu(V)\ne 0}}[\mu(V)]_{y}^{-} \times \nonumber \\
		& \times  \prod_{(E_1,E_2)\in \mathcal{C}_{\text{central}}(\Gamma)}\varphi^{(0)}_y(\wt(E_1),\wt(E_2)) \times  \nonumber \\
		& \times \prod_{(V, E_1, E_2)\in \mathcal{C}_{\text{non-central}}(\Gamma)}\varphi^{(1)}_y\left(\wt(E_1), \wt(E_2), \frac{\mu(V)}{\wt(E_1)+\wt(E_2)}\right)
	\end{align}
	where $ \mu(V) $ denotes the Mikhalkin weight of the vertex $ V $ (i.e. the lattice area of the corresponding cell of the dual subdivision), and the functions $ \varphi^{(0)} $ and $ \varphi^{(1)} $ are defined as follows:
	
	\begin{equation}\label{eq:def_varphi_0}
		\varphi^{(0)}_y(w_1,w_2) = \frac{[w_1]^{-}_y[w_2]^{-}_y}{[w_1+w_2]^{-}_y},
	\end{equation}
	
	\begin{equation}\label{eq:def_varphi_1}
		\varphi^{(1)}_y(w_1,w_2,\nu) = [w_1\nu]^{-}_y [w_2\nu]^{-}_y - \frac{[w_1]^{-}_y [w_2]^{-}_y}{[w_1+w_2]^{-}_y}[(w_1+w_2)\nu]^{-}_y.
	\end{equation}
\end{definition}

\begin{remark}
	The condition $\mu(V)\ne 0$ in the second product in \eqref{eq:refined_weight_def} is needed to omit the unmarked vertices of centrally embedded collinear cycles.
\end{remark}

\begin{remark}
	Note that any collinear cycle contains a marked vertex with Mikhalkin weight $0$, so a centrally embedded collinear cycle contributes
	\[ [0]^{+}_y \cdot \varphi^{(0)}_y(w(E_1), w(E_2))= \frac{2}{y^{1/2}+y^{-1/2}}\varphi^{(0)}_y(w(E_1), w(E_2)) \]
	and similarly for a non centrally embedded collinear cycle (with $\varphi^{(0)}$ replaced by $\varphi^{(1)}$).
\end{remark}

\begin{remark}\label{rem:summing_over_weights_get_Psi}
	If $g=1$ we recover the refined 
weight as defined in \cite[Section 3.3]{SS}.
	Indeed, if we denote (the factor $\frac{1}{2}$ comes from dividing by the automorphisms of the curve)
	\begin{equation*}
		\Phi^{(0)}_y(m) := \frac{1}{2}\sum_{k=1}^{m-1}[0]^{+}_y\varphi^{(0)}_y(k, m-k) = \frac{(m-1)[m+1]_y^- - (m+1)[m-1]_y^-}{(y^{1/2}+y^{-1/2})(y^{1/2}-y^{-1/2})^2[m]_y^-},
	\end{equation*}
	and
	\begin{align*}
		\Phi^{(1)}_y(m,\nu) := & \frac{1}{2}\sum_{k=1}^{m-1}[0]^{+}_y\varphi_y^{(1)}(k, m-k, \nu) =\\
		= & 2\frac{[\nu]_{y^m}^{-}[m-1]_{y}^{-} - [m-1]_{y^\nu}^{-} - (m-1)[\nu-1]_{y^m}^{-}}{(y^{1/2}+y^{-1/2})(y^{1/2}-y^{-1/2})^2},
	\end{align*}
	then
	\begin{align*}
		\Psi^{(2)}_y\left(m,\nu_1, \nu_2\right) = [m\nu_1]^{-}_y [m\nu_2]^{-}_y\Phi^{(0)}_y(m) &+ [m\nu_1]^{-}_y\Phi^{(1)}_y(m, \nu_2) +\\
		&+ [m\nu_2]^{-}_y\Phi^{(1)}_y(m, \nu_1),
	\end{align*}
	where $ \Psi^{(2)}_y(m,\nu_1,\nu_2) $ is the weight associated to a collinear cycle as defined in \cite[Section 3.3]{SS}.
\end{remark}

We now can state the main results of this section

\begin{theorem}\label{thm:refined_invariant_one_marked}
	If either $ g\in\{0, 1\} $ or $ n_2\in\{0,1\} $, then
	\begin{equation}\label{eq:refined_for_points}
		GS_y(\Delta, g, (n_2, n_1), \bu) := \sum_{\begin{matrix} (\Gamma, h, \bg, \bq)\in \modulispace \\ h(\bq) = \bu	\end{matrix}} GS_y(\Gamma)
	\end{equation}
	does not depend on the choice of a generic $ \bu\in \R^{2n} $.
\end{theorem}

\begin{remark}
	Note that by Lemma $\ref{lemma:dim_moduli_cells}$ all the tropical curves appearing in \eqref{eq:refined_for_points} are simple, so the sum is well defined.
\end{remark}

The second result deals with a general genus and value of $n_2$, but for points in Mikhalkin position.

\begin{theorem}\label{tcn4}
	Let $\Delta\subset\Z^2\setminus\{0\}$ be a nondegenerate primitive balanced multiset, $P\subset\R^2$ the induced lattice polygon. Suppose that $P$ is $h$-transversal (assumption (A1) in Section \ref{scn4}). For a fixed $g\ge 0$ and $n_1\ge0$, $n_2>0$ such that $n_1+2n_2=|\Delta|+g-1$ take a configuration of points $\bu=\bu_{(1)}\cup\bu_{(2)}$ of $n_1+n_2$ points in Mikhalkin position (assumption (A2) in Section \ref{scn4}) on the line $\{y=\eps x\}$, where $0<|\eps|\ll1$, $|\bu_{(1)}|=n_1$, $|\bu_{(2)}|=n_2$.	
	Then $\GS_y(P,g,(n_1,n_2),\bu)$ does not depend neither on the choice of $\eps\ne0$, $|\eps|\ll1$, nor on the choice of the configuration $\bu$ in Mikhalkin position on the line $\{y=\eps x\}$, nor on the replacement of $P$ by the polygon $\overline P$ obtained from $P$ by a reflection with respect to a vertical line.
\end{theorem}

\begin{remark}
	By Lemma \ref{lemma:trop_curve_through_mikh_pos}, the tropical curves appearing in the sum in Theorem \ref{tcn4} are simple, so the sum is well defined.
\end{remark}

\begin{remark}
	For an h-transversal polygon $P$ and points in Mikhalkin position, $GS_y(P,0,(n_1,n_2),\bu)$ can be computed, and in fact even defined, using floor diagrams, see \cite{BJP}.
	Similar approach works in higher genus as well, and was implemented in \cite{Mevel}.
	It is an elementary verification that the refined weight defined in \cite{Mevel} coincides with the one defined in this paper.
\end{remark}

By taking the limit $y\to 1$ we get the weight of the tropical curves as defined in \eqref{eq:trop_cn_def_genus0} and \eqref{eq:def_trop_cn}, and so we immediately get the following corollary of Theorems \ref{thm:refined_invariant_one_marked} and \ref{tcn3}.

\begin{corollary}\label{ccn1}
	\begin{enumerate}
		\item
		For any nondegenerate lattice polygon $P$ and genus $g=0$, we have
		\begin{equation}CN_0(P,(n_1,n_2))=\GS_y(P,0,(n_1,n_2))\big|_{y=1}.\label{ecn26}\end{equation}

		\item
		Under the assumption (A1) and $g=1$, we have
		\begin{equation}CN_1(P,(n_1,n_2))=\GS_y(P,1,(n_1,n_2))\big|_{y=1}.\label{ecn24}\end{equation}

		\item Under assumptions (A1) and (A2), we have
		\begin{equation}
			CN_g(P,(n_1,n_2))=\GS_y(P,g,(n_1,n_2), \bu)\big|_{y=1}.\label{eq:refined_equal_characteristic_general_genus}
		\end{equation}
	\end{enumerate}
\end{corollary}

\begin{remark}\label{rcn3}
	Relation \eqref{ecn24} holds without assumption (A1). The proof, however, requires some extra heavy calculations, and we decided to restrict ourselves to the situation stated above, since it covers all interesting toric surfaces.
\end{remark}

The rest of the section is devoted to the proof of Theorem \ref{thm:refined_invariant_one_marked}. 
Theorem \ref{tcn4} will be proved in Section \ref{pt6.7}.
For the proof of Theorem \ref{thm:refined_invariant_one_marked} we consider $ 2 $ point configurations $ \bu^{(1)}, \bu^{(2)}\in \R^{2n} $ whose preimage under the evaluation map $ \text{Ev}: \modulispace \to \mbb{R}^{2n} $ is zero dimensional and lies in the interior of maximal dimensional cells.
We then connect $ \bu^{(1)} $ and $ \bu^{(2)} $ by a path that do not cross the images (under $ \text{Ev} $) of cells of co-dimension larger than $ 1 $.
The combinatorial types of curves passing through $ \bu $ change only when $ \bu $ cross the images of cells of co-dimension $ 1 $, so we need to prove that $ GS_y(\Delta, g, (n_2, n_1), \bu) $ remain constant when $ \bu $ cross those walls.
We can formulate it precisely in the following lemma.
\begin{lemma}\label{lemma:local_refined_invariance}
	Assume that either $ g\in\{0,1\} $ or $ n_2\in\{0,1\} $.
	Let $ \alpha $ be a combinatorial type of co-dimension $ 1 $, i.e. $ \dim \moduli{\alpha} = |\Delta|+g-2+n_1 $.
	Denote by $ H_{0} $ the hyperplane in $ \mbb{R}^{2n} $ that contains $ \text{Ev}\left(\moduli{\alpha} \right) $, by $ H_{+}, H_{-} $ the two half spaces with common boundary $ H_{0} $, and by $ \mathcal{A}_{+} $ (respectively $ \mathcal{A}_{-} $) the set of regenerations $ \beta $ of $\alpha$ with $ \openmoduli{\alpha}\subseteq \moduli{\beta} $ and $ \text{Ev}\left(\moduli{\beta} \right) \subseteq H_{+}$ (respectively $ \text{Ev}\left(\moduli{\beta} \right) \subseteq H_{-}$).
	Then
	\begin{equation}\label{eq:local_refined_invariance}
		\sum_{\beta\in \mathcal{A}_{+}} GS_y(\beta) = \sum_{\beta \in \mathcal{A}_{-}} GS_y(\beta)
	\end{equation}
	where we denote by $ GS_y(\beta) $ the refined weight of any curve with combinatorial type $ \beta $ (this is well defined since the refined weight depends only on the combinatorial type).
\end{lemma}

We now will describe an outline for the proof of Lemma \ref{lemma:local_refined_invariance}.
For every codimension $1$ combinatorial type $\alpha$ as classified in Lemma \ref{lemma:codim1_degenerations_list}, we enumerate all its regenerations $\beta$.
We then check for each such regeneration $\beta$ if its image under $\text{Ev}$ lies inside $H_{+}$ or $H_{-}$, i.e. if $\beta\in\mathcal{A}_{+}$ or $\beta\in\mathcal{A}_{-}$ in the notation of Lemma \ref{lemma:local_refined_invariance}.
Finally we sum the refined weights of combinatorial types in $\mathcal{A}_{+}$ and the refined weights of the combinatorial types in $\mathcal{A}_{-}$ and compare them.
We note that in all cases there is a small fragment which varies between the different regenerations while the rest of the curve remains the same (and identical to the corresponding part in $\alpha$), so we only need to compare the contributions to the refined weight of this small fragment.
This is the content of Sections \ref{sec:refined_invariance_proof_regular_4valent}-\ref{sec:refined_invariance_proof_smaller_genus}.

\begin{remark}\label{remark:exhaustiveness_of_regenerations}
	We briefly explain the argument, used throughout Sections \ref{sec:refined_invariance_proof_regular_4valent}--\ref{sec:refined_invariance_proof_smaller_genus}, that the listed regenerations in each case are exhaustive.
	A regeneration of a codimension $1$ combinatorial type $\alpha$ is obtained by resolving each vertex of valency $n>3$ in the codimension $1$ fragment into a tree of trivalent vertices connected by $n-3$ new bounded edges.
	For each such vertex, the number of topologically distinct resolutions is finite: one must distribute the $n$ original half-edges among the vertices of a tree so that every vertex has valency at most $3$.
	Once the tree topology is fixed, the \emph{direction} of each new edge is uniquely determined by the balancing condition at every new vertex.
	The candidates are then filtered by the following necessary conditions for a regeneration to lie in the codimension $0$ stratum adjacent to $\alpha$:
	\begin{enumerate}[label=(\roman*)]
		\item the regenerated curve must be regular and simple, and
		\item the moduli space of the regenerated combinatorial type must have the expected dimension $|\Delta|+g-1+n_1$.
	\end{enumerate}
	This eliminates most of the possible candidates and leaves the regenerations listed in the corresponding figures.
\end{remark}

In the proof of Lemma \ref{lemma:local_refined_invariance} we will use the following evident properties of the symbols $ [\mu]^-_y, [\mu]^+_y $.

\begin{lemma}\label{lemma:algebra_of_quantum_weights}
	For integers $ a,b\in \mbb{Z} $ the following hold:
	\begin{enumerate}
		\item $[-a]_y^+ = [a]_y^+$.
		
		\item $[-a]_y^- = -[a]_y^-$.
		
		\item $ [a]^-_y[b]^-_y = \frac{y^{1/2}+y^{-1/2}}{(y^{1/2}-y^{-1/2})^2}\cdot\left( [a+b]^+_y - [a-b]^+_y \right).  $
		
		\item $ [a]^-_y[b]^+_y = \frac{1}{y^{1/2}+y^{-1/2}}\cdot\left( [a+b]^-_y + [a-b]^-_y \right).  $
		
		\item $ [a]^+_y[b]^+_y = \frac{1}{y^{1/2}+y^{-1/2}}\cdot\left( [a+b]^+_y + [a-b]^+_y \right).  $
		
		\item $ [a]^-_y[b]^+_y \pm [a]^+_y[b]^-_y = \frac{2}{y^{1/2}+y^{-1/2}}[a\pm b]^-_y. $
		
		\item Let $ \langle a_k \rangle_{k=1}^n\subset \mbb{Z} $ is a finite arithmetic progression of integers with difference $ d $. Suppose that $ \langle a_k \rangle_{k=1}^n $ is symmetric around $ 0 $ (i.e. $ a_i=-a_{n+1-i} $ for all $ i=1,\dots,n $) then
		\[ \sum_{k=1}^n[a_k]^{-}_{y} = 0 \]
		and
		\[ \sum_{k=1}^n[a_k]^+_y = \frac{2}{y^{1/2}+y^{-1/2}}[n]^-_{y^{d/2}}. \]
	\end{enumerate}
\end{lemma}

\subsection{Regular curve with four-valent vertices}\label{sec:refined_invariance_proof_regular_4valent}
We begin with case $(1)$ in Lemma \ref{lemma:codim1_degenerations_list}.

\begin{lemma}\label{lemma:invariance_regular_4_valent}
	Let $\alpha$ be a combinatorial type of a regular tropical curve all of whose vertices are at most $4$-valent, and all its marked vertices are trivalent.
	Suppose that $\alpha$ is of co-dimension $1$, i.e. \linebreak ${\dim \openmoduli{\alpha} = |\Delta|+g-2+n_1}$.
	Then Lemma \ref{lemma:local_refined_invariance} holds for $\alpha$.
\end{lemma}

\begin{proof}
	Let $(\Gamma, h, \bg, \bq)$ be a tropical curve with $[\Gamma]=\alpha$ and let $V_1, \dots, V_k$ be the $4$-valent vertices of $\Gamma$ which are not part of a collinear cycle.
	Let $\bs{x}\in \text{Ev}(\openmoduli{\alpha})$, and let $\bs{w}\in\R^{2n}$ be a vector with $ \bs{x}+\varepsilon\bs{w}\in H^+$ for small enough positive $\varepsilon$.
	We need to show that
	\[ \sum_{\Gamma'\in Ev^{-1}(\bs{x}+\varepsilon\cdot \bs{w})}\GS_y(\Gamma') = \sum_{\Gamma'\in Ev^{-1}(\bs{x}-\varepsilon\cdot \bs{w})}\GS_y(\Gamma'). \]
	
	Let $V_i$ be a $4$-valent vertex of $\Gamma$ with edges $E_{i,1}, E_{i,2}, E_{i,3}$, and $E_{i,4}$ such that $E_{i,1}, E_{i,2}$, and $E_{i,3}$ are oriented towards $V_i$ in the regular orientation of $\Gamma$, while $E_{i,4}$ is oriented away from $V_i$.
	Suppose that $\beta$ is a proper regeneration of $\alpha$.
	Since $\beta$ is maximal dimensional, it is trivalent.
	This means that there are exactly two vertices in $\beta$ that degenerate to $V_i$.
	One of those vertices is adjacent to the edge of $\beta$ with limit equal to $E_{i,4}$, and this edge is adjacent to two other edges, one that contracts in $\alpha$, and the other becomes in $\alpha$ one of $\{E_{i,1}, E_{i,2}, E_{i,3}\}$.
	We thus get a well-defined map
	\begin{align*}
		\tau: \mathcal{A}_{+}\cup \mathcal{A}_{-} \to \prod_{i=1}^k\{E_{i,1}, E_{i,2}, E_{i,3}\}.
	\end{align*}
	It is easy to see that this map is injective, since the knowledge of the local deformations of each $4$-valent vertex is enough to determine the combinatorial type $\beta$.
	
	Moreover, we can determine $\GS_y(\beta)$ directly from $\tau(\beta)$ as follows.
	Denote by $W$ the contribution to $\GS_y(\Gamma)$ of all the collinear cycles and vertices except $V_1,\dots,V_k$.
	For $1\le i \le k$ define
	\begin{align*}
		\GS_y(E_{i,1}) & = \left[\left| \oa_{V_i}(E_{i,1})\wedge \oa_{V_i}(E_{i,4})\right| \right]_y^{-}\cdot \left[\left| \oa_{V_i}(E_{i,2})\wedge \oa_{V_i}(E_{i,3})\right|\right]_y^{-}\\
		\GS_y(E_{i,2}) & = \left[\left| \oa_{V_i}(E_{i,2})\wedge \oa_{V_i}(E_{i,4})\right|\right]_y^{-}\cdot \left[\left| \oa_{V_i}(E_{i,1})\wedge \oa_{V_i}(E_{i,3})\right|\right]_y^{-}\\
		\GS_y(E_{i,3}) & = \left[\left| \oa_{V_i}(E_{i,3})\wedge \oa_{V_i}(E_{i,4})\right|\right]_y^{-}\cdot \left[\left| \oa_{V_i}(E_{i,1})\wedge \oa_{V_i}(E_{i,2})\right|\right]_y^{-}.	
	\end{align*}
	Then
	\begin{equation*}
		\GS_y(\beta) = W\cdot \prod_{i=1}^k \GS_y(\tau(\beta)_i).
	\end{equation*}
	
	For $E\in \Gamma^1$ define $\mathcal{L}(E):=\nicefrac{\R^2}{\oa_V(E)\cdot \R}$ where $V$ is an endpoint of $E$.
	We define \emph{displacement vectors} $\delta(E)\in \mathcal{L}(E)$ for all edges $E\in \Gamma^1$.
	Those vectors will describe where the edges move as we move $\bs{x}$ along the line $\bs{x}+\bs{w}\cdot\R$.
	We will define $\delta(E)$ inductively.
	If $p_i$ is a marking in the $E$ (either in the interior of $E$ or a marked endpoint of $E$), we define $\delta(E)$ as the image of $w_i$ in $\mathcal{L}(E)$.
	If $V$ is a trivalent vertex with edges $E_1,E_2,E_3$ s.t. $\delta(E_1)$ and $\delta(E_2)$ are already defined, we define $\delta(E_3)$ to be the image in $\mathcal{L}(E_3)$ of the intersection point of the lines $\delta(E_1) + \oa_V(E_1)\cdot \R$ and $\delta(E_2) + \oa_V(E_2)\cdot \R$.
	If $V$ is a $4$-valent vertex with edges $E_1,E_2,E_3,E_4$, s.t. $\delta(E_i)$ is already defined for $1\le i\le 3$, we define $\delta'$ to be the intersection point of the lines $\delta(E_1)+\oa_V(E_1)\cdot\R$ and $\delta(E_2)+\oa_V(E_2)\cdot\R$, and then define $\delta(E_4)$ as the image in $\mathcal{L}(E_4)$ of the intersection point of the lines $ \delta'+(\oa_V(E_1)+\oa_V(E_2))\cdot \R$ and $\delta(E_3)+\oa_V(E_3)\cdot \R$.
	Using the balancing condition we get that
	\begin{align*}
		&\oa_V(E_4)\wedge \delta(E_4) = -\left(\oa_V(E_1)+\oa_V(E_2)+\oa_V(E_3)\right)\wedge \delta(E_4) \\
		=& -\left(\oa_V(E_1)+\oa_V(E_2)\right)\wedge \delta(E_4) - \oa_V(E_3)\wedge \delta(E_4) \\
		=& -\left(\oa_V(E_1)+\oa_V(E_2)\right)\wedge \delta' - \oa_V(E_3)\wedge \delta(E_3) \\
		=& -\oa_V(E_1)\wedge \delta' - \oa_V(E_2)\wedge \delta' - \oa_V(E_3)\wedge \delta(E_3) \\
		=& -\sum_{i=1}^3 \delta(E_i)\wedge \oa_V(E_i),
	\end{align*} so the definition of $\delta(E_4)$ does not depend on the choice of order among $E_1,E_2$, and $E_3$.
	
	We will now assign to each $E_{i,j}$ ($1\le i\le k$, $1\le j\le 3$) a sign $\sigma(E_{i,j})$ that will indicate if the corresponding deformation of $V_i$ occur for regenerations in $\mathcal{A}_{+}$ or for regenerations in $\mathcal{A}_{-}$.
	Let $\{E_{i,1}, E_{i,2}, E_{i,3}\}\setminus\{E_{i,j}\} = \{E,E'\}$.
	Denote by $\delta'$ the intersection point of the lines $ \delta(E)+\oa_{V_i}(E)\cdot \R$ and $\delta(E') + \oa_{V_i}(E')\cdot \R$.
	We set $\sigma(E_{i,j})$ to be the sign of $\oa_{V_i}(E_{i,j})\wedge \left( \delta'-\delta(E_{i,j}) \right)$.
	It is defined in this way so that $\sigma(E_{i,j})$ is positive exactly when the local deformation of $V_i$ that correspond to $E_{i,j}$ can occur with displacements $\delta(E_{i,1}), \delta(E_{i,2})$, and $\delta(E_{i,3})$.
	
	Since everything is local, and $\Gamma$ is regular, we get that $\tau(\mathcal{A}_{+})$ (correspondingly, $ \tau(\mathcal{A}_{-})$) is exactly $\prod_{i=1}^k\setc{E_{i,j}}{\sigma(E_{i,j})=1}$ (correspondingly, $\prod_{i=1}^k\setc{E_{i,j}}{\sigma(E_{i,j})=-1}$).
	
	The final piece of the proof is the fact that locally the refined 
weight is invariant, i.e. for all $1\le i\le k$, $\sum_{j=1}^3 \sigma(E_{i,j})\cdot \GS_y(E_{i,j})=0$.
	This is well known, see for example \cite[Lemma 3.3]{IM}.
	
	Putting everything together we get that (here $W$ is the contribution to $\GS_y(\Gamma)$ of all the collinear cycles and vertices except $V_1,\dots, V_k$)
	\begin{align*}
		& \sum_{\beta\in\mathcal{A}_{+}}\GS_y(\beta) - \sum_{\beta\in \mathcal{A}_{-}}\GS_y(\beta) = \\
		= & \sum_{\beta\in \mathcal{A}_{+}} W\cdot \prod_{i=1}^k \GS_y(\tau(\beta)_i) - \sum_{\beta\in \mathcal{A}_{-}} W\cdot \prod_{i=1}^k \GS_y(\tau(\beta)_i) = \\
		= & W\cdot \left( \prod_{i=1}^k \sum_{\substack{j \\\sigma(E_{i,j})=1}}\GS_y(E_{i,j}) - \prod_{i=1}^k \sum_{\substack{j \\ \sigma(E_{i,j})=-1}}\GS_y(E_{i,j}) \right) =0
	\end{align*}
\end{proof}

\subsection{Unmarked $5$-valent vertex}\label{sec:refined_invariance_5valent}
Let $\alpha$ be a combinatorial type of a tropical curve as in case $(2)$ of Lemma \ref{lemma:codim1_degenerations_list}.
We denote by $v_1$ the primitive vector in the direction of the unique collinear cycle in the codimension $1$ fragment and by $w_1$ and $w_1'$ the weights of the edges in the collinear cycle.
By $v_2, v_3$, and $v_4$ we denote the vectors in the direction of the other edges adjacent to the unique $5$ valent vertex of $\alpha$.
For simplicity of exposition we add some assumptions on the values of $v_1,\dots,v_4$ and $w_1,w_1'$.
Namely, we assume $ v_1\wedge v_2$, $v_2 \wedge v_3$, $v_3 \wedge v_4$, $v_4 \wedge v_1$, $v_3 \wedge v_1$, and $ v_4 \wedge v_2$ are all positive.
This correspond to the vectors $v_1$, $v_2$, $v_3$, and $v_4$ all to appear in order when traversing them counter-clockwise, as well as the angles between $v_3$ to $v_1$ and between $v_4$ to $v_2$ to be less than $\pi$ (see Figure \ref{fig:degen_unmarked_4valent_r1}(a); recall that in all figures a black circle denotes a marking restricted to a vertex, while a black cross denotes an unrestricted marking, see Remark \ref{remark:figure_convention}).
Some choices for the signs of the above wedge products will lead to a symmetric picture, while others would lead to different proper regenerations of $\alpha$, but the analysis and final result stays the same.

Following the procedure outlined in Remark \ref{remark:exhaustiveness_of_regenerations}, we get the regenerations depicted in Figure \ref{fig:degen_unmarked_4valent_r1}(b)-(f).
The conditions on $v_1,\dots,v_4$ and $w_1,w_1'$ we have assumed ensure that the combinatorial type that is depicted in Figure \ref{fig:degen_unmarked_4valent_r1}(b) is in $\mathcal{A}_{-}$ (see the notations of Lemma \ref{lemma:local_refined_invariance}) and all the other regenerations are in $\mathcal{A}_{+}$.
It is thus enough to show, after cancelation by $\frac{2}{y^{1/2}+y^{-1/2}}$, that
\begin{align*}
	& \varphi^{(1)}_y(w_1, w_1', v_1\wedge v_2)\cdot \left[ v_3 \wedge v_4  \right]^{-}_y  = \\
	= & \left[ w_1 v_4\wedge v_1 \right]^{-}_y\cdot \left[(w_1v_1+v_4)\wedge v_2 \right]_y^{-}\cdot \left[ w_1' v_3\wedge v_1 \right]^{-}_y  + \\
	& + \varphi^{(1)}_y(w_1, w_1', v_4\wedge v_1)\cdot \left[ v_2 \wedge v_3 \right]_y^{-}  +\\
	& + \left[ w_1' v_4\wedge v_1 \right]^{-}_y\cdot \left[(w_1'v_1+v_4)\wedge v_2 \right]_y^{-}\cdot \left[ w_1 v_3\wedge v_1 \right]^{-}_y + \\
	& + \varphi^{(1)}_y(w_1, w_1', v_3\wedge v_1)\cdot \left[ v_4 \wedge v_2 \right]_y^{-}.
\end{align*}
We now verify this identity.
Denote $\alpha := v_1\wedge v_2$, $\beta := v_2\wedge v_3$, $\gamma := v_3\wedge v_1$, $w := w_1+w_1'$, and $\delta := w_1-w_1'$.
The balancing condition $wv_1+v_2+v_3+v_4=0$ gives
\[
	v_4\wedge v_1 = \alpha-\gamma, \quad v_3\wedge v_4 = \beta - w\gamma, \quad v_4\wedge v_2 = \beta - w\alpha,
\]
as well as $(w_1 v_1+v_4)\wedge v_2 = \beta - w_1'\alpha$ and $(w_1' v_1+v_4)\wedge v_2 = \beta - w_1\alpha$.
Expanding $\varphi^{(1)}_y$ using \eqref{eq:def_varphi_1}, the identity separates into a part involving $\varphi^{(0)}_y(w_1,w_1')$ and a part consisting of triple products of $[\cdot]^{-}_y$-brackets.

The $\varphi^{(0)}_y$-part reads
\[
	[w\alpha]^{-}_y [\beta-w\gamma]^{-}_y = [w(\alpha-\gamma)]^{-}_y [\beta]^{-}_y + [w\gamma]^{-}_y [\beta-w\alpha]^{-}_y.
\]
By Lemma~\ref{lemma:algebra_of_quantum_weights}(3), both sides equal $\frac{y^{1/2}+y^{-1/2}}{(y^{1/2}-y^{-1/2})^2}\left( [w\alpha+\beta-w\gamma]^{+}_y - [w\alpha-\beta+w\gamma]^{+}_y \right)$, since the remaining terms from the right-hand side cancel by $[-a]^{+}_y = [a]^{+}_y$.

The triple-product part reads
\begin{align*}
	[w_1\alpha]^{-}_y [w_1'\alpha]^{-}_y [\beta-w\gamma]^{-}_y = \ & [w_1(\alpha-\gamma)]^{-}_y [\beta-w_1'\alpha]^{-}_y [w_1'\gamma]^{-}_y \\
	+ \ & [w_1(\alpha-\gamma)]^{-}_y [w_1'(\alpha-\gamma)]^{-}_y [\beta]^{-}_y \\
	+ \ & [w_1'(\alpha-\gamma)]^{-}_y [\beta-w_1\alpha]^{-}_y [w_1\gamma]^{-}_y \\
	+ \ & [w_1\gamma]^{-}_y [w_1'\gamma]^{-}_y [\beta-w\alpha]^{-}_y.
\end{align*}
By Lemma~\ref{lemma:algebra_of_quantum_weights}(3) followed by~(4), any product $[p]^{-}_y[q]^{-}_y[r]^{-}_y$ equals
\[
	\tfrac{1}{(y^{1/2}-y^{-1/2})^2} \left( [p+q+r]^{-}_y + [r-p-q]^{-}_y - [p-q+r]^{-}_y - [-p+q+r]^{-}_y \right).
\]
Expanding all five triple products and simplifying with $[-a]^{-}_y = -[a]^{-}_y$, one checks that both sides equal
\[
	\tfrac{1}{(y^{1/2}-y^{-1/2})^2}\left( [w\alpha+\beta-w\gamma]^{-}_y - [w\alpha+w\gamma-\beta]^{-}_y - [\delta\alpha+\beta-w\gamma]^{-}_y + [\delta\alpha+w\gamma-\beta]^{-}_y \right).
\]

\begin{figure}
	\begin{center}
	\begin{scaletikzpicturetowidth}{\textwidth}
	\begin{tikzpicture}
		\tikzset{every path/.style={line width=0.7pt}}
		\begin{scope}[scale=\tikzscale, xshift=0, yshift=-0, decoration={markings,mark=at position 0.5 with {\arrow{>}}}]
			\draw [red, postaction={decorate}] (0.272727, 0.272727) -- node [above, very near end, black] {\footnotesize $v_1$} (-0.000000, -0.000000);
			\draw [red, postaction={decorate}] (1.000000, 0.363636) -- node [above, very near start, black] {\footnotesize $v_2$} (0.720909, 0.733637);
			\draw [red, postaction={decorate}] (0.836364, 1.000000) -- node [above, very near start, black] {\footnotesize $v_3$} (0.720909, 0.733637);
			\draw [red, postaction={decorate}] (0.720909, 0.733637) -- node [above, very near end, black] {\footnotesize $v_4$} (0.454545, 1.000000);
			\draw [red] (0.733637, 0.720909) -- node [left, black] {\tiny $w_1$} (0.279091, 0.266363);
			\draw [red] (0.720909, 0.733637) -- node [below, black] {\tiny $w_1'$} (0.266363, 0.279091);
			\draw [black, fill=black] (0.272727, 0.272727) circle (0.015000);
			
		\end{scope}
		\begin{scope}[scale=\tikzscale, xshift=20, yshift=-0]
			\draw (0,0) node {$ (a) $};
		\end{scope}
		\begin{scope}[scale=\tikzscale, xshift=40, yshift=-0, decoration={markings,mark=at position 0.5 with {\arrow{>}}}]
			\draw [red, postaction={decorate}] (0.272727, 0.272727) -- node [above, very near end, black] {\footnotesize $v_1$} (0.000000, 0.000000);
			\draw [red, postaction={decorate}] (0.733733, 0.097370) -- node [above, very near start, black] {\footnotesize $v_2$} (0.454642, 0.467370);
			\draw [red, postaction={decorate}] (0.570097, 1.000000) -- node [above, very near start, black] {\footnotesize $v_3$} (0.461006, 0.727273);
			\draw [red, postaction={decorate}] (0.461006, 0.727273) -- node [above, very near end, black] {\footnotesize $v_4$} (0.188279, 1.000000);
			\draw [red] (0.454642, 0.467370) -- (0.461006, 0.727273);
			\draw [red] (0.467370, 0.454642) -- node [left, black] {\tiny $w_1$} (0.279091, 0.266363);
			\draw [red] (0.454642, 0.467370) -- node [below, black] {\tiny $w_1'$} (0.266363, 0.279091);
			\draw [black, fill=black] (0.272727, 0.272727) circle (0.015000);
			
		\end{scope}
		\begin{scope}[scale=\tikzscale, xshift=50, yshift=-0]
			\draw (0,0) node {$ (b) $};
		\end{scope}
		\begin{scope}[scale=\tikzscale, xshift=80, yshift=-0, decoration={markings,mark=at position 0.5 with {\arrow{>}}}]
			\draw [red, postaction={decorate}] (0.272727, 0.272727) -- node [above, very near end, black] {\footnotesize $v_1$} (0.000000, 0.000000);
			\draw [red, postaction={decorate}] (1.000000, 0.016395) -- node [above, very near start, black] {\footnotesize $v_2$} (0.727273, 0.380031);
			\draw [red, postaction={decorate}] (0.836364, 1.000000) -- node [above, very near start, black] {\footnotesize $v_3$} (0.733637, 0.720909);
			\draw [red, postaction={decorate}] (0.460477, 0.473205) -- node [above, very near end, black] {\footnotesize $v_4$} (-0.000000, 0.933683);
			\draw [red] (0.460477, 0.473205) -- (0.727273, 0.380031);
			\draw [red] (0.460477, 0.473205) -- node [left, black] {\tiny $w_1$} (0.266363, 0.279091);
			\draw [red] (0.727273, 0.380031) -- (0.733637, 0.720909);
			\draw [red] (0.733637, 0.720909) -- node [below, near end, black] {\tiny $w_1'$} (0.279091, 0.266363);
			\draw [black, fill=black] (0.272727, 0.272727) circle (0.015000);
			
		\end{scope}
		\begin{scope}[scale=\tikzscale, xshift=100, yshift=-0]
			\draw (0,0) node {$ (c) $};
		\end{scope}
	
		\begin{scope}[scale=\tikzscale, xshift=00, yshift=-40, decoration={markings,mark=at position 0.5 with {\arrow{>}}}]
			\draw [red, postaction={decorate}] (0.272727, 0.272727) -- node [above, very near end, black] {\footnotesize $v_1$} (0.000000, 0.000000);
			\draw [red, postaction={decorate}] (1.000000, 0.152356) -- node [above, very near start, black] {\footnotesize $v_2$} (0.727273, 0.515993);
			\draw [red, postaction={decorate}] (0.836364, 0.788720) -- node [above, very near start, black] {\footnotesize $v_3$} (0.727273, 0.515993);
			\draw [red, postaction={decorate}] (0.456809, 0.469537) -- node [above, very near end, black] {\footnotesize $v_4$} (0.137625, 0.788720);
			\draw [red] (0.456809, 0.469537) -- (0.727273, 0.515993);
			\draw [red] (0.469537, 0.456809) -- node [left, black] {\tiny $w_1$} (0.279091, 0.266363);
			\draw [red] (0.456809, 0.469537) -- node [below, black] {\tiny $w_1'$} (0.266363, 0.279091);
			\draw [black, fill=black] (0.272727, 0.272727) circle (0.015000);
			
		\end{scope}
		\begin{scope}[scale=\tikzscale, xshift=20, yshift=-40]
			\draw (0,0) node {$ (d) $};
		\end{scope}\begin{scope}[scale=\tikzscale, xshift=40, yshift=-40, decoration={markings,mark=at position 0.5 with {\arrow{>}}}]
			\draw [red, postaction={decorate}] (0.272727, 0.272727) -- node [above, very near end, black] {\footnotesize $v_1$} (0.000000, 0.000000);
			\draw [red, postaction={decorate}] (0.975703, 0.000000) -- node [above, very near start, black] {\footnotesize $v_2$} (0.653017, 0.430248);
			\draw [red, postaction={decorate}] (0.836364, 1.000000) -- node [above, very near start, black] {\footnotesize $v_3$} (0.720909, 0.733637);
			\draw [red, postaction={decorate}] (0.436612, 0.423884) -- node [above, very near end, black] {\footnotesize $v_4$} (0.000000, 0.860497);
			\draw [red] (0.720909, 0.733637) -- (0.653017, 0.430248);
			\draw [red] (0.720909, 0.733637) -- node [above, black] {\tiny $w_1$} (0.266363, 0.279091);
			\draw [red] (0.653017, 0.430248) -- (0.436612, 0.423884);
			\draw [red] (0.436612, 0.423884) -- node [below, black] {\tiny $w_1'$} (0.279091, 0.266363);
			\draw [black, fill=black] (0.272727, 0.272727) circle (0.015000);
			
		\end{scope}
		\begin{scope}[scale=\tikzscale, xshift=60, yshift=-40]
			\draw (0,0) node {$ (e) $};
		\end{scope}
		\begin{scope}[scale=\tikzscale, xshift=80, yshift=-40, decoration={markings,mark=at position 0.5 with {\arrow{>}}}]
			\draw [red, postaction={decorate}] (0.292683, 0.371843) -- node [above, very near end, black] {\footnotesize $v_1$} (0.000000, 0.079160);
			\draw [red, postaction={decorate}] (1.000000, 0.000000) -- node [above, very near start, black] {\footnotesize $v_2$} (0.780488, 0.292683);
			\draw [red, postaction={decorate}] (0.708573, 0.963342) -- node [above, very near start, black] {\footnotesize $v_3$} (0.585136, 0.677023);
			\draw [red, postaction={decorate}] (0.780488, 0.292683) -- node [above, very near end, black] {\footnotesize $v_4$} (0.109828, 0.963342);
			\draw [red] (0.585136, 0.677023) -- (0.780488, 0.292683);
			\draw [red] (0.597864, 0.664295) -- node [left, black] {\tiny $w_1$} (0.299047, 0.365479);
			\draw [red] (0.585136, 0.677023) -- node [right, near end, black] {\tiny $w_1'$} (0.286319, 0.378207);
			\draw [black, fill=black] (0.292683, 0.371843) circle (0.015000);
			
		\end{scope}
		\begin{scope}[scale=\tikzscale, xshift=100, yshift=-40]
			\draw (0,0) node {$ (f) $};
		\end{scope}
	\end{tikzpicture}
	\end{scaletikzpicturetowidth}
	\end{center}
	
	\caption{Case $ (2) $ of Lemma \ref{lemma:codim1_degenerations_list} and its possible regenerations.}\label{fig:degen_unmarked_4valent_r1}
\end{figure}
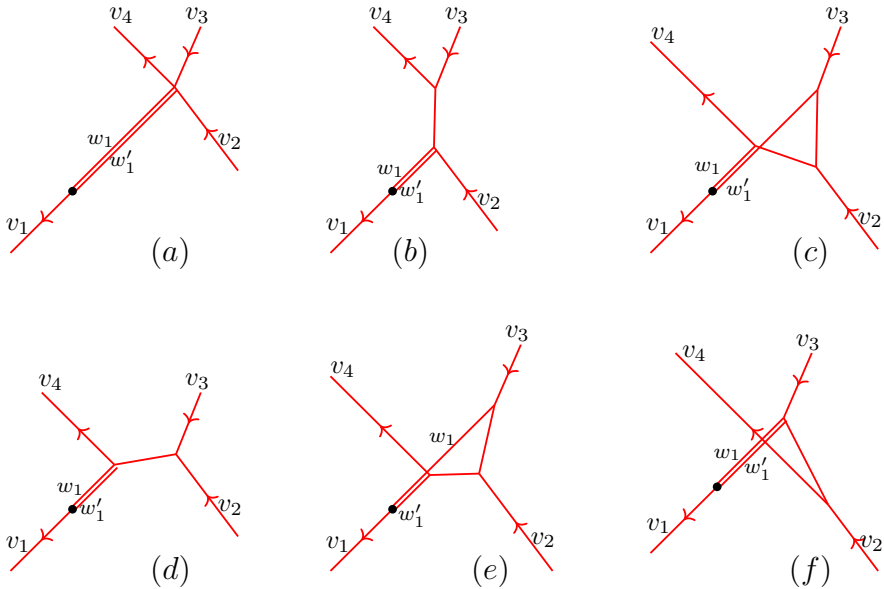

\subsection{Marked four valent vertex}\label{sec:refined_invariance_proof_marked_4valent}
We will now deal with case $(3)$ of Lemma \ref{lemma:codim1_degenerations_list}.
Case $(3a)$ with $r\in \{1,2\}$ was dealt with in \cite[Sections 4.3, 4.4]{SS}, we will not repeat it here.

\subsubsection{Case $(3a)$ with $r=3$}
The regenerations in this case are very similar to the regenerations in case $(3a)$ with $r=1$ (see \cite[Section 4.3]{SS}).
Let $E_1,E_2,E_3$ be the edges adjacent to the marked $4$-valent vertex which are contained in the bounded component of $\Gamma\setminus \bq$, and let $E_4$ be the edge adjacent to the $4$-valent vertex which is contained in an unbounded component of $\Gamma\setminus\bq$.
For $i =1,\dots,4$ denote by $ v_i$ the vector in the direction of $E_i$ with lattice length equal to the weight of $E_i$.
Assume for simplicity that $v_1\wedge v_2$, $v_2\wedge v_3$, $v_3\wedge v_4$, $v_4\wedge v_1$, $v_3\wedge v_1$, and $v_4 \wedge v_2$ are all positive, see Figure \ref{fig:degen_marked_4valent_a_r3}.
Other assumptions lead to similar analysis.
Note that the arrows in Figure \ref{fig:degen_marked_4valent_a_r3} denote the regular orientation, not the orientation of the vectors $v_1,v_2,v_3,v_4$, which are always point away from the $4$-valent vertex.

We assert that under those assumptions there exist $3$ regenerations of $\alpha$ depicted in Figure \ref{fig:degen_marked_4valent_a_r3}(b)-(d), and that the regeneration that contains a vertex adjacent to edges with directions $v_1$ and $v_2$ (Figure \ref{fig:degen_marked_4valent_a_r3}(d)) is in $\mathcal{A}_{-}$, while the other two regenerations are in $\mathcal{A}_{+}$.
Denote by $\beta_1$ (respectively $\beta_2$, respectively $\beta_3$) the regeneration depicted in Figure \ref{fig:degen_marked_4valent_a_r3}(b) (respectively Figure \ref{fig:degen_marked_4valent_a_r3}(c), respectively Figure \ref{fig:degen_marked_4valent_a_r3}(d)), and let $\Gamma_i$ be a curve of type $\beta_i$.
For a curve $\Gamma$ whose combinatorial type is a regeneration of $\alpha$ and for $j\in \{1,2,3\} $, denote by $V_j(\Gamma)$ the vertex in $\Gamma$ that degenerates to the $4$-valent vertex in $\alpha$ and is adjacent to the edge with direction $v_j$, and by $\delta_j(\Gamma)$ the displacement vector of $V_j(\Gamma)$ relative to the position of the $4$-valent vertex in $\alpha$.
The bounded component of $\alpha$ imposes linear condition on the vectors $\delta_j(\Gamma)$ of the form $\sum_{j=1}^3 a_j\cdot (v_j\wedge \delta_j(\Gamma)) = 0$ with $a_1,a_2,a_3\in \R$, that holds for every $\Gamma$ with combinatorial type which is a regeneration of $\alpha$.
Picking $\Gamma$ with $v_3\wedge \delta_3(\Gamma)=0$ and traversing the cycle in the bounded component that connects $E_1$ and $E_2$ we get that $a_1$ and $a_2$ have the same sign.
Similarly $a_2$ and $a_3$ also have the same sign, so we can assume $a_1,a_2,a_3>0$.
Now obviously $\delta_j(\Gamma_i)=\delta_k(\Gamma_i)$ whenever $j$ and $k$ are different from $i$, and additionally $v_1\wedge (\delta_1(\Gamma_1)-\delta_2(\Gamma_1)) < 0$, $v_2\wedge (\delta_2(\Gamma_2)-\delta_1(\Gamma_2)) < 0$, and $v_3\wedge (\delta_3(\Gamma_3)-\delta_1(\Gamma_3)) > 0$.
Combining this together gives us $\left(\sum_{j=1}^3 a_jv_j\right)\wedge \delta_2(\Gamma_1) > 0$, $\left(\sum_{j=1}^3 a_jv_j\right)\wedge \delta_1(\Gamma_2) > 0$, and $\left(\sum_{j=1}^3 a_jv_j\right)\wedge \delta_1(\Gamma_3) < 0$, which proves the assertion.

Thus the local invariance of the refined count in this case follows from the same formula as in case $(3a)$ with $r=1$, i.e.
\begin{equation*}
	\left[ v_1\wedge v_2 \right]^{+}_y \cdot \left[ v_3 \wedge v_4 \right]^{-}_y =  \left[ v_3\wedge v_1 \right]^{+}_y \cdot \left[ v_4 \wedge v_2 \right]^{-}_y + \left[ v_2\wedge v_3 \right]^{+}_y \cdot \left[ v_4 \wedge v_1 \right]^{-}_y,
\end{equation*}
which follows readily from the balancing condition $v_1+v_2+v_3+v_4 =0 $ and Lemma \ref{lemma:algebra_of_quantum_weights}.

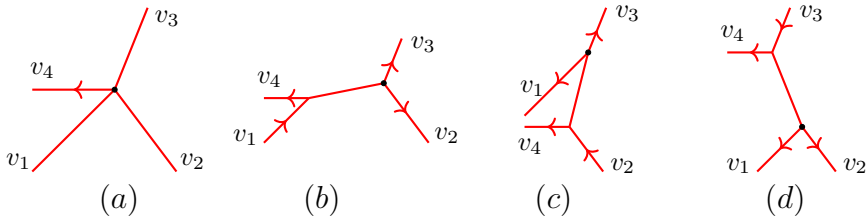
\begin{figure}
	\begin{center}
	\begin{scaletikzpicturetowidth}{\textwidth}
	\begin{tikzpicture}
		\begin{scope}[scale=\tikzscale, xshift=0, yshift=-0, decoration={markings,mark=at position 0.5 with {\arrow{>}}}]
			\draw [red, thick] (0.000000, 0.000000) -- node [black, left, very near start] {\footnotesize $v_1$} (0.500000, 0.500000);
			\draw [red, thick] (0.875000, 0.000000) -- node [black, right, very near start] {\footnotesize $v_2$} (0.500000, 0.500000);
			\draw [red, thick] (0.700000, 1.000000) -- node [right, very near start, black] {\footnotesize $v_3$} (0.500000, 0.500000);
			\draw [thick, red, postaction={decorate}] (0.500000, 0.500000) -- node [above, very near end, black] {\footnotesize $v_4$} (0.000000, 0.500000);
			\draw [black, fill=black] (0.500000, 0.500000) circle (0.015000);
			
		\end{scope}
		\begin{scope}[scale=\tikzscale, xshift=15, yshift=-5]
			\draw (0,0) node {$ (a) $};
		\end{scope}
		\begin{scope}[scale=\tikzscale, xshift=40, yshift=5, decoration={markings,mark=at position 0.5 with {\arrow{>}}}]
			\draw [thick, red, postaction={decorate}] (0.000000, 0.000000) -- node [left, very near start, black] {\footnotesize $v_1$} (0.272727, 0.272727);
			\draw [thick, red, postaction={decorate}] (0.727273, 0.363636) -- node [right, very near end, black] {\footnotesize $v_2$} (1.000000, 0.000000);
			\draw [thick, red, postaction={decorate}] (0.727273, 0.363636) -- node [right, very near end, black] {\footnotesize $v_3$} (0.836364, 0.636364);
			\draw [thick, red, postaction={decorate}] (0.272727, 0.272727) -- node [above, very near end, black] {\footnotesize $v_4$} (0.000000, 0.272727);
			\draw [thick, red] (0.727273, 0.363636) -- (0.272727, 0.272727);
			\draw [black, fill=black] (0.727273, 0.363636) circle (0.015000);
			
		\end{scope}
		\begin{scope}[scale=\tikzscale, xshift=50, yshift=-5]
			\draw (0,0) node {$ (b) $};
		\end{scope}
		\begin{scope}[scale=\tikzscale, xshift=85, yshift=-0, decoration={markings,mark=at position 0.5 with {\arrow{>}}}]
			\draw [thick, red, postaction={decorate}] (0.386364, 0.727273) -- node [above, very near end, black] {\footnotesize $v_1$} (0.000000, 0.340909);
			\draw [thick, red, postaction={decorate}] (0.477273, 0.000000) -- node [right, very near start, black] {\footnotesize $v_2$} (0.272727, 0.272727);
			\draw [thick, red, postaction={decorate}] (0.386364, 0.727273) -- node [right, very near end, black] {\footnotesize $v_3$} (0.495455, 1.000000);
			\draw [thick, red, postaction={decorate}] (0.272727, 0.272727) -- node [below, very near end, black] {\footnotesize $v_4$} (0.000000, 0.272727);
			\draw [thick, red] (0.386364, 0.727273) -- (0.272727, 0.272727);
			\draw [black, fill=black] (0.386364, 0.727273) circle (0.015000);
			
		\end{scope}
		\begin{scope}[scale=\tikzscale, xshift=90, yshift=-5]
			\draw (0,0) node {$ (c) $};
		\end{scope}

		\begin{scope}[scale=\tikzscale, xshift=120, yshift=-0, decoration={markings,mark=at position 0.5 with {\arrow{>}}}]
			\draw [thick, red, postaction={decorate}] (0.454545, 0.272727) -- node [black, left, very near end] {\footnotesize $v_1$} (0.181818, 0.000000);
			\draw [thick, red, postaction={decorate}] (0.454545, 0.272727) -- node [black, right, very near end] {\footnotesize $v_2$} (0.659091, 0.000000);
			\draw [thick, red, postaction={decorate}] (0.381818, 1.000000) -- node [black, right, very near start] {\footnotesize $v_3$} (0.272727, 0.727273);
			\draw [thick, red, postaction={decorate}] (0.272727, 0.727273) -- node [black, above, very near end] {\footnotesize $v_4$} (0.000000, 0.727273);
			\draw [thick, red] (0.454545, 0.272727) -- (0.272727, 0.727273);
			\draw [black, fill=black] (0.454545, 0.272727) circle (0.015000);
			
		\end{scope}
		\begin{scope}[scale=\tikzscale, xshift=130, yshift=-5]
			\draw (0,0) node {$ (d) $};
		\end{scope}
	\end{tikzpicture}
	\end{scaletikzpicturetowidth}
	\end{center}
	\caption{Case $ (3a) $ of Lemma \ref{lemma:codim1_degenerations_list} with $ r=3 $ and its possible regenerations.}\label{fig:degen_marked_4valent_a_r3}
\end{figure}

\subsubsection{Case $(3b)$}
This case is similar to the case $(2)$ of Lemma \ref{lemma:codim1_degenerations_list}.

We denote by $v_1$ the primitive vector in the direction of the unique collinear cycle in the codimension $1$ fragment and by $w_1$ and $w_1'$ the weights of the edges in the collinear cycle.
We denote by $E_2, E_3, E_4$ the edges adjacent to the $5$-valent vertex which are not part of a collinear cycle, and by $v_2, v_3, v_4$ vectors in their respective directions whose lattice lengths are equal to the weight of the corresponding edge.

As before, we assume $ v_1\wedge v_2$, $v_2 \wedge v_3$, $v_3 \wedge v_4$, $v_4 \wedge v_1$, $v_3 \wedge v_1$, and $ v_4 \wedge v_2$ are all positive.
The regenerations in this case depicted in Figure \ref{fig:degen_marked_4valent_b}(b)-(f).
The conditions on $v_1,\dots,v_4$ and $w_1,w_1'$ we have assumed ensure that the combinatorial type that is depicted in Figure \ref{fig:degen_marked_4valent_b}(b) is in $\mathcal{A}_{-}$ (see the notations of Lemma \ref{lemma:local_refined_invariance}) and all the other regenerations are in $\mathcal{A}_{+}$.
It is thus enough to show, after cancelation of $\frac{2}{y^{1/2}+y^{-1/2}}$, that
\begin{align*}
	& \varphi^{(1)}_y(w_1, w_1', v_1\wedge v_2)\cdot \left[ v_3 \wedge v_4  \right]^{+}_y  = \\
	= & \left[ w_1 v_4\wedge v_1 \right]^{-}_y\cdot \left[(w_1v_1+v_4)\wedge v_2 \right]_y^{+}\cdot \left[ w_1' v_3\wedge v_1 \right]^{-}_y  + \\
	& + \varphi^{(1)}_y(w_1, w_1', v_4\wedge v_1)\cdot \left[ v_2 \wedge v_3 \right]_y^{+}  +\\
	& + \left[ w_1' v_4\wedge v_1 \right]^{-}_y\cdot \left[(w_1'v_1+v_4)\wedge v_2 \right]_y^{+}\cdot \left[ w_1 v_3\wedge v_1 \right]^{-}_y + \\
	& + \varphi^{(1)}_y(w_1, w_1', v_3\wedge v_1)\cdot \left[ v_4 \wedge v_2 \right]_y^{+}.
\end{align*}
This is verified by the same argument as in Section~\ref{sec:refined_invariance_5valent}: one expresses all wedge products in terms of $v_1\wedge v_2$, $v_2\wedge v_3$, $v_3\wedge v_1$ using the balancing condition $(w_1+w_1')v_1+v_2+v_3+v_4=0$, expands $\varphi^{(1)}_y$ using~\eqref{eq:def_varphi_1}, and applies Lemma~\ref{lemma:algebra_of_quantum_weights}(1,3--5).

\begin{figure}
	\begin{center}
	\begin{scaletikzpicturetowidth}{\textwidth}
	\begin{tikzpicture}
		\tikzset{every path/.style = {line width=0.7pt, red}, every node/.style = {black}}
		\begin{scope}[scale=\tikzscale, xshift=0, yshift=-0, decoration={markings,mark=at position 0.5 with {\arrow{>}}}]
			\draw [postaction={decorate}] (0.272727, 0.272727) -- node [above, very near end, black] {\footnotesize $v_1$} (-0.000000, -0.000000);
			\draw [postaction={decorate}] (0.720909, 0.733637) -- node [above, very near end] {\footnotesize $v_2$} (1.000000, 0.363636);
			\draw [postaction={decorate}] (0.720909, 0.733637) -- node [above, very near end] {\footnotesize $v_3$} (0.836364, 1.000000);
			\draw [postaction={decorate}] (0.720909, 0.733637) -- node [above, very near end] {\footnotesize $v_4$} (0.454545, 1.000000);
			\draw [black, fill=black] (0.723454, 0.731091) circle (0.015000);
			\draw [] (0.733637, 0.720909) -- node [above left=-0.1, scale=0.8] {\tiny $w_1$} (0.279091, 0.266363);
			\draw [] (0.720909, 0.733637) -- node [below right=-0.1, scale=0.8] {\tiny $w'_1$} (0.266363, 0.279091);
			\draw [black, fill=black] (0.272727, 0.272727) circle (0.015000);
			
		\end{scope}
		\begin{scope}[scale=\tikzscale, xshift=20, yshift=-0]
			\draw (0,0) node {$ (a) $};
		\end{scope}
		\begin{scope}[scale=\tikzscale, xshift=40, yshift=-0, decoration={markings,mark=at position 0.5 with {\arrow{>}}}]
			\draw [postaction={decorate}] (0.272727, 0.272727) -- node [above, very near end] {\footnotesize $v_1$} (-0.000000, -0.000000);
			\draw [postaction={decorate}] (0.454642, 0.467370) -- node [above, very near end] {\footnotesize $v_2$} (0.733733, 0.097370);
			\draw [postaction={decorate}] (0.461006, 0.727273) -- node [above, very near end] {\footnotesize $v_3$} (0.570097, 1.000000);
			\draw [postaction={decorate}] (0.461006, 0.727273) -- node [above, very near end] {\footnotesize $v_4$} (0.188279, 1.000000);
			\draw [] (0.461006, 0.727273) -- (0.454642, 0.467370);
			\draw [black, fill=black] (0.461006, 0.727273) circle (0.015000);
			\draw [] (0.467370, 0.454642) -- node [above left=-0.1, scale=0.8] {\tiny $w_1$} (0.279091, 0.266363);
			\draw [] (0.454642, 0.467370) -- node [below right=-0.1, scale=0.8] {\tiny $w'_1$} (0.266363, 0.279091);
			\draw [black, fill=black] (0.272727, 0.272727) circle (0.015000);
			
		\end{scope}
		\begin{scope}[scale=\tikzscale, xshift=50, yshift=-0]
			\draw (0,0) node {$ (b) $};
		\end{scope}
		\begin{scope}[scale=\tikzscale, xshift=80, yshift=-0, decoration={markings,mark=at position 0.5 with {\arrow{>}}}]
			\draw [postaction={decorate}] (0.292683, 0.371843) -- node [above, very near end] {\footnotesize $v_1$} (-0.000000, 0.079160);
			\draw [postaction={decorate}] (0.780488, 0.292683) -- node [above, very near end] {\footnotesize $v_2$} (1.000000, 0.000000);
			\draw [postaction={decorate}] (0.585136, 0.677023) -- node [above, very near end] {\footnotesize $v_3$} (0.708573, 0.963342);
			\draw [postaction={decorate}] (0.780488, 0.292683) -- node [above, very near end] {\footnotesize $v_4$} (0.109828, 0.963342);
			\draw [] (0.780488, 0.292683) -- (0.585136, 0.677023);
			\draw [black, fill=black] (0.780488, 0.292683) circle (0.015000);
			\draw [] (0.597864, 0.664295) -- node [above left=-0.1, scale=0.8] {\tiny $w_1$} (0.299047, 0.365479);
			\draw [] (0.585136, 0.677023) -- node [below right=-0.1, scale=0.8] {\tiny $w'_1$} (0.286319, 0.378207);
			\draw [black, fill=black] (0.292683, 0.371843) circle (0.015000);
			
		\end{scope}
		\begin{scope}[scale=\tikzscale, xshift=95, yshift=-0]
			\draw (0,0) node {$ (c) $};
		\end{scope}
		\begin{scope}[scale=\tikzscale, xshift=120, yshift=-0, decoration={markings,mark=at position 0.5 with {\arrow{>}}}]
			\draw [postaction={decorate}] (0.272727, 0.272727) -- node [above, very near end] {\footnotesize $v_1$} (0.000000, 0.000000);
			\draw [postaction={decorate}] (0.727273, 0.380031) -- node [above, very near end] {\footnotesize $v_2$} (1.000000, 0.016395);
			\draw [postaction={decorate}] (0.720909, 0.733637) -- node [above, very near end] {\footnotesize $v_3$} (0.836364, 1.000000);
			\draw [postaction={decorate}] (0.473205, 0.460477) -- node [above, very near end] {\footnotesize $v_4$} (0.000000, 0.933683);
			\draw [] (0.727273, 0.380031) -- (0.720909, 0.733637);
			\draw [] (0.727273, 0.380031) -- (0.473205, 0.460477);
			\draw [black, fill=black] (0.727273, 0.380031) circle (0.015000);
			\draw [] (0.720909, 0.733637) -- node [left, scale=0.8, pos=0.2] {\tiny $w_1$} (0.266363, 0.279091);
			\draw [] (0.473205, 0.460477) -- node [below right=-0.1, scale=0.8] {\tiny $w'_1$} (0.279091, 0.266363);
			\draw [black, fill=black] (0.272727, 0.272727) circle (0.015000);
			
		\end{scope}
		\begin{scope}[scale=\tikzscale, xshift=135, yshift=-0]
			\draw (0,0) node {$ (d) $};
		\end{scope}
	
	\begin{scope}[scale=\tikzscale, xshift=0, yshift=-40, decoration={markings,mark=at position 0.5 with {\arrow{>}}}]
			\draw [postaction={decorate}] (0.272727, 0.272727) -- node [above, very near end] {\footnotesize $v_1$} (0.000000, 0.000000);
			\draw [postaction={decorate}] (0.653017, 0.430248) -- node [above, very near end] {\footnotesize $v_2$} (0.975703, 0.000000);
			\draw [postaction={decorate}] (0.720909, 0.733637) -- node [above, very near end] {\footnotesize $v_3$} (0.836364, 1.000000);
			\draw [postaction={decorate}] (0.436612, 0.423884) -- node [above, very near end] {\footnotesize $v_4$} (0.000000, 0.860497);
			\draw [] (0.653017, 0.430248) -- (0.720909, 0.733637);
			\draw [] (0.653017, 0.430248) -- (0.436612, 0.423884);
			\draw [black, fill=black] (0.653017, 0.430248) circle (0.015000);
			\draw [] (0.720909, 0.733637) -- node [above left=-0.1, scale=0.8, pos=0.25] {\tiny $w_1'$} (0.266363, 0.279091);
			\draw [] (0.436612, 0.423884) -- node [below right=-0.1, scale=0.8] {\tiny $w_1$} (0.279091, 0.266363);
			\draw [black, fill=black] (0.272727, 0.272727) circle (0.015000);
			
		\end{scope}
		\begin{scope}[scale=\tikzscale, xshift=15, yshift=-40]
			\draw (0,0) node {$ (e) $};
		\end{scope}
		\begin{scope}[scale=\tikzscale, xshift=40, yshift=-40, decoration={markings,mark=at position 0.5 with {\arrow{>}}}]
			\draw [postaction={decorate}] (0.272727, 0.272727) -- node [above, very near end] {\footnotesize $v_1$} (-0.000000, -0.000000);
			\draw [postaction={decorate}] (0.727273, 0.515993) -- node [above, very near end] {\footnotesize $v_2$} (1.000000, 0.152356);
			\draw [postaction={decorate}] (0.727273, 0.515993) -- node [above, very near end] {\footnotesize $v_3$} (0.836364, 0.788720);
			\draw [postaction={decorate}] (0.456809, 0.469537) -- node [above, very near end] {\footnotesize $v_4$} (0.137625, 0.788720);
			\draw [] (0.727273, 0.515993) -- (0.456809, 0.469537);
			\draw [black, fill=black] (0.727273, 0.515993) circle (0.015000);
			\draw [] (0.469537, 0.456809) -- node [above left=-0.1, scale=0.8] {\tiny $w_1$} (0.279091, 0.266363);
			\draw [] (0.456809, 0.469537) -- node [below right=-0.1, scale=0.8] {\tiny $w'_1$} (0.266363, 0.279091);
			\draw [black, fill=black] (0.272727, 0.272727) circle (0.015000);
			
		\end{scope}
		\begin{scope}[scale=\tikzscale, xshift=55, yshift=-40]
			\draw (0,0) node {$ (f) $};
		\end{scope}
	\end{tikzpicture}
	\end{scaletikzpicturetowidth}
	\end{center}
	\caption{Case $(3b)$ of Lemma \ref{lemma:codim1_degenerations_list} and its possible regenerations.}\label{fig:degen_marked_4valent_b}
\end{figure}

\subsection{Collinear cycle with 2 marked endpoints}\label{sec:refined_invariance_proof_collinear_2marked}
Case $(4)$ of Lemma \ref{lemma:codim1_degenerations_list} is already addressed in \cite[Section 4.5]{SS}.
Here, the relevant fragment of $\alpha$ appears in Figure \ref{fig:degen_collinear_2marked}(a), with its four possible regenerations given in Figures \ref{fig:degen_collinear_2marked}(b)--(e).
Let $V_1$ and $V_2$ be the marked endpoints of the collinear cycle, and let $v$ denote any vector parallel to one of the edges joining $V_1$ and $V_2$.
For any tropical plane curve $(\Gamma, h, \boldsymbol{g}, \boldsymbol{q})$ of combinatorial type $\alpha$, the expression $\left(h(V_1)-h(V_2)\right)\wedge v$ vanishes, and it changes sign between the regenerations in Figures \ref{fig:degen_collinear_2marked}(b,c) and those in Figures \ref{fig:degen_collinear_2marked}(d,e).
Consequently, the regenerations in Figures \ref{fig:degen_collinear_2marked}(b) and \ref{fig:degen_collinear_2marked}(c) belong to $\mathcal{A}_{+}$, while those in Figures \ref{fig:degen_collinear_2marked}(d) and \ref{fig:degen_collinear_2marked}(e) are contained in $\mathcal{A}_{-}$.
Furthermore, the refined weights of regenerations (b) and (d) are equal, as are those of (c) and (e).
Thus, equation~\eqref{eq:local_refined_invariance} holds in this case.

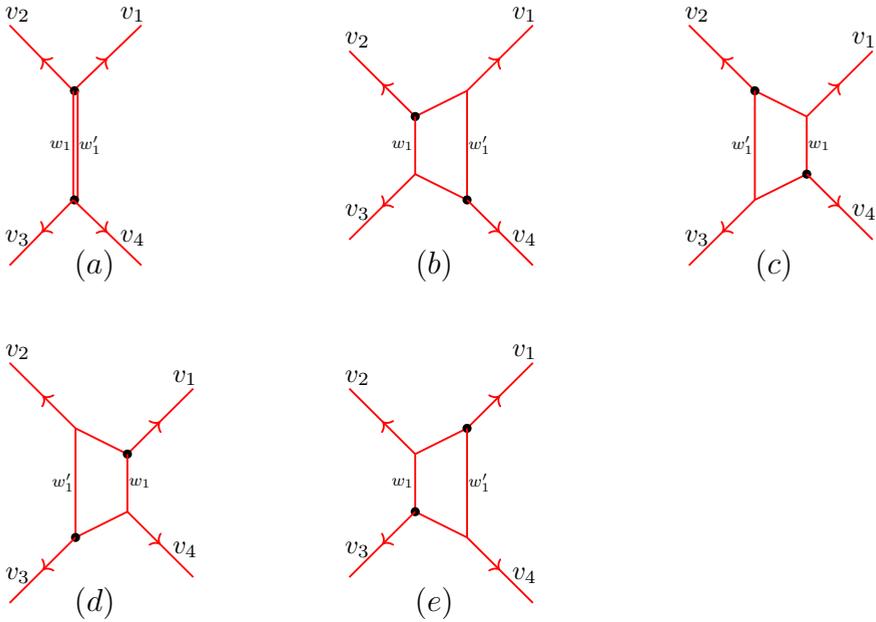
\begin{figure}
	\begin{center}
	\begin{scaletikzpicturetowidth}{\textwidth}
	\begin{tikzpicture}
		\tikzset{
			every path/.style = {line width=0.7pt, red},
			every node/.style = {black}
		}
		\begin{scope}[scale=\tikzscale, xshift=0, yshift=-0, decoration={markings,mark=at position 0.5 with {\arrow{>}}}]
			\draw [postaction={decorate}] (0.263727, 0.727273) -- node [above, very near end] {\footnotesize $v_1$} (0.545455, 1.000000);
			\draw [postaction={decorate}] (0.263727, 0.727273) -- node [above, very near end] {\footnotesize $v_2$} (-0.000000, 1.000000);
			\draw [black, fill=black] (0.268227, 0.727273) circle (0.015000);
			\node[anchor=east] at (0.268227, 0.727273) {\footnotesize $V_1$};
			\draw [] (0.281727, 0.727273) -- node [right=-0.1, scale=0.8] {\tiny $w'_1$} (0.281727, 0.272727);
			\draw [] (0.263727, 0.727273) -- node [left=-0.1, scale=0.8] {\tiny $w_1$} (0.263727, 0.272727);
			\draw [black, fill=black] (0.268227, 0.272727) circle (0.015000);
			\node[anchor=west] at (0.268227, 0.272727) {\footnotesize $V_2$};
			\draw [postaction={decorate}] (0.263727, 0.272727) -- node [above, very near end] {\footnotesize $v_3$} (0.000000, 0.000000);
			\draw [postaction={decorate}] (0.263727, 0.272727) -- node [above, very near end] {\footnotesize $v_4$} (0.545455, 0.000000);
		\end{scope}

		\begin{scope}[scale=\tikzscale, xshift=10, yshift=-0]
			\draw (0,0) node {$ (a) $};
		\end{scope}

		\begin{scope}[scale=\tikzscale, xshift=40, yshift=-0, decoration={markings,mark=at position 0.5 with {\arrow{>}}}]
			\draw [] (0.487335, 0.727273) -- (0.272727, 0.619969);
			\draw [postaction={decorate}] (0.487335, 0.727273) -- node [above, very near end] {\footnotesize $v_1$} (0.760062, 1.000000);
			\draw [] (0.487335, 0.727273) -- node [right=-0.1, scale=0.8] {\tiny $w'_1$} (0.487335, 0.272727);
			\draw [postaction={decorate}] (0.272727, 0.619969) -- node [above, very near end] {\footnotesize $v_2$} (0.000000, 0.892696);
			\draw [black, fill=black] (0.272727, 0.619969) circle (0.015000);
			\node[anchor=east] at (0.272727, 0.619969) {\footnotesize $V_1$};
			\draw [] (0.272727, 0.619969) -- node [left=-0.1, scale=0.8] {\tiny $w_1$} (0.272727, 0.380031);
			\draw [] (0.272727, 0.380031) -- (0.487335, 0.272727);
			\draw [postaction={decorate}] (0.272727, 0.380031) -- node [above, very near end] {\footnotesize $v_3$} (0.000000, 0.107304);
			\draw [black, fill=black] (0.487335, 0.272727) circle (0.015000);
			\node[anchor=west] at (0.487335, 0.272727) {\footnotesize $V_2$};
			\draw [postaction={decorate}] (0.487335, 0.272727) -- node [above, very near end] {\footnotesize $v_4$} (0.760062, 0.000000);
		\end{scope}
		\begin{scope}[scale=\tikzscale, xshift=50, yshift=-0]
			\draw (0,0) node {$ (b) $};
		\end{scope}
		\begin{scope}[scale=\tikzscale, xshift=80, yshift=-0, decoration={markings,mark=at position 0.5 with {\arrow{>}}}]
			\draw [] (0.487335, 0.619969) -- (0.272727, 0.727273);
			\draw [postaction={decorate}] (0.487335, 0.619969) -- node [above, very near end] {\footnotesize $v_1$} (0.760062, 0.892696);
			\draw [] (0.487335, 0.619969) -- node [right=-0.1, scale=0.8] {\tiny $w_1$} (0.487335, 0.380031);
			\draw [postaction={decorate}] (0.272727, 0.727273) -- node [above, very near end] {\footnotesize $v_2$} (0.000000, 1.000000);
			\draw [black, fill=black] (0.272727, 0.727273) circle (0.015000);
			\node[anchor=east] at (0.272727, 0.727273) {\footnotesize $V_1$};
			\draw [] (0.272727, 0.727273) -- node [left=-0.1, scale=0.8] {\tiny $w'_1$} (0.272727, 0.272727);
			\draw [] (0.487335, 0.380031) -- (0.272727, 0.272727);
			\draw [black, fill=black] (0.487335, 0.380031) circle (0.015000);
			\node[anchor=west] at (0.487335, 0.380031) {\footnotesize $V_2$};
			\draw [postaction={decorate}] (0.487335, 0.380031) -- node [above, very near end] {\footnotesize $v_4$} (0.760062, 0.107304);
			\draw [postaction={decorate}] (0.272727, 0.272727) -- node [above, very near end] {\footnotesize $v_3$} (0.000000, 0.000000);
		\end{scope}
		\begin{scope}[scale=\tikzscale, xshift=90, yshift=-0]
			\draw (0,0) node {$ (c) $};
		\end{scope}
		\begin{scope}[scale=\tikzscale, xshift=0, yshift=-40, decoration={markings,mark=at position 0.5 with {\arrow{>}}}]			
			\draw [] (0.487335, 0.619969) -- (0.272727, 0.727273);
			\draw [postaction={decorate}] (0.487335, 0.619969) -- node [above, very near end] {\footnotesize $v_1$} (0.760062, 0.892696);
			\draw [black, fill=black] (0.487335, 0.619969) circle (0.015000);
			\node[anchor=west] at (0.487335, 0.619969) {\footnotesize $V_1$};
			\draw [] (0.487335, 0.619969) -- node [right=-0.1, scale=0.8] {\tiny $w_1$} (0.487335, 0.380031);
			\draw [postaction={decorate}] (0.272727, 0.727273) -- node [above, very near end] {\footnotesize $v_2$} (0.000000, 1.000000);
			\draw [] (0.272727, 0.727273) -- node [left=-0.1, scale=0.8] {\tiny $w'_1$} (0.272727, 0.272727);
			\draw [] (0.487335, 0.380031) -- (0.272727, 0.272727);
			\draw [postaction={decorate}] (0.487335, 0.380031) -- node [above, very near end] {\footnotesize $v_4$} (0.760062, 0.107304);
			\draw [black, fill=black] (0.272727, 0.272727) circle (0.015000);
			\node[anchor=east] at (0.272727, 0.272727) {\footnotesize $V_2$};
			\draw [postaction={decorate}] (0.272727, 0.272727) -- node [above, very near end] {\footnotesize $v_3$} (0.000000, 0.000000);
		\end{scope}
		\begin{scope}[scale=\tikzscale, xshift=10, yshift=-40]
			\draw (0,0) node {$ (d) $};
		\end{scope}
		
		\begin{scope}[scale=\tikzscale, xshift=40, yshift=-40, decoration={markings,mark=at position 0.5 with {\arrow{>}}}]
			\draw [] (0.487335, 0.727273) -- (0.272727, 0.619969);
			\draw [postaction={decorate}] (0.487335, 0.727273) -- node [above, very near end] {\footnotesize $v_1$} (0.760062, 1.000000);
			\draw [black, fill=black] (0.487335, 0.727273) circle (0.015000);
			\node[anchor=west] at (0.487335, 0.727273) {\footnotesize $V_1$};
			\draw [] (0.487335, 0.727273) -- node [right=-0.1, scale=0.8] {\tiny $w'_1$} (0.487335, 0.272727);
			\draw [postaction={decorate}] (0.272727, 0.619969) -- node [above, very near end] {\footnotesize $v_2$} (0.000000, 0.892696);
			\draw [] (0.272727, 0.619969) -- node [left=-0.1, scale=0.8] {\tiny $w_1$} (0.272727, 0.380031);
			\draw [] (0.272727, 0.380031) -- (0.487335, 0.272727);
			\draw [black, fill=black] (0.272727, 0.380031) circle (0.015000);
			\node[anchor=east] at (0.272727, 0.380031) {\footnotesize $V_2$};
			\draw [postaction={decorate}] (0.272727, 0.380031) -- node [above, very near end] {\footnotesize $v_3$} (0.000000, 0.107304);
			\draw [postaction={decorate}] (0.487335, 0.272727) -- node [above, very near end] {\footnotesize $v_4$} (0.760062, 0.000000);			
		\end{scope}
		\begin{scope}[scale=\tikzscale, xshift=50, yshift=-40]
			\draw (0,0) node {$ (e) $};
		\end{scope}
	\end{tikzpicture}
	\end{scaletikzpicturetowidth}
	\end{center}
	\caption{The degeneration as in case $(4)$ of Lemma \ref{lemma:codim1_degenerations_list} and its regenerations.}\label{fig:degen_collinear_2marked}
\end{figure}

\subsection{Collinear cycle with one marked endpoint}\label{sec:refined_invariance_proof_1marked}
Let $\alpha$ be a combinatorial type as in case $(5)$ Lemma \ref{lemma:codim1_degenerations_list}, and let $V_1$, respectively $V_2$, be the marked, respectively unmarked, vertices of the codimension $1$ fragment in $\alpha$.
We need to consider the $2$ options of the number of edges adjacent to $V_1$ that are contained in the bounded component of $\Gamma\setminus \bq$.

\subsubsection{One edge adjacent to $V_1$ is in the bounded component}\label{sec:type4_inv_case1}
If $V_1$ has one adjacent edge in the bounded component of $\Gamma\setminus\bq$, then $\alpha$ has $4$ regenerations depicted in Figure \ref{fig:degen_collinear_1marked_1bounded_no_collinear}(b-e).
Let $E_1$ be the edge adjacent to $V_1$ that is contained in the bounded component of $\Gamma\setminus\bq$ and let $E_3$ be the edge adjacent to $V_2$ that lie in the bounded component of $\Gamma\setminus(\bq\cup{V_2})$ (i.e. such that the regular orientation of $E_3$ is pointing towards $V_2$).
Denote by $v_0$ the primitive vector in the direction of the codimension $1$ collinear cycle pointing from $V_1$ to $V_2$ and by $v_i$ (for $i=1,3$) the vector with lattice length equal to $w(E_i)$, pointing in the direction of $E_i$ outwards from the codimension $1$ component.
Also denote by $w_1,w_2$ the weights of the edges in the codimension $1$ fragment and $\mu_1:=v_1\wedge v_0$, $\mu_3:=v_3\wedge v_0$.

It is clear that if we move the marked vertex of the codimension $1$ fragment above the line that contains $E_1$ we get the regenerations depicted in Figure \ref{fig:degen_collinear_1marked_1bounded_no_collinear}(b,e), while if we move it above this line we get the regenerations in Figure \ref{fig:degen_collinear_1marked_1bounded_no_collinear}(c,d).
It is thus sufficient to show that
\begin{align}\label{eq:type4_case1_inv}
	& [(w_1+w_2)\mu_1]^{-}_y [0]^{+}_y \left( \varphi^{(1)}_y(w_1, w_2,\mu_3) + [(w_1+w_2)\mu_3]^{-}_y \varphi^{(0)}_y(w_1, w_2) \right) = \\
	= & [w_1 \mu_3]^{-}_y [w_2\mu_3]^{-}_y \left( [w_1 \mu_1]^{-}_y[w_2 \mu_1]^{+}_y +[w_1 \mu_1]^{+}_y [w_2 \mu_1]^{-}_y \right),\nonumber
\end{align}
which is immediate from Lemma \ref{lemma:algebra_of_quantum_weights} and the definition of $\varphi^{(0)}_y, \varphi^{(1)}_y$.

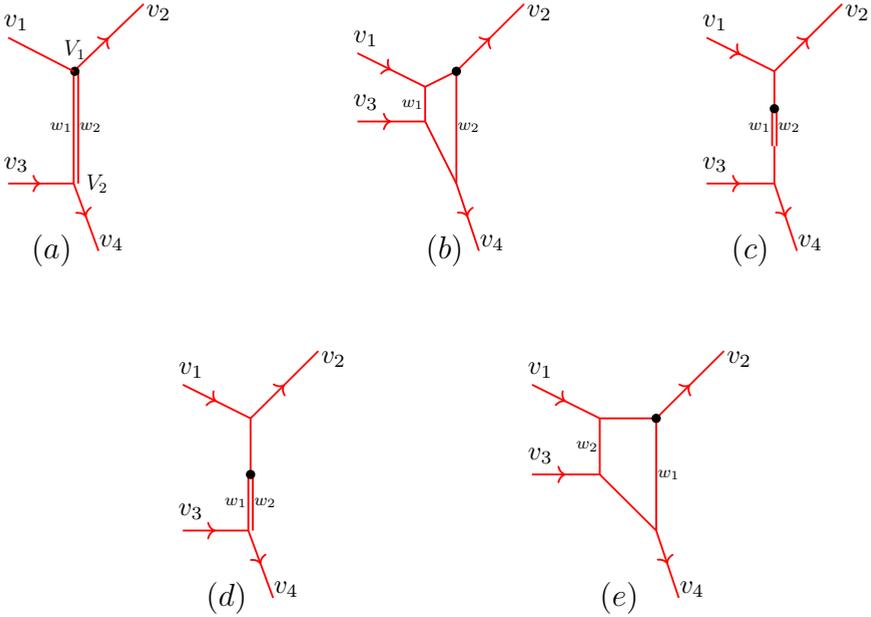
\begin{figure}
	\begin{center}
	\begin{scaletikzpicturetowidth}{\textwidth}
	\begin{tikzpicture}[scale=1.3]
		\tikzset{
			every path/.style = {line width=0.7pt, red},
			every node/.style = {black}
		}
		\begin{scope}[scale=\tikzscale, xshift=0, yshift=-0, decoration={markings,mark=at position 0.5 with {\arrow{>}}}]
			\draw [] (0.281727, 0.727273) -- node [right=-0.1, scale=0.8] {\tiny $w_2$} (0.281727, 0.272727);
			\draw [] (0.263727, 0.727273) -- node [left=-0.1, scale=0.8] {\tiny $w_1$} (0.263727, 0.272727);
			\draw [postaction={decorate}] (0.263727, 0.727273) -- node [right, very near end] {\footnotesize $v_2$} (0.545455, 1.000000);
			\draw [black, fill=black] (0.268227, 0.727273) circle (0.015000) node [above=0.06, scale=0.7] {$V_1$};
			\draw [] (-0.000000, 0.863636) -- node [above, very near start] {\footnotesize $v_1$} (0.263727, 0.727273);
			\draw [postaction={decorate}] (0.263727, 0.272727) -- node [right=-0.1, very near end] {\footnotesize $v_4$} (0.363636, 0.000000);
			\draw [postaction={decorate}] (0.000000, 0.272727) -- node [above, very near start] {\footnotesize $v_3$} (0.263727, 0.272727) node [right=0.06, scale=0.7] {$V_2$};
			
		\end{scope}
		\begin{scope}[scale=\tikzscale, xshift=5, yshift=-0]
			\draw (0,0) node {$ (a) $};
		\end{scope}
		\begin{scope}[scale=\tikzscale, xshift=40, yshift=-0, decoration={markings,mark=at position 0.5 with {\arrow{>}}}]
			\draw [] (0.272727, 0.664456) -- (0.398361, 0.727273);
			\draw [] (0.272727, 0.664456) -- node [left=-0.1, scale=0.8] {\tiny $w_1$} (0.272727, 0.523994);
			\draw [postaction={decorate}] (0.000000, 0.800820) -- node [above, very near start] {\footnotesize $v_1$} (0.272727, 0.664456);
			\draw [] (0.398361, 0.727273) -- node [right=-0.1, scale=0.8] {\tiny $w_2$} (0.398361, 0.272727);
			\draw [postaction={decorate}] (0.398361, 0.727273) -- node [right, very near end] {\footnotesize $v_2$} (0.671088, 1.000000);
			\draw [black, fill=black] (0.398361, 0.727273) circle (0.015000);
			\draw [] (0.272727, 0.523994) -- (0.398361, 0.272727);
			\draw [postaction={decorate}] (0.000000, 0.523994) -- node [above, very near start] {\footnotesize $v_3$} (0.272727, 0.523994);
			\draw [postaction={decorate}] (0.398361, 0.272727) -- node [right=-0.1, very near end] {\footnotesize $v_4$} (0.489270, 0.000000);
			
		\end{scope}
		\begin{scope}[scale=\tikzscale, xshift=50, yshift=-0]
			\draw (0,0) node {$ (b) $};
		\end{scope}
		\begin{scope}[scale=\tikzscale, xshift=80, yshift=-0, decoration={markings,mark=at position 0.5 with {\arrow{>}}}]
			\draw [] (0.272727, 0.575758) -- (0.272727, 0.727273);
			\draw [] (0.281727, 0.575758) -- node [right=-0.1, scale=0.8] {\tiny $w_2$} (0.281727, 0.424242);
			\draw [] (0.263727, 0.575758) -- node [left=-0.1, scale=0.8] {\tiny $w_1$} (0.263727, 0.424242);
			\draw [black, fill=black] (0.272727, 0.575758) circle (0.015000);
			\draw [postaction={decorate}] (0.272727, 0.727273) -- node [right, very near end] {\footnotesize $v_2$} (0.545455, 1.000000);
			\draw [postaction={decorate}] (0.000000, 0.863636) -- node [above, very near start] {\footnotesize $v_1$} (0.272727, 0.727273);
			\draw [] (0.272727, 0.424242) -- (0.272727, 0.272727);
			\draw [postaction={decorate}] (0.272727, 0.272727) -- node [right=-0.1, very near end] {\footnotesize $v_4$} (0.363636, 0.000000);
			\draw [postaction={decorate}] (0.000000, 0.272727) -- node [above, very near start] {\footnotesize $v_3$} (0.272727, 0.272727);
			
		\end{scope}
		\begin{scope}[scale=\tikzscale, xshift=85, yshift=-0]
			\draw (0,0) node {$ (c) $};
		\end{scope}
		
		\begin{scope}[scale=\tikzscale, xshift=20, yshift=-40, decoration={markings,mark=at position 0.5 with {\arrow{>}}}]
			\draw [] (0.272727, 0.500000) -- (0.272727, 0.727273);
			\draw [] (0.281727, 0.500000) -- node [right=-0.1, scale=0.8] {\tiny $w_2$} (0.281727, 0.272727);
			\draw [] (0.263727, 0.500000) -- node [left=-0.1, scale=0.8] {\tiny $w_1$} (0.263727, 0.272727);
			\draw [black, fill=black] (0.272727, 0.500000) circle (0.015000);
			\draw [postaction={decorate}] (0.272727, 0.727273) -- node [right, very near end] {\footnotesize $v_2$} (0.545455, 1.000000);
			\draw [postaction={decorate}] (0.000000, 0.863636) -- node [above, very near start] {\footnotesize $v_1$} (0.272727, 0.727273);
			\draw [postaction={decorate}] (0.263727, 0.272727) -- node [right=-0.1, very near end] {\footnotesize $v_4$} (0.363636, 0.000000);
			\draw [postaction={decorate}] (0.000000, 0.272727) -- node [above, very near start] {\footnotesize $v_3$} (0.263727, 0.272727);
			
		\end{scope}
		\begin{scope}[scale=\tikzscale, xshift=25, yshift=-40]
			\draw (0,0) node {$ (d) $};
		\end{scope}
		\begin{scope}[scale=\tikzscale, xshift=60, yshift=-40, decoration={markings,mark=at position 0.5 with {\arrow{>}}}]
			\draw [] (0.500000, 0.727273) -- (0.272727, 0.727273);
			\draw [] (0.500000, 0.727273) -- node [right=-0.1, scale=0.8] {\tiny $w_1$} (0.500000, 0.272727);
			\draw [postaction={decorate}] (0.500000, 0.727273) -- node [right, very near end] {\footnotesize $v_2$} (0.772727, 1.000000);
			\draw [black, fill=black] (0.500000, 0.727273) circle (0.015000);
			\draw [] (0.272727, 0.727273) -- node [left=-0.1, scale=0.8] {\tiny $w_2$} (0.272727, 0.500000);
			\draw [postaction={decorate}] (0.000000, 0.863636) -- node [above, very near start] {\footnotesize $v_1$} (0.272727, 0.727273);
			\draw [] (0.500000, 0.272727) -- (0.272727, 0.500000);
			\draw [postaction={decorate}] (0.500000, 0.272727) -- node [right=-0.1, very near end] {\footnotesize $v_4$} (0.590909, 0.000000);
			\draw [postaction={decorate}] (0.000000, 0.500000) -- node [above, very near start] {\footnotesize $v_3$} (0.272727, 0.500000);
			
		\end{scope}
		\begin{scope}[scale=\tikzscale, xshift=70, yshift=-40]
			\draw (0,0) node {$ (e) $};
		\end{scope}
	\end{tikzpicture}
	\end{scaletikzpicturetowidth}
	\end{center}
	
	\caption{The degeneration as in case $(5)$ of Lemma \ref{lemma:codim1_degenerations_list} when $V_1$ is adjacent to one edge in the bounded component of $\Gamma\setminus\bq$.}\label{fig:degen_collinear_1marked_1bounded_no_collinear}
\end{figure}

\subsubsection{Two edges adjacent to $V_1$ are in the bounded component}

There are $4$ regenerations in this case, they are depicted in Figure \ref{fig:degen_collinear_1marked_2bounded_no_collinear}(b)-(e).
It is easy to see that the regenerations depicted in Figure \ref{fig:degen_collinear_1marked_2bounded_no_collinear}(b)-(c) are contained in $\mathcal{A}_{+}$, while the regenerations depicted in Figure \ref{fig:degen_collinear_1marked_2bounded_no_collinear}(d)-(e) are contained in $\mathcal{A}_{-}$.
It is also immediate that the refined weight of the curves depicted in Figures \ref{fig:degen_collinear_1marked_2bounded_no_collinear}(b) and \ref{fig:degen_collinear_1marked_2bounded_no_collinear}(e) are equal, and similarly for the curves depicted in Figures \ref{fig:degen_collinear_1marked_2bounded_no_collinear}(c) and \ref{fig:degen_collinear_1marked_2bounded_no_collinear}(d).
We thus get that \eqref{eq:local_refined_invariance} holds.
\begin{figure}
	\begin{center}
	\begin{scaletikzpicturetowidth}{\textwidth}
	\begin{tikzpicture}
		\tikzset{
			every path/.style = {line width=0.7pt, red},
			every node/.style = {black}
		}
		\begin{scope}[scale=\tikzscale, xshift=0, yshift=-0, decoration={markings,mark=at position 0.5 with {\arrow{>}}}]
			\draw [] (0.281727, 0.727273) -- node [right=-0.1, scale=0.8] {\tiny $w_2$} (0.281727, 0.272727);
			\draw [] (0.263727, 0.727273) -- node [left=-0.1, scale=0.8] {\tiny $w_1$} (0.263727, 0.272727);
			\draw [] (0.545455, 1.000000) -- node [right, very near start] {\footnotesize $v_2$} (0.263727, 0.727273);
			\draw [black, fill=black] (0.268227, 0.727273) circle (0.015000) node [above=0.06, scale=0.7] {$V_1$};
			\draw [] (-0.000000, 0.863636) -- node [above, very near start] {\footnotesize $v_1$} (0.263727, 0.727273);
			\draw [postaction={decorate}] (0.263727, 0.272727) -- node [right=-0.1, very near end] {\footnotesize $v_4$} (0.363636, 0.000000);
			\draw [postaction={decorate}] (0.000000, 0.272727) -- node [above, very near start] {\footnotesize $v_3$} (0.263727, 0.272727) node [right=0.06, scale=0.7] {$V_2$};
			
		\end{scope}
		
		\begin{scope}[scale=\tikzscale, xshift=5, yshift=-0]
			\draw (0,0) node {$ (a) $};
		\end{scope}
		
		\begin{scope}[scale=\tikzscale, xshift=40, yshift=-0, decoration={markings,mark=at position 0.5 with {\arrow{>}}}]
			\draw [] (0.500000, 0.727273) -- (0.272727, 0.727273);
			\draw [] (0.500000, 0.727273) -- node [right=-0.1, scale=0.8] {\tiny $w_2$} (0.500000, 0.272727);
			\draw [postaction={decorate}] (0.772727, 1.000000) -- node [right, very near start] {\footnotesize $v_2$} (0.500000, 0.727273);
			\draw [] (0.272727, 0.727273) -- node [left=-0.1, scale=0.8] {\tiny $w_1$} (0.272727, 0.500000);
			\draw [black, fill=black] (0.272727, 0.727273) circle (0.015000);
			\draw [postaction={decorate}] (0.272727, 0.727273) -- node [above, very near end] {\footnotesize $v_1$} (0.000000, 0.863636);
			\draw [] (0.500000, 0.272727) -- (0.272727, 0.500000);
			\draw [postaction={decorate}] (0.500000, 0.272727) -- node [right=-0.1, very near end] {\footnotesize $v_4$} (0.590909, 0.000000);
			\draw [postaction={decorate}] (0.000000, 0.500000) -- node [above, very near start] {\footnotesize $v_3$} (0.272727, 0.500000);
			
		\end{scope}
		
		\begin{scope}[scale=\tikzscale, xshift=50, yshift=-0]
			\draw (0,0) node {$ (b) $};
		\end{scope}
		
		\begin{scope}[scale=\tikzscale, xshift=80, yshift=-0, decoration={markings,mark=at position 0.5 with {\arrow{>}}}]
			\draw [] (0.272727, 0.664456) -- (0.398361, 0.727273);
			\draw [] (0.272727, 0.664456) -- node [left=-0.1, scale=0.8] {\tiny $w_2$} (0.272727, 0.523994);
			\draw [black, fill=black] (0.272727, 0.664456) circle (0.015000);
			\draw [postaction={decorate}] (0.272727, 0.664456) -- node [above, very near end] {\footnotesize $v_1$} (0.000000, 0.800820);
			\draw [] (0.398361, 0.727273) -- node [right=-0.1, scale=0.8] {\tiny $w_1$} (0.398361, 0.272727);
			\draw [postaction={decorate}] (0.671088, 1.000000) -- node [right, very near start] {\footnotesize $v_2$} (0.398361, 0.727273);
			\draw [] (0.272727, 0.523994) -- (0.398361, 0.272727);
			\draw [postaction={decorate}] (0.000000, 0.523994) -- node [above, very near start] {\footnotesize $v_3$} (0.272727, 0.523994);
			\draw [postaction={decorate}] (0.398361, 0.272727) -- node [right=-0.1, very near end] {\footnotesize $v_4$} (0.489270, 0.000000);
			
		\end{scope}
		
		\begin{scope}[scale=\tikzscale, xshift=90, yshift=-0]
			\draw (0,0) node {$ (c) $};
		\end{scope}
		
		\begin{scope}[scale=\tikzscale, xshift=20, yshift=-40, decoration={markings,mark=at position 0.5 with {\arrow{>}}}]
			\draw [] (0.500000, 0.727273) -- (0.272727, 0.727273);
			\draw [] (0.500000, 0.727273) -- node [right=-0.1, scale=0.8] {\tiny $w_2$} (0.500000, 0.272727);
			\draw [postaction={decorate}] (0.500000, 0.727273) -- node [right, very near end] {\footnotesize $v_2$} (0.772727, 1.000000);
			\draw [black, fill=black] (0.500000, 0.727273) circle (0.015000);
			\draw [] (0.272727, 0.727273) -- node [left=-0.1, scale=0.8] {\tiny $w_1$} (0.272727, 0.500000);
			\draw [postaction={decorate}] (0.000000, 0.863636) -- node [above, very near start] {\footnotesize $v_1$} (0.272727, 0.727273);
			\draw [] (0.500000, 0.272727) -- (0.272727, 0.500000);
			\draw [postaction={decorate}] (0.500000, 0.272727) -- node [right=-0.1, very near end] {\footnotesize $v_4$} (0.590909, 0.000000);
			\draw [postaction={decorate}] (0.000000, 0.500000) -- node [above, very near start] {\footnotesize $v_3$} (0.272727, 0.500000);
			
		\end{scope}
	
		\begin{scope}[scale=\tikzscale, xshift=30, yshift=-40]
			\draw (0,0) node {$ (d) $};
		\end{scope}
		
		\begin{scope}[scale=\tikzscale, xshift=60, yshift=-40, decoration={markings,mark=at position 0.5 with {\arrow{>}}}]
			\draw [] (0.272727, 0.664456) -- (0.398361, 0.727273);
			\draw [] (0.272727, 0.664456) -- node [left=-0.1, scale=0.8] {\tiny $w_2$} (0.272727, 0.523994);
			\draw [postaction={decorate}] (0.000000, 0.800820) -- node [above, very near start] {\footnotesize $v_1$} (0.272727, 0.664456);
			\draw [] (0.398361, 0.727273) -- node [right=-0.1, scale=0.8] {\tiny $w_1$} (0.398361, 0.272727);
			\draw [postaction={decorate}] (0.398361, 0.727273) -- node [right, very near end] {\footnotesize $v_2$} (0.671088, 1.000000);
			\draw [black, fill=black] (0.398361, 0.727273) circle (0.015000);
			\draw [] (0.272727, 0.523994) -- (0.398361, 0.272727);
			\draw [postaction={decorate}] (0.000000, 0.523994) -- node [above, very near start] {\footnotesize $v_3$} (0.272727, 0.523994);
			\draw [postaction={decorate}] (0.398361, 0.272727) -- node [right=-0.1, very near end] {\footnotesize $v_4$} (0.489270, 0.000000);
			
		\end{scope}
	
		\begin{scope}[scale=\tikzscale, xshift=65, yshift=-40]
			\draw (0,0) node {$ (e) $};
		\end{scope}
	\end{tikzpicture}
	\end{scaletikzpicturetowidth}
	\end{center}
	
	\caption{The degeneration as in case $(5)$ of Lemma \ref{lemma:codim1_degenerations_list} when $V_1$ is adjacent to two edges in the bounded component of $\Gamma\setminus\bq$.}\label{fig:degen_collinear_1marked_2bounded_no_collinear}
\end{figure}
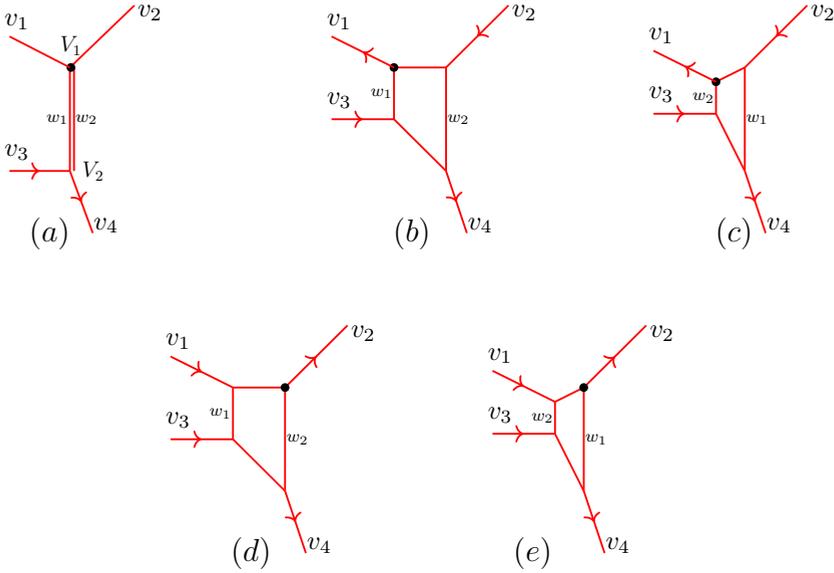

\subsection{Collinear cycle with no marked endpoint}\label{sec:refined_invariance_proof_no_marked}
Suppose that $\alpha$ is as in case $(6) $ of Lemma \ref{lemma:codim1_degenerations_list}.
We depict the degeneration in Figure \ref{fig:degen_collinear_2unmarked_1_collinear}(a), and its regenerations in Figure \ref{fig:degen_collinear_2unmarked_1_collinear}(b-e).
The regenerations depicted in Figure \ref{fig:degen_collinear_2unmarked_1_collinear}(b, e) are contained in $\mathcal{A}_{+}$, while the regenerations depicted in Figure \ref{fig:degen_collinear_2unmarked_1_collinear}(c, d) are contained in $\mathcal{A}_{-}$.
Denote by $v_0$ the primitive vector parallel to the codimension $1$ fragment, by $v_1$ the primitive vector parallel to the additional collinear cycle, and by $\mu:=v_0\wedge v_1$.
Writing the refined weights and canceling out equal terms gives us
\begin{align}\label{eq:type5_collinear_invariance}
	& \varphi^{(1)}_y(w_1', w_2', w_2\mu)\cdot \left[ w_1(w_1'+w_2')\mu \right]^{-}_y + \left[w_1 w_1' \mu\right]^{-}_y\cdot \left[ w_2 w_2' \mu \right]^{-}_y \cdot \left[ (w_1w_2' - w_1'w_2)\mu\right]^{-}_y =\\
	= & \varphi^{(1)}_y(w_1', w_2', w_1\mu)\cdot \left[w_2(w_1'+w_2')\mu\right]^{-}_y + \left[w_1w_2'\mu\right]^{-}_y\cdot \left[w_1'w_2\mu\right]^{-}_y \cdot \left[(w_2w_2'-w_1w_1')\mu\right]^{-}_y , \nonumber
\end{align}
which is straightforward using Lemma \ref{lemma:algebra_of_quantum_weights} and the definition of $\varphi^{(1)}$.

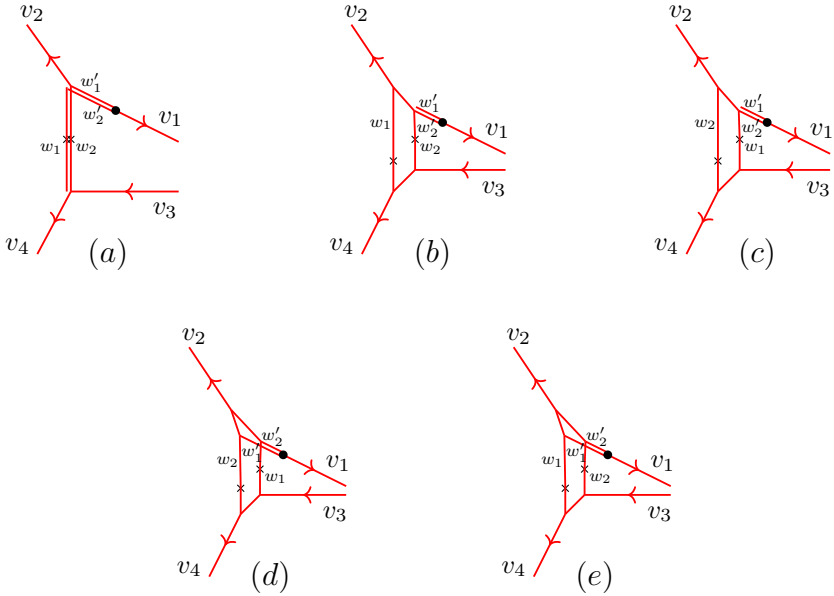
\begin{figure}
	\begin{center}
	\begin{scaletikzpicturetowidth}{\textwidth}
	\begin{tikzpicture}
		\tikzset{
			every path/.style = {line width=0.7pt, red},
			every node/.style = {black}
		}
		\begin{scope}[scale=\tikzscale, xshift=0, yshift=-0, decoration={markings,mark=at position 0.5 with {\arrow{>}}}]
			\draw [postaction={decorate}] (0.385097, 0.625633) -- node [above, very near end] {\footnotesize $v_1$} (0.657824, 0.489270);
			\draw [postaction={decorate}] (0.190818, 0.732835) -- node [above, very near end] {\footnotesize $v_2$} (-0.000000, 1.000000);
			\draw [postaction={decorate}] (0.657824, 0.272727) -- node [below, very near start] {\footnotesize $v_3$} (0.190818, 0.272727);
			\draw [postaction={decorate}] (0.190818, 0.272727) -- node [left, very near end] {\footnotesize $v_4$} (0.045455, -0.000000);
			\draw [black, thin] (0.157818, 0.485000) -- (0.187818, 0.515000);
			\draw [black, thin] (0.157818, 0.515000) -- (0.187818, 0.485000);
			\draw [black, thin] (0.175818, 0.485000) -- (0.205818, 0.515000);
			\draw [black, thin] (0.175818, 0.515000) -- (0.205818, 0.485000);
			\draw [] (0.172818, 0.500000) -- node [below left=-0.08, scale=0.8, pos=0] {\tiny $w_1$} (0.172818, 0.727273);
			\draw [] (0.172818, 0.500000) -- (0.172818, 0.272727);
			\draw [] (0.190818, 0.500000) -- node [below right=-0.08, scale=0.8, pos=0] {\tiny $w_2$} (0.190818, 0.732835);
			\draw [] (0.190818, 0.500000) -- (0.190818, 0.272727);
			\draw [] (0.177793, 0.719223) -- node [above, scale=0.8] {\tiny $w'_1$} (0.381072, 0.617583);
			\draw [] (0.190818, 0.732835) -- node [below, scale=0.8] {\tiny $w'_2$} (0.389122, 0.633683);
			\draw [black, fill=black] (0.385097, 0.625633) circle (0.015000);
			
		\end{scope}
		\begin{scope}[scale=\tikzscale, xshift=10, yshift=-0]
			\draw (0,0) node {$ (a) $};
		\end{scope}
		\begin{scope}[scale=\tikzscale, xshift=40, yshift=-0, decoration={markings,mark=at position 0.5 with {\arrow{>}}}]
			\draw [postaction={decorate}] (0.395036, 0.573594) -- node [above, very near end] {\footnotesize $v_1$} (0.667763, 0.437231);
			\draw [postaction={decorate}] (0.181818, 0.727273) -- node [above, very near end] {\footnotesize $v_2$} (0.000000, 1.000000);
			\draw [postaction={decorate}] (0.667763, 0.366867) -- node [below, very near start] {\footnotesize $v_3$} (0.275958, 0.366867);
			\draw [postaction={decorate}] (0.181818, 0.272727) -- node [left, very near end] {\footnotesize $v_4$} (0.045455, -0.000000);
			\draw [black, thin] (0.166818, 0.390861) -- (0.196818, 0.420861);
			\draw [black, thin] (0.166818, 0.420861) -- (0.196818, 0.390861);
			\draw [black, thin] (0.260958, 0.485000) -- (0.290958, 0.515000);
			\draw [black, thin] (0.260958, 0.515000) -- (0.290958, 0.485000);
			\draw [] (0.181818, 0.405861) -- node [left=-0.1, scale=0.8] {\tiny $w_1$} (0.181818, 0.727273);
			\draw [] (0.181818, 0.405861) -- (0.181818, 0.272727);
			\draw [] (0.275958, 0.500000) -- node [below right=-0.08, scale=0.8, pos=0] {\tiny $w_2$} (0.271933, 0.625083);
			\draw [] (0.275958, 0.500000) -- (0.275958, 0.366867);
			\draw [] (0.181818, 0.727273) -- (0.271933, 0.625083);
			\draw [] (0.279983, 0.641183) -- node [above=-0.1, scale=0.8] {\tiny $w'_1$} (0.399061, 0.581644);
			\draw [] (0.271933, 0.625083) -- node [below=-0.1, scale=0.8] {\tiny $w'_2$} (0.391011, 0.565544);
			\draw [] (0.181818, 0.272727) -- (0.275958, 0.366867);
			\draw [black, fill=black] (0.395036, 0.573594) circle (0.015000);
			
		\end{scope}
		\begin{scope}[scale=\tikzscale, xshift=50, yshift=-0]
			\draw (0,0) node {$ (b) $};
		\end{scope}
		\begin{scope}[scale=\tikzscale, xshift=80, yshift=-0, decoration={markings,mark=at position 0.5 with {\arrow{>}}}]
			\draw [postaction={decorate}] (0.395036, 0.573594) -- node [above, very near end] {\footnotesize $v_1$} (0.667763, 0.437231);
			\draw [postaction={decorate}] (0.181818, 0.727273) -- node [above, very near end] {\footnotesize $v_2$} (-0.000000, 1.000000);
			\draw [postaction={decorate}] (0.667763, 0.366867) -- node [below, very near start] {\footnotesize $v_3$} (0.275958, 0.366867);
			\draw [postaction={decorate}] (0.181818, 0.272727) -- node [left, very near end] {\footnotesize $v_4$} (0.045455, 0.000000);
			\draw [black, thin] (0.260958, 0.485000) -- (0.290958, 0.515000);
			\draw [black, thin] (0.260958, 0.515000) -- (0.290958, 0.485000);
			\draw [black, thin] (0.166818, 0.390861) -- (0.196818, 0.420861);
			\draw [black, thin] (0.166818, 0.420861) -- (0.196818, 0.390861);
			\draw [] (0.275958, 0.500000) -- node [below right=-0.08, scale=0.8, pos=0] {\tiny $w_1$} (0.271933, 0.625083);
			\draw [] (0.275958, 0.500000) -- (0.275958, 0.366867);
			\draw [] (0.181818, 0.405861) -- node [left=-0.1, scale=0.8] {\tiny $w_2$} (0.181818, 0.727273);
			\draw [] (0.181818, 0.405861) -- (0.181818, 0.272727);
			\draw [] (0.181818, 0.727273) -- (0.271933, 0.625083);
			\draw [] (0.279983, 0.641183) -- node [above=-0.1, scale=0.8] {\tiny $w'_1$} (0.399061, 0.581644);
			\draw [] (0.271933, 0.625083) -- node [below=-0.1, scale=0.8] {\tiny $w'_2$} (0.391011, 0.565544);
			\draw [] (0.181818, 0.272727) -- (0.275958, 0.366867);
			\draw [black, fill=black] (0.395036, 0.573594) circle (0.015000);
			
		\end{scope}
		\begin{scope}[scale=\tikzscale, xshift=90, yshift=-0]
			\draw (0,0) node {$ (c) $};
		\end{scope}
	
		\begin{scope}[scale=\tikzscale, xshift=20, yshift=-40, decoration={markings,mark=at position 0.5 with {\arrow{>}}}]
			\draw [postaction={decorate}] (0.407660, 0.530835) -- node [above, very near end] {\footnotesize $v_1$} (0.680387, 0.394472);
			\draw [postaction={decorate}] (0.181818, 0.727273) -- node [above, very near end] {\footnotesize $v_2$} (0.000000, 1.000000);
			\draw [postaction={decorate}] (0.680387, 0.356244) -- node [below, very near start] {\footnotesize $v_3$} (0.307093, 0.356244);
			\draw [postaction={decorate}] (0.223576, 0.272727) -- node [left, very near end] {\footnotesize $v_4$} (0.087213, 0.000000);
			\draw [black, thin] (0.292093, 0.453681) -- (0.322093, 0.483681);
			\draw [black, thin] (0.292093, 0.483681) -- (0.322093, 0.453681);
			\draw [black, thin] (0.208576, 0.370165) -- (0.238576, 0.400165);
			\draw [black, thin] (0.208576, 0.400165) -- (0.238576, 0.370165);
			\draw [] (0.307093, 0.468681) -- node [below right=-0.08, pos=0, scale=0.8] {\tiny $w_1$} (0.311118, 0.589169);
			\draw [] (0.307093, 0.468681) -- (0.307093, 0.356244);
			\draw [] (0.223576, 0.385165) -- node [left=-0.1, scale=0.8] {\tiny $w_2$} (0.219552, 0.614827);
			\draw [] (0.223576, 0.385165) -- (0.223576, 0.272727);
			\draw [] (0.181818, 0.727273) -- (0.219552, 0.614827);
			\draw [] (0.181818, 0.727273) -- (0.311118, 0.589169);
			\draw [] (0.219552, 0.614827) -- node [below=-0.1, scale=0.8, pos=0.3] {\tiny $w'_1$} (0.403635, 0.522785);
			\draw [] (0.311118, 0.589169) -- node [above=-0.1, scale=0.8] {\tiny $w'_2$} (0.411685, 0.538885);
			\draw [] (0.223576, 0.272727) -- (0.307093, 0.356244);
			\draw [black, fill=black] (0.407660, 0.530835) circle (0.015000);
			
		\end{scope}
		\begin{scope}[scale=\tikzscale, xshift=30, yshift=-40]
			\draw (0,0) node {$ (d) $};
		\end{scope}
	
		\begin{scope}[scale=\tikzscale, xshift=60, yshift=-40, decoration={markings,mark=at position 0.5 with {\arrow{>}}}]
			\draw [postaction={decorate}] (0.407660, 0.530835) -- node [above, very near end] {\footnotesize $v_1$} (0.680387, 0.394472);
			\draw [postaction={decorate}] (0.181818, 0.727273) -- node [above, very near end] {\footnotesize $v_2$} (-0.000000, 1.000000);
			\draw [postaction={decorate}] (0.680387, 0.356244) -- node [below, very near start] {\footnotesize $v_3$} (0.307093, 0.356244);
			\draw [postaction={decorate}] (0.223576, 0.272727) -- node [left, very near end] {\footnotesize $v_4$} (0.087213, 0.000000);
			\draw [black, thin] (0.208576, 0.370165) -- (0.238576, 0.400165);
			\draw [black, thin] (0.208576, 0.400165) -- (0.238576, 0.370165);
			\draw [black, thin] (0.292093, 0.453681) -- (0.322093, 0.483681);
			\draw [black, thin] (0.292093, 0.483681) -- (0.322093, 0.453681);
			\draw [] (0.223576, 0.385165) -- node [left=-0.1, scale=0.8] {\tiny $w_1$} (0.219552, 0.614827);
			\draw [] (0.223576, 0.385165) -- (0.223576, 0.272727);
			\draw [] (0.307093, 0.468681) -- node [below right=-0.08, scale=0.8, pos=0] {\tiny $w_2$} (0.311118, 0.589169);
			\draw [] (0.307093, 0.468681) -- (0.307093, 0.356244);
			\draw [] (0.181818, 0.727273) -- (0.219552, 0.614827);
			\draw [] (0.181818, 0.727273) -- (0.311118, 0.589169);
			\draw [] (0.219552, 0.614827) -- node [below=-0.1, scale=0.8, pos=0.3] {\tiny $w'_1$} (0.403635, 0.522785);
			\draw [] (0.311118, 0.589169) -- node [above=-0.1, scale=0.8] {\tiny $w'_2$} (0.411685, 0.538885);
			\draw [] (0.223576, 0.272727) -- (0.307093, 0.356244);
			\draw [black, fill=black] (0.407660, 0.530835) circle (0.015000);
			
		\end{scope}
		\begin{scope}[scale=\tikzscale, xshift=70, yshift=-40]
			\draw (0,0) node {$ (e) $};
		\end{scope}
	\end{tikzpicture}
	\end{scaletikzpicturetowidth}
	\end{center}
	\caption{(a) The degeneration as in case (6) of Lemma \ref{lemma:codim1_degenerations_list}. (b-e) Its regenerations.}\label{fig:degen_collinear_2unmarked_1_collinear}
\end{figure}

\subsection{Additional marked point}\label{sec:refined_invariance_proof_additional_marked}
The case $ (7)$ of Lemma \ref{lemma:codim1_degenerations_list} was dealt with in \cite[Section 4.1]{SS}.
In this case $\alpha $ has 2 regenerations depicted in Figure \ref{fig:degen_extra_marked_vert}(b)-(c), they have the same refined weight, and exactly one of them is contained in $\mathcal{A}_{+}$ while the other is in $\mathcal{A}_{-}$.
If we replace the bounded component in $\alpha$ by a bounded collinear cycle we get the same analysis.

\begin{figure}
	\begin{tikzpicture}
		\tikzset{
			every path/.style = {line width=0.7pt, red},
			every node/.style = {black}
		}
		\begin{scope}[scale=1, decoration={markings,mark=at position 0.5 with {\arrow{>}}}]
			\draw [] (-2,0) -- node [very near start, above] {\footnotesize $ v_1 $} (0, 0);
			\draw [postaction={decorate}] (0,0) -- node [very near end, right] {\footnotesize $ v_2$} (1, -1.73);
			\draw [postaction={decorate}] (0,0) -- node [very near end, right] {\footnotesize $ v_3 $} (1, 1.73);
			
			\draw [black] (-0.1, -0.07) -- (0.1, 0.07);
			\draw [black] (-0.1, 0.07) -- (0.1, -0.07);
			
			\draw (0, -1.9) node { $(a)$};
		\end{scope}
		
		\begin{scope}[xshift=120, scale=1, decoration={markings,mark=at position 0.5 with {\arrow{>}}}]
			\draw [postaction={decorate}] (-2,0) -- node [very near start, above] {\footnotesize $ v_1 $} (0, 0);
			\draw [postaction={decorate}] (0,0) -- node [very near end, right] {\footnotesize $ v_2$} (1, -1.73);
			\draw [postaction={decorate}] (0,0) -- node [very near end, right] {\footnotesize $ v_3 $} (1, 1.73);
			
			\draw [black] (0.2, -0.59) -- (0.4, -0.45);
			\draw [black] (0.2, -0.45) -- (0.4, -0.59);
			
			\draw (0, -1.9) node { $(b)$};
		\end{scope}
		
		\begin{scope}[xshift=240, scale=1, decoration={markings,mark=at position 0.5 with {\arrow{>}}}]
			\draw [postaction={decorate}] (-2,0) -- node [very near start, above] {\footnotesize $ v_1 $} (0, 0);
			\draw [postaction={decorate}] (0,0) -- node [very near end, right] {\footnotesize $ v_2$} (1, -1.73);
			\draw [postaction={decorate}] (0,0) -- node [very near end, right] {\footnotesize $ v_3 $} (1, 1.73);
			
			\draw [black] (0.2, 0.59) -- (0.4, 0.45);
			\draw [black] (0.2, 0.45) -- (0.4, 0.59);
			
			\draw (0, -1.9) node { $(c)$};
		\end{scope}
	\end{tikzpicture}
	\caption{The degeneration as in case (7) of Lemma \ref{lemma:codim1_degenerations_list} and its regenerations.}\label{fig:degen_extra_marked_vert}
\end{figure}
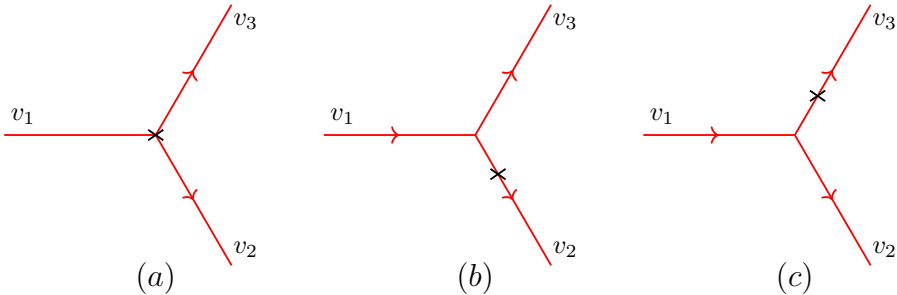

\subsection{Elliptic vertex}\label{sec:refined_invariance_proof_smaller_genus}
Suppose $\alpha$ has an elliptic vertex $V$ as in case $(8)$ of Lemma \ref{lemma:codim1_degenerations_list}.
$V$ has either one or two adjacent edges contained in the unique bounded component.
The regenerations in those cases are the same locally and differ only by the regular orientation, so it is enough to consider the former case.
The degeneration $\alpha$ and possible regenerations in this case are depicted in Figure \ref{fig:degen_genus_reduction_1edge}.
The regenerations in Figure \ref{fig:degen_genus_reduction_1edge}  should all be understood as families of regenerations, where in the regenerations depicted in Figure \ref{fig:degen_genus_reduction_1edge}(b,c) we can vary the weights of the edges in the collinear cycle, and in the regenerations depicted in Figure \ref{fig:degen_genus_reduction_1edge}(d,e) we can vary the subdivision of the triangle dual to $V$ into $3$ triangles (see Figure \ref{fig:degen_genus_reduction_triangle_split}).

Let $E_1$ and $E_2$ be the edges adjacent to the contracted cycle that are outside the bounded component of $\Gamma \setminus \bq$. For each $i \in \{1,2\}$, let $v_i$ denote the primitive direction vector of $E_i$, $w_i$ the weight of $E_i$, and define $\nu := |v_1 \wedge v_2|$. When the marking in Figure~\ref{fig:degen_genus_reduction_1edge}(a) is moved downwards compared to the edge in the bounded component, the resulting regenerations correspond to those shown in Figure~\ref{fig:degen_genus_reduction_1edge}(b,d); we denote this collection by $\mathcal{A}_{+}$. Conversely, moving the marking upwards (above the bounded edge) leads to the regenerations depicted in Figure~\ref{fig:degen_genus_reduction_1edge}(c,e), which we denote by $\mathcal{A}_{-}$.
We thus need to show that
\begin{align*}
	& \frac{1}{2} \sum_{k=1}^{w_1}[0]^{+}_y\varphi^{(1)}_y(k, w_1 - k, w_2\nu) + \sum_{x\in \Int(\mathcal{D}(V))\cap \mathbb{Z}^2}([\mu_0(x)]^{-}_y\cdot [\mu_1(x)]^{+}_y \cdot [\mu_2(x)]^{-}_y) =\\
	= &  \frac{1}{2}\sum_{k=1}^{w_1}[0]^{+}_y\varphi^{(1)}_y(k, w_2 - k, w_1\nu) + \sum_{x\in \Int(\mathcal{D}(V))\cap \mathbb{Z}^2}([\mu_0(x)]^{-}_y\cdot [\mu_1(x)]^{-}_y \cdot [\mu_2(x)]^{+}_y)
\end{align*}
where $\mathcal{D}(V)$ is the dual triangle to $V$, $\mu_i(x)$ (for $i=1,2$) is the lattice area of the triangle in the subdivision induced by $x$ dual to the vertex adjacent to $E_i$, and $\mu_0(x)$ is the lattice area of the triangle dual to the vertex not adjacent to neither $E_1$ nor $E_2$ (see Figure \ref{fig:degen_genus_reduction_triangle_split}).
After rearranging and using Remark \ref{rem:summing_over_weights_get_Psi} we are left to show that

\begin{gather*}
	\sum_{x \in \Int(\mathcal{D}(V))} [\mu_0(x)]^{-}_y\left( [\mu_1(x)]^{+}_y[\mu_2(x)]^{-}_y - [\mu_1(x)]^{-}_y[\mu_2(x)]^{+}_y \right) =\\
	= \Phi^{(1)}_y(w_2, w_1\nu) - \Phi^{(1)}_y(w_1, w_2\nu) = \\
	= 2\frac{[w_1\nu]_{y^{w_2}}^{-}[w_2-1]_{y}^{-} - [w_2-1]_{y^{w_1\nu}}^{-} - (w_2-1)[w_1\nu-1]_{y^{w_2}}^{-}}{(y^{1/2}+y^{-1/2})(y^{1/2}-y^{-1/2})^2} -\\
	- 2\frac{[w_2\nu]_{y^{w_1}}^{-}[w_1-1]_{y}^{-} - [w_1-1]_{y^{w_2\nu}}^{-} - (w_1-1)[w_2\nu-1]_{y^{w_1}}^{-}}{(y^{1/2}+y^{-1/2})(y^{1/2}-y^{-1/2})^2}.
\end{gather*}

By \cite[Lemma 4.1]{SS},

\begin{align*}
	\sum_{x \in \Int(\mathcal{D}(V))} [\mu_0(x)]^{-}_y\left( [\mu_1(x)]^{+}_y[\mu_2(x)]^{-}_y - [\mu_1(x)]^{-}_y[\mu_2(x)]^{+}_y \right) = \\
	2 \frac{w_2 [w_1\nu -1]_{y^{w_2}}^{-} + [w_2-1]_{y^{w_1\nu}}^{-} - w_1 [w_2\nu -1]_{y^{w_1}}^{-} - [w_1-1]_{y^{w_2\nu}}^{-} }{(y^{1/2}-y^{-1/2})^2(y^{1/2}+y^{-1/2})}
\end{align*}

so it is enough to show that
\begin{equation*}
	[w_1\nu]_{y^{w_2}}^{-}[w_2-1]_y^{-} + [w_1\nu - 1]_{y^{w_2}}^{-} - [w_2\nu]_{y^{w_1}}^{-}[w_1-1]_y^{-} - [w_2\nu - 1]_{y^{w_1}}^{-} = 0
\end{equation*}
which is straightforward.

\begin{figure}
	\begin{center}
	\begin{scaletikzpicturetowidth}{\textwidth}
	\begin{tikzpicture}
		\tikzset{
			every path/.style = {line width=0.7pt, red},
			every node/.style = {black}
		}
		
		\begin{scope}[scale=\tikzscale, decoration={markings,mark=at position 0.5 with {\arrow{>}}}]
			\draw [] (-2,0) -- node [very near start, above] {\footnotesize $ v_3 $} (0, 0);
			\draw [postaction={decorate}] (0,0) -- node [very near end, right] {\footnotesize $ v_1$} (1, -1.73);
			\draw [postaction={decorate}] (0,0) -- node [very near end, right] {\footnotesize $ v_2 $} (1, 1.73);
			
			\draw [black, fill=black] (0,0) circle (0.05);
			
			\draw (0, -1.9) node { $(a)$};
		\end{scope}
		
		\begin{scope}[xshift=140, scale=\tikzscale, decoration={markings,mark=at position 0.5 with {\arrow{>}}}]
			\draw [postaction={decorate}] (-2,0) -- node [very near start, above] {\footnotesize $ v_3 $} (0, 0);
			\draw [postaction={decorate}] (0.5, -0.865) -- node [very near end, right] {\footnotesize $ v_1$} (1, -1.73);
			\draw [postaction={decorate}] (0,0) -- node [very near end, right] {\footnotesize $ v_2 $} (1, 1.73);
			
			\draw [black, fill=black] (0.5, -0.865) circle (0.05);
			\draw [] (0.04325, 0.025) -- (0.54325, -0.84);
			\draw [] (-0.04325, -0.025) -- (0.45675, -0.89);
			
			\draw (0, -1.9) node { $(b)$};
		\end{scope}
		
		\begin{scope}[xshift=280, scale=\tikzscale, decoration={markings,mark=at position 0.5 with {\arrow{>}}}]
			\draw [postaction={decorate}] (-2,0) -- node [very near start, above] {\footnotesize $ v_3 $} (0, 0);
			\draw [postaction={decorate}] (0,0) -- node [very near end, right] {\footnotesize $ v_1$} (1, -1.73);
			\draw [postaction={decorate}] (0.5, 0.865) -- node [very near end, right] {\footnotesize $ v_2 $} (1, 1.73);
			
			\draw [black, fill=black] (0.5, 0.865) circle (0.05);
			\draw [] (0.04325, -0.025) -- (0.54325, 0.84);
			\draw [] (-0.04325, 0.025) -- (0.45675, 0.89);
			
			\draw (0, -1.9) node { $(c)$};
		\end{scope}
		
		\begin{scope}[xshift=70, yshift=-140, scale=\tikzscale, decoration={markings,mark=at position 0.5 with {\arrow{>}}}]
			\draw [postaction={decorate}] (-2,0) -- node [very near start, above] {\footnotesize $ v_3 $} (-0.8, 0);
			\draw [postaction={decorate}] (0.5, -0.865) -- node [very near end, right] {\footnotesize $ v_1$} (1, -1.73);
			\draw [postaction={decorate}] (0.5, 0.865) -- node [very near end, right] {\footnotesize $ v_2 $} (1, 1.73);
			
			\draw [] (-0.8, 0) -- (0.5, -0.865) -- (0.5, 0.865) -- (-0.8, 0);
			
			\draw [black, fill=black] (0.5, -0.865) circle (0.05);
			
			\draw (0, -1.9) node { $(d)$};
		\end{scope}
		
		\begin{scope}[xshift=210, yshift=-140, scale=\tikzscale, decoration={markings,mark=at position 0.5 with {\arrow{>}}}]
			\draw [postaction={decorate}] (-2,0) -- node [very near start, above] {\footnotesize $ v_3 $} (-0.8, 0);
			\draw [postaction={decorate}] (0.5, -0.865) -- node [very near end, right] {\footnotesize $ v_1$} (1, -1.73);
			\draw [postaction={decorate}] (0.5, 0.865) -- node [very near end, right] {\footnotesize $ v_2 $} (1, 1.73);
			
			\draw [black, fill=black] (0.5, 0.865) circle (0.05);
			\draw [] (-0.8, 0) -- (0.5, -0.865) -- (0.5, 0.865) -- (-0.8, 0);
			
			\draw (0, -1.9) node { $(e)$};
		\end{scope}
		
	\end{tikzpicture}
\end{scaletikzpicturetowidth}
\end{center}
	\caption{The degeneration as in the case $(8)$ of Lemma \ref{lemma:codim1_degenerations_list} for no bounded collinear cycle and $1$ edge of the bounded component of $\Gamma\setminus\bq$ adjacent to the image of the contracted cycle.}\label{fig:degen_genus_reduction_1edge}
\end{figure}
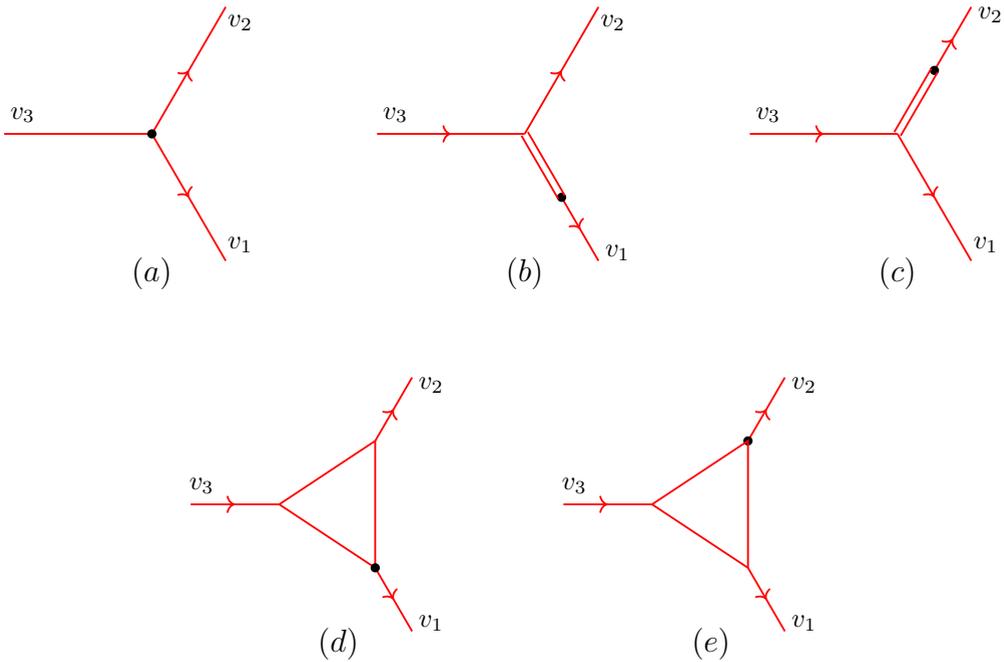

\begin{figure}
	\begin{tikzpicture}[scale=2]
		\tikzset{
			every path/.style = {line width=0.7pt, blue},
			every node/.style = {black}
		}
		\draw (0,1) -- (0,-1) -- (1.7, 0) -- (0, 1);
		\draw (0,1) -- (0.8,0);
		\draw (0, -1) -- (0.8,0);
		\draw (0.8,0) -- (1.7, 0);
		
		\draw (1, -0.2) node {$\mu_1(x)$};
		\draw (1, 0.2) node {$\mu_2(x)$};
		\draw (0.4, 0) node {$\mu_0(x)$};
	\end{tikzpicture}
	
	\caption{A separation of the triangle dual to $V$ into $3$ triangles.}\label{fig:degen_genus_reduction_triangle_split}
\end{figure}
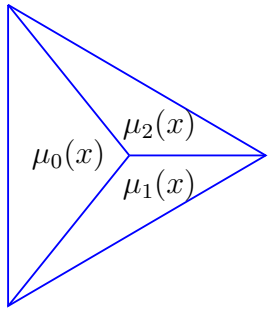

\subsection{Refined invariance for point conditions in Mikhalkin position}\label{pt6.7}
\begin{proof}[Proof of Theorem \ref{tcn4}]
	{\it Step 1.} To begin with, we shall precise the choice of the parameters $\eps$ and $M_1,...,M_n$ (where $n=n_1+n_2$).
	
	Denote by $A\subset\Z^2$ the set of all nonzero vectors $z_1-z_2$, $z_1,z_2\in P\cap\Z^2$, rotated by $\frac{\pi}{2}$.
	
	First, we fix $\eps_0>0$ which is less than the slope of any line directed by a non-horizontal vector $\ba\in A$.
	Then, for any $\eps\ne0$, $|\eps|<\eps_0$, the linear functional $x+\eps y$ is injective on $P\cap\Z^2$ and defines on this set the order
	$$(i,j)\prec(i',j')\ \Longleftrightarrow\ \begin{cases}i<i',\ \text{or}&\\ i=i',\ j<j',&\end{cases}\quad\text{if}\ \eps>0,$$
	$$(i,j)\prec(i',j')\ \Longleftrightarrow\ \begin{cases}i<i',\ \text{or}&\\ i=i',\ j>j',&\end{cases}\quad\text{if}\ \eps<0.$$
	
	Now fix any $M_1>0$ and choose $M_2,...,M_n$ inductively. Suppose, we have chosen $M_1,...,M_k$, $1\le k<n$. Consider the set $S_k$ of trivalent plane tropical curves such that
	\begin{itemize}\item they intersect the line $\{y=\eps x\}$ only at the points $u_i=(M_i,\eps M_i)$, $i=1,...,k$;
		\item the directing vector of an edge containing a point $u_i$ $1\le i\le k$, belongs to $A$.
	\end{itemize}
	Clearly, $S_k$ is finite, and there is a closed disk $D_k$ containing all the vertices of all the curves from $S_k$. Choose $M'_{k+1}>M_k$ so that any line directed by a vector from $A$ and passing through the point $(M'_{k+1},\eps M'_{k+1})$ is disjoint from $D_k$. Then set $M_{k+1}$ to be any number $\ge2M'_{k+1}$.
	
	It follows that the sequences $(M_1,...,M_n)$ constructed in the above procedure fill in a full-dimensional polyhedral cone $\Con_\eps\subset\R^n$. Observe that $\GS_y(P,g,(n_1,n_2),\bu)$ does not change when $\eps$ varies in the punctured disk $0<|\eps|<\eps_0$ and the sequence $(M_1,...,M_n)$ varies in the cone $\Con_\eps$.
	
	\smallskip{\it Step 2.}
	It is obvious that a simultaneous change of the sign of $\eps$ and replacement of $P$ by $\overline P$ does not affect the refined count, and we get $\GS_y(P,g,(n_1,n_2),\bu) = \GS_y(\overline{P}, g, (n_1, n_2), -\bu)$.
	Indeed, reflection around the $x$-axis gives a multiplicity preserving bijection between the curves counted in the lefthand side and the curves counted in the righthand side.

	We will now show that $\GS_y(P,g,(n_1,n_2),\bu) = \GS_y(\overline{P}, g, (n_1, n_2), \bu)$ by constructing a multiplicity preserving bijection between the curves counted in the lefthand side and the curves counted in the righthand side.
	Let $\Gamma$ be a tropical curve passing through $\bu$, let $E$ be a non-horizontal edge of $\Gamma$, let $V\in \Gamma^{(0)}$ be a vertex incident to $E$, and denote $\oa_V(E)=(a,b)$ where $a\in \Z$ and $b\in \{\pm 1\}$ (see Lemma \ref{lemma:trop_curve_through_mikh_pos}).
	Suppose that the non horizontal ends of the floor containing $E$ have directions $(c,1)$ and $(d, -1)$ for some $c,d\in \Z$.
	Define $u'_{V,E} = (a + (d-c)\cdot b, b)$.
	Taking $\alpha$ to be the combinatorial type of tropical curves which is identical to $[\Gamma]$ except the direction $\oa_V(E)$ is replaced by $u'_{V,E}$, the condition that $|\eps|<<1$ ensure that there exist a (necessarily unique) tropical curve $\Gamma'$ passing through $\bu$ of type $\alpha$.
	It is immediate that the refined multiplicity of $\Gamma'$ is equal to the refined multiplicity of $\Gamma$, since the Mikhalkin multiplicities of all the vertices is the same, and by construction if $\Gamma$ has degree induced by $P$ then $\Gamma'$ has degree induced by $\overline{P}$.
	
	\smallskip{\it Step 3.}
	Now we chose $\eps$ and $\bu$ as specified in Step 1. The required invariance amounts to the following statement: if $u_i\in\bu_{(1)}$, $u_{i+1}\in\bu_{(2)}$, and $\bu'$ is obtained from $\bu$ by setting $u_i\in\bu'_{(2)}$, $u_{i+1}\in\bu'_{(1)}$, then
	$$\GS_y(P,g,(n_1,n_2),\bu)=\GS_y(P,g,(n_1,n_2),\bu').$$
	
	Equivalently, we shall prove the constancy of $\GS_y(P,g,(n_1,n_2),\bu^{(\tau)})$ along the family of configurations
	$\bu^{(\tau)}$, $M_i\le \tau\le M'_{i+1}$, where \linebreak ${M_{i+1}\ll M'_{i+1}\ll M_{i+2}}$, and $\bu^{(\tau)}\setminus\{u_i^{(\tau)}\}=
	\bu\setminus\{u_i\}$, while $u_i^{(\tau)}=(\tau,\eps\tau)$.
	
	(a) First, we analyze the behavior of $\GS_y(P,g,(n_1,n_2),\bu^{(\tau)})$ when
	\begin{equation}\tau\in(M_i,M_{i+1})\quad\text{or}\quad\tau\in(M_{i+1},M'_{i+1}).\label{ecn34}\end{equation}
	By the balancing condition, for every $\tau$ in either of the intervals (\ref{ecn34}), each curve $(\Gamma^{(\tau)},h^{(\tau)}, \bg^{(\tau)}, \bq^{(\tau)})\in{\mathcal M}^{trop}_{g,(n_1,n_2)}(P,\bu^{(\tau)})$, and every vertex $V\in \Gamma^{(\tau)}$, there exist an edge adjacent to $V$ that points away from the line $(1,\eps)\cdot \R$.
	Thus $\Gamma^{(\tau)}\setminus \bq^{(\tau)}$ does not contain bounded component and so, from usual Euler characteristic computation, it is the union of trees such that any tree contains exactly one unbounded edge.
	Suppose that $q_i$ and $q_{i+1}$ do not belong to the closure of the same connected component of $\Gamma\setminus\bq$. Since the configuration $\bu^{(\tau)}\setminus\{u_{i+1}\}$ is in Mikhalkin position for all values of $\tau\in[M_i,M'_{i+1}]$, under the assumption made, the combinatorial type of $(\Gamma^{(\tau)},\bq^{(\tau)})$ stays unchanged in each of the intervals (\ref{ecn34}), which yields the constancy of $\GS_y(\Gamma^{(\tau)},h^{(\tau)}, \bg^{(\tau)}, \bq^{(\tau)})$, and hence the constancy of $\GS_y(P,g,(n_1,n_2),\bu^{(\tau)})$. In case $q_i,q_{i+1}$ belong to the closure of one component $K_1$ or a pair of components $K_1,K_2$ of $\Gamma\setminus\bq$, the curves $(\Gamma^{(\tau)},h^{(\tau)}, \bg^{(\tau)}, \bq^{(\tau)})\in{\mathcal M}^{trop}_{g,(n_1,n_2)}(P,\bu^{(\tau)})$ may undergo degenerations of codimension one. In the considered situation, the only possible degeneration is the appearance of a four-valent vertex in $K_1$ or $K_2$; the constancy of the refined invariant in the passage through such a degeneration is well-known (see, for example, \cite[Section 4.2]{SS}).
	
	(b) Now, consider the passage of $\tau$ through the critical value $\tau_0=M_{i+1}$. Let $\qquad$ $(\Gamma^{(\tau_0)},h^{(\tau_0)}, \bg^{(\tau_0)}, \bq^{(\tau_0)})\in{\mathcal M}^{trop}_{g,(n_1,n_2)}(P,\bu^{(\tau_0)})$.
	
	(b1) Suppose that $q^{(\tau_0)}_i=q^{(\tau_0)}_{i+1}$. If $q^{(\tau_0)}_i=q^{(\tau_0)}_{i+1}$ is a vertex of $\Gamma^{(\tau_0)}$, then the limit curve associated with this vertex (in the sense of Section \ref{sec-tl}) matches two generic points in the torus $(\C^*)^2$ and a generic contact element at one of these points. It follows that the considered vertex is either a rational four-valent vertex, or an elliptic trivalent vertex.
	This means that the passage through $\tau_0$ can be locally modelled by rational and elliptic curves, and then the required invariance follows from Theorem \ref{thm:refined_invariant_one_marked}. If $q^{(\tau_0)}_i=q^{(\tau_0)}_{i+1}$ is an interior point of a horizontal edge of $\Gamma^{(\tau_0)}$ (see Figure \ref{fcn6}(a)), then for the curve
	$(\widehat\Gamma^{(\tau_0)},\widehat h^{(\tau_0)}, \widehat\bg^{(\tau_0)}, \widehat\bq^{(\tau_0)})$, the points $\widehat q^{(\tau_0)}_i$ and $\widehat q^{(\tau_0)}_{i+1}$ differ from each other, and the corresponding fragment of $\widehat\Gamma^{(\tau_0)}$ looks as shown in Figure \ref{fcn6}(b). In this fragment, the edges incident to $u_{i+1}$ can be extended to the right in the ways shown in Figures \ref{fcn2}(e,f,g) or \ref{fig:high_genus_corr_theorem}(a); respectively, the edge incident to $u^{(\tau_0)}_i$ can be unbounded or finite. We consider here the most interesting extended fragment as shown in Figure \ref{fcn6}(c). It is the central element in the bifurcation exhibited in Figure \ref{fcn6}, where the weights of the edges incident to the marked points stay unchanged, and hence $\GS_y(\Gamma^{(\tau)},h^{(\tau)}, \bg^{(\tau)}, \bq^{(\tau)})$ remains constant as $\tau$ varies in a neighborhood of $\tau_0$. The same phenomenon occurs for other extensions of the fragment in Figure \ref{fcn6}(b).
	
	(b2) Suppose that $q^{(\tau_0)}_i\ne q^{(\tau_0)}_{i+1}$. Then the combinatorial type of \linebreak $(\Gamma^{(\tau)},h^{(\tau)}, \bg^{(\tau)}, \bq^{(\tau)})$ remains unchanged when $\tau$ varies in a neighborhood of $\tau_0$ (cf. item (a)), and hence $\GS_y(\Gamma^{(\tau)},h^{(\tau)}, \bg^{(\tau)}, \bq^{(\tau)})$ is constant in this variation.
	
	The proof is completed.
\end{proof}

\begin{remark}
	There are two other approaches that can be used to provide an alternative proof of Theorem \ref{tcn4}.
	The first is to compute $GS_y(P,g,(n_1,n_2),\bu)$ by means of floor diagrams and to prove this computation is invariant under the changes of the parameters as in the Theorem \ref{tcn4}, this approach was implemented in \cite{Mevel}.
	Another alternative is to note that locally the degenerations of codimension one that occur in the proof of Theorem \ref{tcn4} are the same as the ones that occur in the proof of Theorem \ref{thm:refined_invariant_one_marked}, so in fact one can deduce Theorem \ref{tcn4} from Theorem \ref{thm:refined_invariant_one_marked}.
	This latter approach requires too much auxiliary notation to be practical, but it is conceptually similar to the proof given above.
\end{remark}

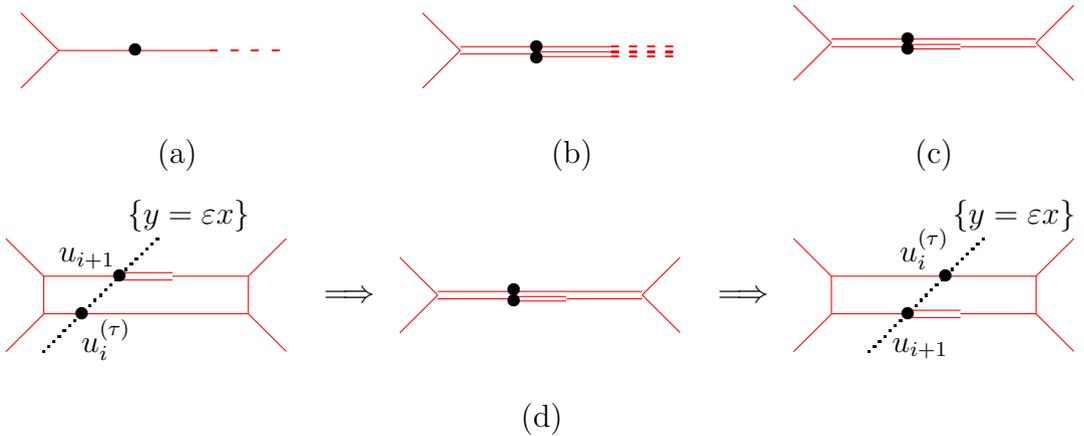
\begin{figure}
	\setlength{\unitlength}{1mm}
	\begin{picture}(140,55)(0,0)
		\thicklines
		%
		\dottedline{1}(5,10)(20,25)\dottedline{1}(114,10)(129,25)
		
		\thinlines
		{\color{red}
			\put(0,10){\line(1,1){5}}\put(5,15){\line(1,0){27}}\put(5,15){\line(0,1){5}}
			\put(5,20){\line(-1,1){5}}\put(5,20){\line(1,0){10}}\put(15,19.5){\line(1,0){7}}\put(15,20.5){\line(1,0){7}}
			\put(22,20){\line(1,0){10}}\put(32,15){\line(1,-1){5}}\put(32,15){\line(0,1){5}}\put(32,20){\line(1,1){5}}
			\put(52,12.5){\line(1,1){5}}\put(52,22.5){\line(1,-1){5}}\put(57,17){\line(1,0){10}}\put(57,18){\line(1,0){27}}
			\put(67,16.7){\line(1,0){7}}\put(67,17.3){\line(1,0){7}}\put(74,17){\line(1,0){10}}
			\put(84,17.5){\line(1,1){5}}\put(84,17.5){\line(1,-1){5}}
			\put(104,10){\line(1,1){5}}\put(109,15){\line(0,1){5}}\put(109,15){\line(1,0){10}}
			\put(109,20){\line(1,0){27}}\put(119,14.5){\line(1,0){7}}\put(119,15.5){\line(1,0){7}}
			\put(126,15){\line(1,0){10}}\put(136,15){\line(1,-1){5}}\put(136,15){\line(0,1){5}}
			\put(136,20){\line(1,1){5}}\put(109,20){\line(-1,1){5}}
			
			\put(2,45){\line(1,1){5}}\put(7,50){\line(1,0){20}}\put(7,50){\line(-1,1){5}}\dashline{1}(27,50)(36,50)
			\put(55,45){\line(1,1){5}}\put(55,55){\line(1,-1){5}}\put(60,49.5){\line(1,0){10}}\put(60,50.5){\line(1,0){20}}
			\put(70,49.2){\line(1,0){10}}\put(70,49.8){\line(1,0){10}}
			\dashline{1}(80,49.2)(88,49.2)\dashline{1}(80,49.8)(88,49.8)\dashline{1}(80,50.5)(88,50.5)
			\put(104,46){\line(1,1){5}}\put(104,56){\line(1,-1){5}}\put(109,50.5){\line(1,0){10}}\put(109,51.5){\line(1,0){27}}
			\put(119,50.2){\line(1,0){7}}\put(119,50.8){\line(1,0){7}}\put(126,50.5){\line(1,0){10}}
			\put(136,51){\line(1,1){5}}\put(136,51){\line(1,-1){5}}
		}
		
		\put(9,14){$\bullet$}\put(14,19){$\bullet$}\put(66,15.8){$\bullet$}\put(66,17.2){$\bullet$}\put(118,14){$\bullet$}
		\put(123,19){$\bullet$}\put(16,49){$\bullet$}\put(69,48.1){$\bullet$}\put(69,49.5){$\bullet$}
		\put(118,49.2){$\bullet$}\put(118,50.5){$\bullet$}
		\put(20,35){(a)}\put(72,35){(b)}\put(120,35){(c)}\put(68,0){(d)}
		\put(42,16.5){$\Longrightarrow$}\put(94,16.5){$\Longrightarrow$}
		\put(10,10){$u^{(\tau)}_i$}\put(7,22){$u_{i+1}$}\put(117,10){$u_{i+1}$}\put(118,22){$u^{(\tau)}_i$}
		\put(16,27){$\{y=\eps x\}$}\put(125,27){$\{y=\eps x\}$}
		
	\end{picture}
	\caption{Proof of Theorem \ref{tcn4}}\label{fcn6}
\end{figure}

\section*{Declarations}

\begin{itemize}
\item Funding: The authors were supported by the ISF grant no. 501/18, the BSF grant no.  2022157, and the Bauer-Neuman Chair in Real and Complex Geometry.
The second author was also supported by ERC Starting Grant 757585, ISF grant no. 1918/23, and by the Milner Foundation.
\item Conflict of interest The authors declare no conflict of interest.
\end{itemize}

\appendix

\section{Non existence of local refined count in higher genus}\label{sec:refined_non_local}

We say that a refined count is local if the weight of a curve is a product of weights corresponding to the following elements:
\begin{enumerate}
	\item unmarked vertices,
	\item marked vertices,
	\item centrally embedded collinear cycles,
	\item collinear cycles with one 3-valent and one 4-valent vertex,
	\item and pairs of collinear cycles meeting in a 5-valent vertex.
\end{enumerate}
We will call a local refined count \emph{GS-type} if it extends G\"{o}ttsche-Schroeter invariants \cite{GS}, i.e. if the weight of an unmarked vertex is $ [\mu]_y^{-} $ and the weight of a marked vertex is $ [\mu]_y^{+} $, where $ \mu $ is the Mikhalkin weight of the vertex, i.e. the lattice area of the dual triangle.
In this section we will prove that there is no invariant local GS-type refined count of tropical curves for non-simple curves (see Definition \ref{def:separated_curve}).

Denote by $ \varphi^{(0)}_y(w_1, w_2) $ the weight of a centrally embedded collinear cycle with edges of weights $ w_1,w_2 $, by $ \varphi^{(1)}_y(w_1, w_2, \nu) $ the weight of a collinear cycle with edges of weights $ w_1, w_2 $ and adjacent to a 4-valent vertex with Mikhalkin weight $ (w_1+w_2) \nu $.
Finally, denote by $ \varphi^{(2)}_y(w_1,w_2,w_1',w_2',\nu) $ the weight corresponding to 2 collinear cycles, one with edges of weights $w_1$ and $w_2$ and the other with edges of weights $w_1'$ and $w_2'$, meeting in a 5-valent vertex with Mikhalkin weight $ (w_1+w_2)(w_1'+w_2')  \nu $.
Note that we do not assume $\varphi^{(0)}$ and $\varphi^{(1)}$ are the functions defined in \eqref{eq:def_varphi_0}, \eqref{eq:def_varphi_1}, but rather we are going to prove they have to be equal to those functions to define an invariant count.

\begin{lemma}\label{lemma:comp_of_phi_1_1}
	For a local refined weight as above, $ \varphi^{(1)}_y(w_1,w_2,1)=0 $.
\end{lemma}

\begin{proof}
	Consider the degeneration depicted on the left of Figure \ref{fig:degeneration_phi_1_1} where the unique vertex in the degeneration have genus $1$ and its Mikhalkin multiplicity is equal to the weight of the vertical end.
	Then the regeneration in the right of Figure \ref{fig:degeneration_phi_1_1} is the only one so its weight has to be 0.
\end{proof}

\begin{figure}
	\begin{tikzpicture}
		\tikzset{
			every path/.style = {line width=0.7pt, red},
			every node/.style = {black}
		}
		\begin{scope}[scale=3, decoration={markings,mark=at position 0.5 with {\arrow{>}}}]
			\draw [] (0.000000, 0.500000) -- node [above, very near start] {\footnotesize $v_1$} (0.500000, 0.500000);
			\draw [postaction={decorate}] (0.500000, 0.500000) -- node [right, very near end] {\footnotesize $v_2$} (0.666667, 1.000000);
			\draw [postaction={decorate}] (0.500000, 0.500000) -- node [right, very near end] {\footnotesize $v_3$} (0.500000, 0.000000);
			\draw [black, fill=black] (0.500000, 0.500000) circle (0.015);
			
		\end{scope}
		\begin{scope}[scale=3, xshift=40, decoration={markings,mark=at position 0.5 with {\arrow{>}}}]
			\draw [postaction={decorate}] (0.000000, 0.750000) -- node [above, very near start] {\footnotesize $v_1$} (0.250000, 0.750000);
			\draw [postaction={decorate}] (0.250000, 0.750000) -- node [right, very near end] {\footnotesize $v_2$} (0.333333, 1.000000);
			\draw (0.26, 0.750000) -- (0.26, 0.250000);
			\draw (0.24, 0.750000) -- (0.24, 0.250000);
			\draw [postaction={decorate}] (0.250000, 0.250000) -- node [right, very near end] {\footnotesize $v_3$} (0.250000, 0.000000);
			\draw [black, fill=black] (0.250000, 0.250000) circle (0.015);
			
		\end{scope}
	\end{tikzpicture}
	\caption{The degeneration (on the left) and regeneration (on the right) in the proof of Lemma \ref{lemma:comp_of_phi_1_1}.}\label{fig:degeneration_phi_1_1}
\end{figure}
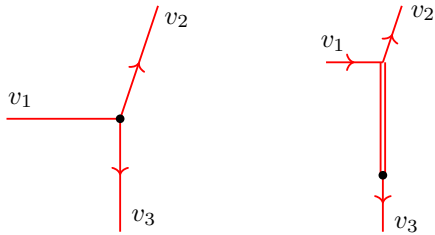

\begin{lemma}
	The count introduced in Definition \ref{def:refined_broccoli_weight} is the only refined count of simple tropical curves without $5$-valent vertices that is local in the sense described in the beginning of this appendix and that extends the count of G\"{o}ttsche and Schroeter.
\end{lemma}

\begin{proof}
	Consider a tropical curve as in case $(5)$ of Lemma \ref{lemma:codim1_degenerations_list}, where neither of the endpoints of the codimension $1$ fragment are adjacent to a collinear cycle which is not part of the codimension $1$ fragment, and only one edge of the bounded component of $\Gamma\setminus\bq$ is adjacent to the codimension $1$ fragment, see Figure \ref{fig:degen_collinear_1marked_1bounded_no_collinear}(a).
	
	This case was considered in Section \ref{sec:type4_inv_case1}.
	It follows from \eqref{eq:type4_case1_inv} and from Lemma \ref{lemma:algebra_of_quantum_weights} that in order for the count to be invariant it has to satisfy
	\begin{equation*}
		 \varphi^{(1)}_y(w_1, w_2,\mu_3) + [(w_1+w_2)\mu_3]^{-}_y \varphi^{(0)}_y(w_1, w_2)  = [w_1 \mu_3]^{-}_y [w_2\mu_3]^{-}_y.
	\end{equation*}
	Plugging $\mu_3 = 1$ and using Lemma \ref{lemma:comp_of_phi_1_1} gives us \eqref{eq:def_varphi_0}, which readily implies \eqref{eq:def_varphi_1}.
\end{proof}

\begin{proposition}\label{prop:no_local_refined_high_genus}
	Let $ g\ge 2 $ and $n_2\ge 2$. Then there is no local (in the sense of the beginning of the current section) invariant refined count of tropical curves of genus $ g $ with $n_2$ marked vertices that extends G\"{o}ttsche-Schroeter invariants as defined in \cite{GS}.
\end{proposition}

\begin{proof}
	Consider a tropical curve as in case (4) of Lemma \ref{lemma:codim1_degenerations_list}, where the unmarked end of the codimension $1$ fragment is attached to an additional collinear cycle and only one edge of the bounded component of $\Gamma\setminus\bq$ is adjacent to the codimension $1$ fragment, see Figure \ref{fig:non_locality_proof}(a).
	
	It has $6$ regenerations depicted in Figure \ref{fig:non_locality_proof}(b-g).
	In the notation of Lemma \ref{lemma:local_refined_invariance}, the regenerations depicted in Figure \ref{fig:non_locality_proof}(e,f) lie in $\mathcal{A}_{+}$, while the regenerations depicted in Figure \ref{fig:non_locality_proof}(b,c,d,g) lie in $\mathcal{A}_{-}$.
	Thus, for the count to be invariant it has to satisfy
	\begin{align*}
		\left(\varphi_y^{(2)}(w_1, w_2, w_1', w_2', \nu_1) \right. + & \left. \varphi^{(1)}_y(w_1', w'_2, \nu(w_1+w_2)) \varphi_y^{(0)}(w_1, w_2) \right)\times \\
		\times \left[(w_1+w_2)\nu_2\right]_y^{-} = &[w_2 w_2' \nu_1]^{-}_y[w_1 w_1' \nu_1]^{-}_y[(w_1w_2'-w_2w_1')\nu_1]^{-}_y[w_2\nu_2]^{-}_y[w_1\nu_2]^{+}_y+\\
		+ & [w_2 w_1' \nu_1]^{-}_y[w_1 w_2' \nu_1]^{-}_y[(w_1w_1'-w_2w_2')\nu_1]^{-}_y[w_2\nu_2]^{-}_y[w_1\nu_2]^{+}_y +\\
		+ & \varphi^{(1)}_y(w_1', w_2', \nu_1w_2)[w_1(w_1'+w_2')\nu_1]^{-}_y[w_2\nu_2]^{-}_y[w_1\nu_2]^{+}_y + \\
		+ & \varphi^{(1)}_y(w_1', w_2', \nu_1w_1)[w_2(w_1'+w_2')\nu_1]^{-}_y[w_1\nu_2]^{-}_y[w_2\nu_2]^{+}_y,
	\end{align*}
	But after dividing both sides by $[(w_1+w_2)\nu_2]^{-}_y$ we get that the left-hand side does not depend on $\nu_2$, while the right-hand side does depend on it.
\end{proof}

\begin{figure}
	\begin{center}
	\begin{scaletikzpicturetowidth}{\textwidth}

	\begin{tikzpicture}
		\tikzset{
			every path/.style = {line width=0.7pt, red},
			every node/.style = {black}
		}
		\begin{scope}[scale=\tikzscale, xshift=0, yshift=-0, decoration={markings,mark=at position 0.5 with {\arrow{>}}}]
			\draw [] (0.736273, 0.727273) -- node [right=-0.1, scale=0.8] {\tiny $w_2$} (0.736273, 0.272727);
			\draw [] (0.718273, 0.727273) -- node [left=-0.1, scale=0.8] {\tiny $w_1$} (0.718273, 0.281727);
			\draw [postaction={decorate}] (0.718273, 0.727273) -- node [above, very near end] {\footnotesize $v_2$} (0.909091, 1.000000);
			\draw [black, fill=black] (0.722773, 0.727273) circle (0.015000);
			\draw [] (0.181818, 1.000000) -- node [above, very near start] {\footnotesize $v_1$} (0.718273, 0.727273);
			\draw [postaction={decorate}] (0.718273, 0.281727) -- node [right, very near end] {\footnotesize $v_4$} (0.931818, -0.000000);
			\draw [] (0.727273, 0.263727) -- node [above=-0.07, scale=0.8] {\tiny $w_1'$} (0.272727, 0.263727);
			\draw [] (0.718273, 0.281727) -- node [below=-0.07, scale=0.8] {\tiny $w_2'$} (0.272727, 0.281727);
			\draw [postaction={decorate}] (0.272727, 0.272727) -- node [above, very near end] {\footnotesize $v_3$} (0.000000, 0.272727);
			\draw [black, fill=black] (0.272727, 0.272727) circle (0.015000);
			
		\end{scope}
		\begin{scope}[scale=\tikzscale, xshift=10, yshift=-0]
			\draw (0,0) node {$ (a) $};
		\end{scope}
	
		\begin{scope}[scale=\tikzscale, xshift=40, yshift=-0, decoration={markings,mark=at position 0.5 with {\arrow{>}}}]
			\draw [] (0.461044, 0.620766) -- (0.727247, 0.620766);
			\draw [] (0.461044, 0.620766) -- node [right=-0.1, scale=0.8] {\tiny $w_2$} (0.461044, 0.461405);
			\draw [postaction={decorate}] (0.000000, 0.851288) -- node [above, very near start] {\footnotesize $v_1$} (0.461044, 0.620766);
			\draw [] (0.727247, 0.620766) -- node [right=-0.1, scale=0.8] {\tiny $w_1$} (0.727247, 0.443405);
			\draw [postaction={decorate}] (0.727247, 0.620766) -- node [above, very near end] {\footnotesize $v_2$} (0.922369, 0.913449);
			\draw [black, fill=black] (0.727247, 0.620766) circle (0.015000);
			\draw [] (0.461044, 0.461405) -- (0.780488, 0.292683);
			\draw [] (0.461044, 0.461405) -- node [above=-0.07, scale=0.8] {\tiny $w_2'$} (0.292683, 0.461405);
			\draw [] (0.780488, 0.292683) -- (0.727247, 0.443405);
			\draw [postaction={decorate}] (0.780488, 0.292683) -- node [right, very near end] {\footnotesize $v_4$} (1.000000, -0.000000);
			\draw [] (0.727247, 0.443405) -- node [below=-0.07, pos=0.7, scale=0.8] {\tiny $w_1'$} (0.292683, 0.443405);
			\draw [postaction={decorate}] (0.292683, 0.452405) -- node [above, very near end] {\footnotesize $v_3$} (0.000000, 0.452405);
			\draw [black, fill=black] (0.292683, 0.452405) circle (0.015000);
			
		\end{scope}
		\begin{scope}[scale=\tikzscale, xshift=50, yshift=-0]
			\draw (0,0) node {$ (b) $};
		\end{scope}
	
		\begin{scope}[scale=\tikzscale, xshift=80, yshift=-0, decoration={markings,mark=at position 0.5 with {\arrow{>}}}]
			\draw [] (0.386364, 0.727273) -- (0.500000, 0.727273);
			\draw [] (0.386364, 0.727273) -- node [right=-0.12, scale=0.8] {\tiny $w_2$} (0.386364, 0.622636);
			\draw [postaction={decorate}] (0.000000, 0.920455) -- node [above, very near start] {\footnotesize $v_1$} (0.386364, 0.727273);
			\draw [] (0.500000, 0.727273) -- node [right=-0.1, scale=0.8] {\tiny $w_1$} (0.500000, 0.604636);
			\draw [postaction={decorate}] (0.500000, 0.727273) -- node [above, very near end] {\footnotesize $v_2$} (0.681818, 1.000000);
			\draw [black, fill=black] (0.500000, 0.727273) circle (0.015000);
			\draw [] (0.386364, 0.622636) -- (0.727273, 0.272727);
			\draw [] (0.386364, 0.622636) -- node [above=-0.07, scale=0.8] {\tiny $w_1'$} (0.272727, 0.622636);
			\draw [] (0.727273, 0.272727) -- (0.500000, 0.604636);
			\draw [postaction={decorate}] (0.727273, 0.272727) -- node [right, very near end] {\footnotesize $v_4$} (0.931818, -0.000000);
			\draw [] (0.500000, 0.604636) -- node [below=-0.07, scale=0.8, pos=0.7] {\tiny $w_2'$} (0.272727, 0.604636);
			\draw [postaction={decorate}] (0.272727, 0.613636) -- node [above, very near end] {\footnotesize $v_3$} (0.000000, 0.613636);
			\draw [black, fill=black] (0.272727, 0.613636) circle (0.015000);
			
		\end{scope}
		\begin{scope}[scale=\tikzscale, xshift=90, yshift=-0]
			\draw (0,0) node {$ (c) $};
		\end{scope}
		\begin{scope}[scale=\tikzscale, xshift=0, yshift=-40, decoration={markings,mark=at position 0.5 with {\arrow{>}}}]
			\draw [] (0.536585, 0.617886) -- (0.780488, 0.617886);
			\draw [] (0.536585, 0.617886) -- node [right=-0.1, scale=0.8] {\tiny $w_2$} (0.536585, 0.364984);
			\draw [postaction={decorate}] (0.000000, 0.886179) -- node [above, very near start] {\footnotesize $v_1$} (0.536585, 0.617886);
			\draw [] (0.780488, 0.617886) -- node [right=-0.1, scale=0.8] {\tiny $w_1$} (0.780488, 0.292683);
			\draw [postaction={decorate}] (0.780488, 0.617886) -- node [above, very near end] {\footnotesize $v_2$} (0.975610, 0.910569);
			\draw [black, fill=black] (0.780488, 0.617886) circle (0.015000);
			\draw [] (0.536585, 0.364984) -- (0.780488, 0.292683);
			\draw [] (0.536585, 0.382984) -- node [above=-0.1, scale=0.8] {\tiny $w_1'$} (0.292683, 0.382984);
			\draw [] (0.536585, 0.364984) -- node [below=-0.1, scale=0.8] {\tiny $w_2'$} (0.292683, 0.364984);
			\draw [postaction={decorate}] (0.780488, 0.292683) -- node [right, very near end] {\footnotesize $v_4$} (1.000000, 0.000000);
			\draw [postaction={decorate}] (0.292683, 0.373984) -- node [above, very near end] {\footnotesize $v_3$} (-0.000000, 0.373984);
			\draw [black, fill=black] (0.292683, 0.373984) circle (0.015000);
			
		\end{scope}
		\begin{scope}[scale=\tikzscale, xshift=10, yshift=-40]
			\draw (0,0) node {$ (d) $};
		\end{scope}
	
		\begin{scope}[scale=\tikzscale, xshift=40, yshift=-40, decoration={markings,mark=at position 0.5 with {\arrow{>}}}]
			\draw [] (0.424242, 0.575758) -- (0.424242, 0.727273);
			\draw [] (0.415242, 0.575758) -- node [left=-0.1, scale=0.8] {\tiny $w_1$} (0.415242, 0.424242);
			\draw [] (0.433242, 0.575758) -- node [right=-0.1, scale=0.8] {\tiny $w_2$} (0.433242, 0.424242);
			\draw [black, fill=black] (0.424242, 0.575758) circle (0.015000);
			\draw [postaction={decorate}] (0.424242, 0.727273) -- node [above, very near end] {\footnotesize $v_2$} (0.606061, 1.000000);
			\draw [postaction={decorate}] (0.000000, 0.939394) -- node [above, very near start] {\footnotesize $v_1$} (0.424242, 0.727273);
			\draw [] (0.424242, 0.424242) -- (0.424242, 0.263727);
			\draw [postaction={decorate}] (0.424242, 0.263727) -- node [right, very near end] {\footnotesize $v_4$} (0.628788, 0.000000);
			\draw [] (0.424242, 0.281727) -- node [above=-0.07, scale=0.8] {\tiny $w_1'$} (0.272727, 0.281727);
			\draw [] (0.424242, 0.263727) -- node [below=-0.07, scale=0.8] {\tiny $w_2'$} (0.272727, 0.263727);
			\draw [postaction={decorate}] (0.272727, 0.272727) -- node [above, very near end] {\footnotesize $v_3$} (0.000000, 0.272727);
			\draw [black, fill=black] (0.272727, 0.272727) circle (0.015000);
			
		\end{scope}
	
		\begin{scope}[scale=\tikzscale, xshift=50, yshift=-40]
			\draw (0,0) node {$ (e) $};
		\end{scope}
	
		\begin{scope}[scale=\tikzscale, xshift=80, yshift=-40, decoration={markings,mark=at position 0.5 with {\arrow{>}}}]
			\draw [] (0.500000, 0.500000) -- (0.500000, 0.727273);
			\draw [] (0.491000, 0.500000) -- node [left=-0.1, scale=0.8, pos=0.3] {\tiny $w_1$} (0.491000, 0.272727);
			\draw [] (0.509000, 0.500000) -- node [right=-0.1, scale=0.8, pos=0.3] {\tiny $w_2$} (0.509000, 0.281727);
			\draw [black, fill=black] (0.500000, 0.500000) circle (0.015000);
			\draw [postaction={decorate}] (0.500000, 0.727273) -- node [above, very near end] {\footnotesize $v_2$} (0.681818, 1.000000);
			\draw [postaction={decorate}] (0.000000, 0.977273) -- node [above, very near start] {\footnotesize $v_1$} (0.500000, 0.727273);
			\draw [postaction={decorate}] (0.509000, 0.281727) -- node [right, very near end] {\footnotesize $v_4$} (0.704545, 0.000000);
			\draw [] (0.500000, 0.263727) -- node [above=-0.07, scale=0.8] {\tiny $w_1'$} (0.272727, 0.263727);
			\draw [] (0.509000, 0.281727) -- node [below=-0.07, scale=0.8] {\tiny $w_2'$} (0.272727, 0.281727);
			\draw [postaction={decorate}] (0.272727, 0.272727) -- node [above, very near end] {\footnotesize $v_3$} (0.000000, 0.272727);
			\draw [black, fill=black] (0.272727, 0.272727) circle (0.015000);
			
		\end{scope}
	
		\begin{scope}[scale=\tikzscale, xshift=90, yshift=-40]
			\draw (0,0) node {$ (f) $};
		\end{scope}
	
		\begin{scope}[scale=\tikzscale, xshift=0, yshift=-80, decoration={markings,mark=at position 0.5 with {\arrow{>}}}]
			\draw [] (0.594139, 0.727273) -- (0.461006, 0.594139);
			\draw [] (0.594139, 0.727273) -- node [right=-0.1, scale=0.8] {\tiny $w_2$} (0.594139, 0.272727);
			\draw [postaction={decorate}] (0.594139, 0.727273) -- node [above, very near end] {\footnotesize $v_2$} (0.775958, 1.000000);
			\draw [black, fill=black] (0.594139, 0.727273) circle (0.015000);
			\draw [] (0.461006, 0.594139) -- node [right=-0.1, scale=0.8] {\tiny $w_1$} (0.461006, 0.396861);
			\draw [postaction={decorate}] (0.000000, 0.824643) -- node [above, very near start] {\footnotesize $v_1$} (0.461006, 0.594139);
			\draw [] (0.594139, 0.272727) -- (0.461006, 0.396861);
			\draw [postaction={decorate}] (0.594139, 0.272727) -- node [right, very near end] {\footnotesize $v_4$} (0.798685, 0.000000);
			\draw [] (0.461006, 0.414861) -- node [above=-0.07, scale=0.8] {\tiny $w_1'$} (0.272727, 0.414861);
			\draw [] (0.461006, 0.396861) -- node [below=-0.07, scale=0.8] {\tiny $w_2'$} (0.272727, 0.396861);
			\draw [postaction={decorate}] (0.272727, 0.405861) -- node [above, very near end] {\footnotesize $v_3$} (0.000000, 0.405861);
			\draw [black, fill=black] (0.272727, 0.405861) circle (0.015000);
			
		\end{scope}
		\begin{scope}[scale=\tikzscale, xshift=10, yshift=-80]
			\draw (0,0) node {$ (g) $};
		\end{scope}
	\end{tikzpicture}
	\end{scaletikzpicturetowidth}
	\end{center}
	\caption{The degeneration (a) as in the proof of Proposition \ref{prop:no_local_refined_high_genus} and all its regenerations.}\label{fig:non_locality_proof}
\end{figure}
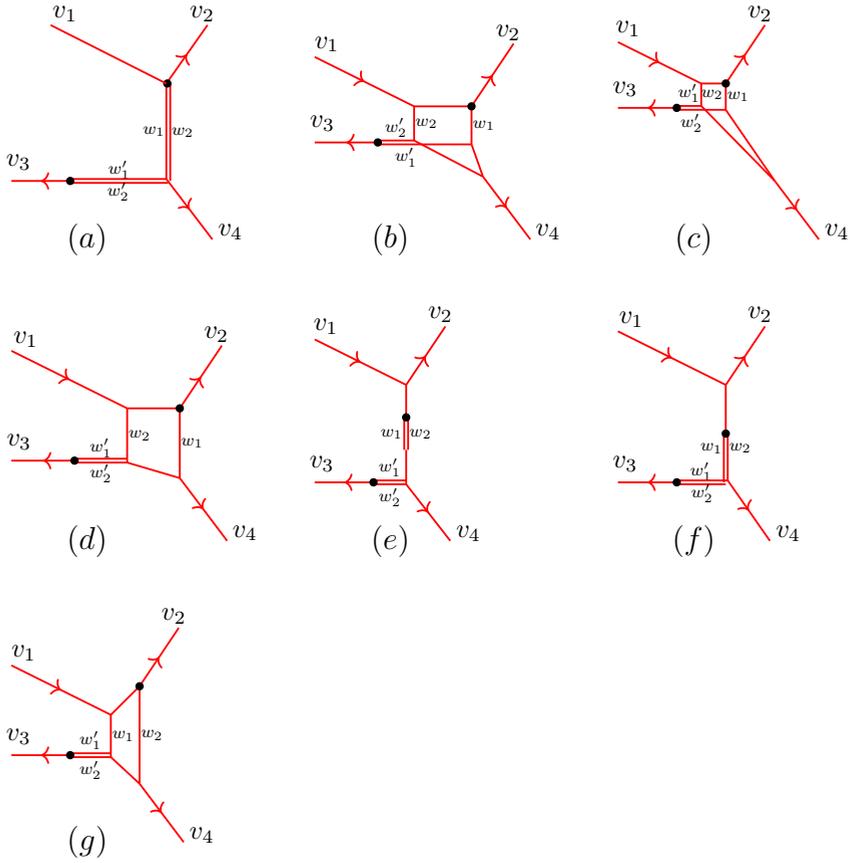

\end{document}